\documentclass[11pt]{amsart}
\usepackage{amsfonts}
\usepackage{amssymb, amstext, amscd, amsmath}
\usepackage{amsthm}
\usepackage{times}
\usepackage{mathspec}
\usepackage{indentfirst}
\usepackage[all]{xy}
\usepackage{subfigure}
\usepackage{graphicx}
\usepackage{tikz}
\usepackage{eepic}
\usepackage{pifont}
\usepackage{bbm}
\usepackage{amsfonts}
\usepackage{mathrsfs}
\usepackage{lscape}
\usepackage{setspace}
\usepackage{graphicx}
\usepackage{subfigure}
\usepackage{xcolor}
\usepackage{indentfirst}
\usepackage{url}
\usepackage{fancyhdr}
\usepackage[nocompress]{cite}
\newtheoremstyle{noparens}%
{}{}%
{\itshape}{}%
{\bfseries}{.}%
{ }%
{\thmname{#1}\thmnumber{ #2}\mdseries\thmnote{ #3}}
\theoremstyle{noparens}
\allowdisplaybreaks

\usepackage{color,xcolor} 

\usepackage{color,xcolor} 

\usepackage{color,xcolor} 

\usepackage{color,xcolor} 

\usepackage{color,xcolor}

\newtheorem{theorem}{Theorem}[section]
\newtheorem{corollary}{Corollary}[section]
\newtheorem{proposition}{Proposition}[section]
\newtheorem{lemma}{Lemma}[section]
\newtheorem{remark}{Remark}[section]
\newtheorem{definition}{Definition}[section]
\newtheorem{example}{Example}[section]
\numberwithin{equation}{section}
\numberwithin{equation}{section}

\usepackage{ulem}  

\usepackage{etoolbox} 
\usepackage{hyperref} 

\usepackage{caption}
\usepackage{hyperref} 
\usepackage{natbib}

\hypersetup{
	colorlinks=true,
	linkcolor=blue,
	citecolor=blue,
	urlcolor=cyan
}
\usepackage{cleveref} 

\crefname{figure}{Figure}{Figures}
\crefname{lemma}{Lemma}{Lemmas}
\crefname{theorem}{Theorem}{Theorems}
\crefname{definition}{Definition}{Definitions}
\crefname{proposition}{Proposition}{Propositions}
\crefname{corollary}{Corollary}{Corollaries}
\crefname{equation}{Equation}{equations}
\crefname{remark}{Remark}{remarks}
\crefname{example}{Example}{examples}
\crefname{section}{Section}{sections}
\makeatother

\begin{document}
\baselineskip 18pt

\title[The unknotting number for plus-welded knotoids]{The unknotting numbers for plus-welded knotoids}
\author[Fengling Li]{Fengling Li$^*$}\thanks{$^*$ supported  by a grant of NSFC (No. 12331003) and the Fundamental Research Funds for the Central Universities (No. DUT25LAB302)}
\address{School of Mathematical Sciences, Dalian University of Technology, Dalian 116024, P. R. China}
\email{fenglingli@dlut.edu.cn}
\author[Andrei Vesnin]{Andrei Vesnin$^\dag$}\thanks{$^\dag$ supported by the state task to the Sobolev Institute of Mathematics (No. FWNF-2022-0004)}
\address{Sobolev Institute of Mathematics of the Siberian Branch of the Russian Academy of Sciences, 630090 Novosibirsk,Russia}
\email{vesnin@math.nsc.ru}
\author[Xuan Yang]{Xuan Yang}
\address{School of Mathematical Sciences, Dalian University of Technology, Dalian 116024, P. R. China}
\email{yangxuan6210@163.com}

\begin{abstract}
Knotoid theory is a generalization of knot theory introduced by Turaev in 2012. In recent years, various invariants of knotoids have been studied. In this paper, we mainly discuss unknotting moves and unknotting numbers of plus-welded  knotoids.
Firstly, we prove that a descending diagram of a plus-welded  knotoid can be transformed into a trivial one through a finite sequence of $\Omega_1$, $V\Omega_1 - V\Omega_4$, $\Omega_v$, $\Phi_{\text{over}}$, and $\Phi_+$-moves. Secondly, we extend the warping degree of knots to plus-welded knotoids and discuss its properties. Finally, by utilizing the descending diagram and the warping degree, we obtain two unknotting operations for plus-welded knotoids, referred as a crossing change and a crossing virtualization. For both operations, we find upper bounds for corresponding unknotting numbers of plus-welded knotoids.
\end{abstract}

\date{\today}
\maketitle

\section{Introduction}

Virtual knot theory~\citep{kauffman1999virtual} and welded knot theory~\citep{rourke2007welded,Fenn1997} are generalizations of the classical knot theory in $S^3$.
One of the purposes of knot theory is to classify knots by their invariants. An unknotting operation for a classical or virtual knot is a move on a knot diagram which transforms any knot diagram to a trivial knot diagram after finite number of steps. The unknotting number of a given knot is the minimal number of steps among all diagrams of the knot. Thus, unknotting number is an invariant of a knot.
It is shown by Murakami~\citep{Murakami1985}, and then by Murakami and Nakanishi~\citep{MurakamiNakanishi1989} that the crossing change, the Delta move, and the sharp move are unknotting operations for classical knots. At the same time according to~\citep{FennTuraev2007,Kadokami2006}, the crossing change is not an unknotting operation for virtual knots. Since the Delta move and sharp move can be expressed by a sequence of crossing changes, neither of these two moves is an unknotting operation for virtual knots~\citep{SatohTaniguchi2014}.
In~\citep{satoh2018crossing}, Satoh proved that the crossing changes, Delta moves and sharp moves are unknotting operations for welded knots, and obtained an upper bound for the unknotting number of welded knots using Gauss diagrams and descending diagrams. In~\citep{NNSW2025} Nakamura, Nakanishi, Satoh and Wada defined virtualized Delta, sharp and pass moves for oriented virtual knots and links. Shimizu gave the definition  the warping degree on knot diagrams in ~\citep{shimizu2010warping} and on link diagrams in~\citep{shimizu2011warping} and studied their properties. Later Li, Lei, and Wu~\citep{li2017unknotting} generalized the definition and properties of the warping degree to welded knots, and obtained an upper bound for the unknotting number of welded knots that is smaller than the one obtained by Satoh, also they obtained a low bound inspired by the Alexander quandle coloring. 

Jeong, Park, and Park~\citep{Jeong2017Polynomials} defined homotopic virtual knot diagrams and the Gordian distance between two virtual knots derived using crossing changes.
It is well known that forbidden moves~\citep{kanenobu2001forbidden,nelson2001unknotting} and crossing virtualization~\citep{Ohyama2019Virtualization} are unknotting operations for virtual knots. Kaur et al.~\citep{KKKM2019} introduced an unknotting index for virtual links by crossing changes and crossing virtualizations, and Kaur, Prabhakar, and Vesnin~\citep{Kaur2019} introduced an unknotting index for virtual links.  
Horiuchi et al.~\citep{Horiuchi2012Gordian} extended the concept of Gordian complex to virtual knots using crossing virtualization. Later, Horiuchi and Ohyama~\citep{Horiuchi2013Gordian} defined Gordian complex of virtual knots using forbidden moves. Gill, Prabhakar and Vesnin~\citep{gill2019gordian} summarized the previously proposed Gordian complexes of classical knots defined by various moves and introduced the Gordian complex defined by arc shift move for virtual knots.
For welded knots, Li, Lei and Wu~\citep{li2017unknotting} introduced the Gordian distance between welded knots using crossing changes, Delta moves, and sharp moves. Later, Gill et al.~\citep{Gill2021UnknottingInvariant} introduced the Gordian complex of welded knots by twist move and the welded unknotting number of welded knot defined by crossing virtualization.

Knotoid theory is a generalization of knot theory introduced  by  Turaev~\citep{turaev2012knotoids} in 2012. He proposed the concept of knotoid and defined two closure maps which convert knotoids to knots. Feng, Li, and Vesnin~\citep{feng2025three} introduced the homotopy of two planar knotoids and the Gordian distance of planar knotoids defined by crossing changes.
Barbensi and Goudaroulis \citep{Barbensi2021} introduced  \textit{f}-distance of knotoids and protein structure.
 G{\"u}g{\"u}m{\"c}{\"u} and Kauffman~\citep{gugumcu2017new} gave the definition of virtual knotoids and the virtual closure map which convert virtual knotoids to virtual knots, they also gave the definition of \textit{welded virtual knotoids} which allow $\Phi_-$-move and $\Phi_{\text{over}}$-move. 
 
 In this paper, we mainly study \textit{plus-welded  knotoids}, which are virtual knotoids allowing $\Phi_+$-move and $\Phi_{\text{over}}$-move. Since crossing change is an unknotting operation for plus-welded  knotoids, we consider the unknotting number of a plus-welded  knotoid.
Firstly, in~\cref{section2}, we recall classical and virtual Reidemeister moves, forbidden knotoid moves, and virtual forbidden moves, give a definition of a plus-welded knotoid, see ~\cref{Definition 2.1}. Moreover, we generalize notions of the virtual closure and the Chord diagram for the class of plus-welded knotoids.  
In~\cref{section3}, we introduce the descending diagram for plus-welded knotoids, see ~\cref{Definition 3.1}, and prove, using Gauss diagrams, that the descending diagram of any plus-welded knotoid can be transformed into a trivial  knotoid diagram through a finite sequence of $\Omega_1$, $V\Omega_1 - V\Omega_4$, $\Omega_v$, $\Phi_{\text{over}}$, and $\Phi_+$-moves, see ~\cref{Theorem 3.1}. 
Secondly, in~\cref{section4}, we give the definition of the warping degree of plus-welded  knotoids, see ~\cref{Definition 4.1}, and explore its properties.
Thirdly, in~\cref{section5}, by utilizing the warping degree and the descending diagram of plus-welded knotoids, we obtain two upper bounds for the unknotting number of plus-welded knotoids, see ~\cref{Theorem 5.1} and ~\cref{Corollary 5.2}.
Moreover, by leveraging the virtual closure of plus-welded  knotoids, we establish a connection between the unknotting number of welded knots and that of plus-welded  knotoids, see ~\cref{Corollary 5.5}. Finally, in~\cref{section6}, we give the definition of the virtualization unknotting number for plus-welded knotoids, see~\cref{Definition 6.1}, and obtain its two upper bounds, see~\cref{Theorem 6.1} and~\cref{Corollary 6.1}.

\section{Plus-welded  knotoids} \label{section2}

Following~\citep{turaev2012knotoids}, a knotoid diagram is defined to be a generic immersion of the unit interval $[0, 1]$ flowing along a surface \(\Sigma\), with a finite number of transverse intersection points, as shown in \cref{Figure 1}.
The images of 0 and 1 are distinct from each other and from any crossing points. These two points are the endpoints of the knotoid diagram and referred to as the \textit{tail} and the \textit{head} respectively. Every knotoid diagram has a definite orientation which is from the tail to the head.
In particular, a trivial knotoid diagram is an embedding of the unit interval into a surface \(\Sigma\), as shown in \cref{Figure 1} (a). 

\vspace{-8.mm}

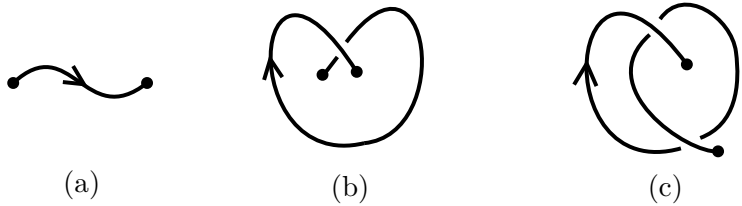
\begin{figure}[htbp]
\centering
\vspace*{-5mm}
\tikzset{every picture/.style={line width=1.5pt}}
\begin{tikzpicture}[x=0.75pt,y=0.75pt,yscale=-0.75,xscale=0.75]
    \draw    (39,73) .. controls (78,37) and (87,106) .. (129,73) ;
    \draw   (79.26,61.61) -- (86.57,73.85) -- (72.38,72.37) ;
	\draw   (207.69,70.93) -- (211.87,56.98) -- (219.97,69.08) ;
	\draw    (275,113.5) .. controls (187,134.5) and (197,-49.5) .. (270,65.5) ;
	\draw    (275,113.5) .. controls (336,107.5) and (320,-29.5) .. (263,45.5) ;
	\draw    (257,55.5) -- (247,68.5) ;
	\draw    (474,30) .. controls (490,7) and (522,29) .. (525,52) .. controls (528,75) and (525,101) .. (501,110) ;
	\draw    (488,117) .. controls (395,143) and (411,-49.5) .. (492,60.5) ;
	\draw    (513,119) .. controls (493,121) and (427,76) .. (467,40.5) ;
	\draw   (416.15,74.09) -- (424.66,60.01) -- (431.22,75.1) ;
	\filldraw[color={rgb, 255:red, 0; green, 0; blue, 0 }  ,draw opacity=1] (39,73) circle (2pt);
	\filldraw[color={rgb, 255:red, 0; green, 0; blue, 0 }  ,draw opacity=1] (129,73) circle (2pt);
	\filldraw[color={rgb, 255:red, 0; green, 0; blue, 0 }  ,draw opacity=1] (270,65.5) circle (2pt);
	\filldraw[color={rgb, 255:red, 0; green, 0; blue, 0 }  ,draw opacity=1] (247,67.5) circle (2pt);
	\filldraw[color={rgb, 255:red, 0; green, 0; blue, 0 }  ,draw opacity=1] (492,60.5) circle (2pt);
	\filldraw[color={rgb, 255:red, 0; green, 0; blue, 0 }  ,draw opacity=1] (513,119) circle (2pt);
	\draw (70,129) node [anchor=north west][inner sep=0.75pt]   [align=left] {(a)};
	\draw (250,132) node [anchor=north west][inner sep=0.75pt]   [align=left] {(b)};
	\draw (463,132) node [anchor=north west][inner sep=0.75pt]   [align=left] {(c)};
\end{tikzpicture}
 \caption{Examples of knotoid diagrams.}
 \label{Figure 1}
\end{figure}

The three classical Reidemeister moves are denoted by $\Omega_1$, $\Omega_2$, $\Omega_3$ respectively and defined on knotoid diagrams, as shown in \cref{Figure 2}. 

\smallskip 

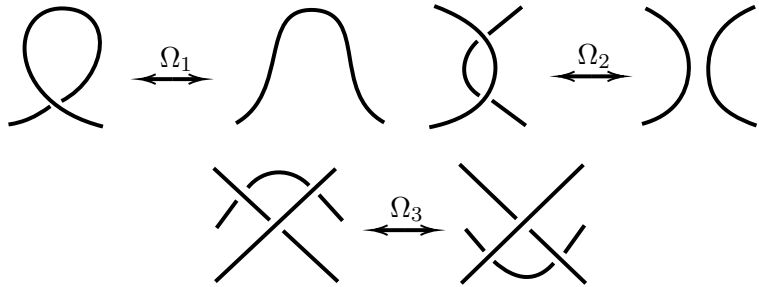
\begin{figure}[htbp]
	\centering
	\tikzset{every picture/.style={line width=1.5pt}} 
	\begin{tikzpicture}[x=0.75pt,y=0.75pt,yscale=-0.75,xscale=0.75]
	\draw    (127.54,123) .. controls (61.77,106.78) and (63.44,43.04) .. (100.23,43.5) .. controls (137.01,43.97) and (130.88,89.62) .. (99.67,105.85) ;
	\draw     (64,121.61) .. controls (74.03,120.22) and (82.39,116.51) .. (92.15,109.33) ;
	\draw    (155.06,91.08) -- (174.06,91.08) ;
	\draw [shift={(152.06,91.08)}, rotate = 360] [color={rgb, 255:red, 0; green, 0; blue, 0 }  ][line width=0.75]    (10.93,-3.29) .. controls (6.95,-1.4) and (3.31,-0.3) .. (0,0) .. controls (3.31,0.3) and (6.95,1.4) .. (10.93,3.29)   ;
	\draw    (174.06,91.08) -- (193.06,91.08) ;
	\draw [shift={(196.06,91.08)}, rotate = 180] [color={rgb, 255:red, 0; green, 0; blue, 0 }  ][line width=0.75]    (10.93,-3.29) .. controls (6.95,-1.4) and (3.31,-0.3) .. (0,0) .. controls (3.31,0.3) and (6.95,1.4) .. (10.93,3.29)   ;
	\draw    (217,120.74) .. controls (253.69,99.82) and (231.67,43.13) .. (268.97,44.53) .. controls (306.27,45.92) and (280.59,100.29) .. (316.67,120.27) ;
	\draw    (436.76,89.22) -- (455.76,89.22) ;
	\draw [shift={(433.76,89.22)}, rotate = 360] [color={rgb, 255:red, 0; green, 0; blue, 0 }  ][line width=0.75]    (10.93,-3.29) .. controls (6.95,-1.4) and (3.31,-0.3) .. (0,0) .. controls (3.31,0.3) and (6.95,1.4) .. (10.93,3.29)   ;
	\draw    (455.76,89.22) -- (474.76,89.22) ;
	\draw [shift={(477.76,89.22)}, rotate = 180] [color={rgb, 255:red, 0; green, 0; blue, 0 }  ][line width=0.75]    (10.93,-3.29) .. controls (6.95,-1.4) and (3.31,-0.3) .. (0,0) .. controls (3.31,0.3) and (6.95,1.4) .. (10.93,3.29)   ;
	\draw    (350.31,41.9) .. controls (378.8,52.24) and (392.71,67.76) .. (391.38,85.57) .. controls (390.06,103.39) and (376.15,117.75) .. (347,123.5) ;
	\draw     (379.46,65.46) .. controls (368.2,75.8) and (366.21,93.04) .. (380.78,102.24) ;
	\draw    (409.27,42.47) -- (386.75,60.86) ;
	\draw    (389.07,105.4) -- (411.92,122.35) ;
	\draw    (492.35,43.44) .. controls (512.55,53.3) and (522.41,68.1) .. (521.48,85.08) .. controls (520.54,102.06) and (510.67,115.76) .. (490,121.24) ;
	\draw    (566.06,42.35) .. controls (541.16,52.75) and (534.12,69.19) .. (534.59,85.63) .. controls (535.06,102.06) and (541.63,113.57) .. (567,122.33) ;
	\draw     (202,151.01) -- (239.7,185.98) ;
	\draw     (284.02,151.38) -- (203.17,227.34) ;
	\draw     (248.64,193.12) -- (285.57,227.34) ;
	\draw     (224.93,164.17) .. controls (235.82,150.63) and (254.48,150.63) .. (266.14,162.66) ;
	\draw   (272.8,168.88) -- (288.67,186.38) ;
	\draw   (217.55,172.81) -- (203.94,191.24) ;
	\draw    (311,192.72) -- (330,192.72) ;
	\draw [shift={(308,192.72)}, rotate = 360] [color={rgb, 255:red, 0; green, 0; blue, 0 }  ][line width=0.75]    (10.93,-3.29) .. controls (6.95,-1.4) and (3.31,-0.3) .. (0,0) .. controls (3.31,0.3) and (6.95,1.4) .. (10.93,3.29)   ;
	\draw    (330,192.72) -- (349,192.72) ;
	\draw [shift={(352,192.72)}, rotate = 180] [color={rgb, 255:red, 0; green, 0; blue, 0 }  ][line width=0.75]    (10.93,-3.29) .. controls (6.95,-1.4) and (3.31,-0.3) .. (0,0) .. controls (3.31,0.3) and (6.95,1.4) .. (10.93,3.29)   ;
	\draw   (367.33,148.34) -- (405.61,185.06) ;
	\draw   (450.6,148.73) -- (368.52,228.5) ;
	\draw   (414.69,192.57) -- (452.18,228.5) ;
	\draw   (390.33,213.64) .. controls (405.33,227.26) and (419.33,227.88) .. (430.67,214.88) ;
	\draw   (371.33,191.98) -- (384,206.83) ;
	\draw   (451.33,189.5) -- (438,208.07) ;
	\draw (163,67) node [anchor=north west][inner sep=0.75pt]    {$\Omega_{1}$};
	\draw (444,65) node [anchor=north west][inner sep=0.75pt]    {$\Omega_{2}$};
	\draw (318,168) node [anchor=north west][inner sep=0.75pt]    {$\Omega_{3}$};			
\end{tikzpicture}
\caption{Classical Reidemeister moves.}
\label{Figure 2}
\end{figure}
These moves modify the knotoid diagram within a~small disk around the local diagram region and without utilizing the endpoints.
Two knotoid diagrams are said to be \textit{equivalent} if one can be obtained from the other by a finite sequence of classical Reidemeister moves and isotopies of \(\Sigma\). The corresponding equivalence classes are called \textit{knotoids}. If \(\Sigma=S^2\), the knotoid is said to be \textit{spherical}, and if \(\Sigma= \mathbb{R}^2\), the knotoid is said to be \textit{planar}.
It is forbidden to pull the strand adjacent to an endpoint over or under a transversal strand as shown in \cref{Figure 3}. 
\begin{figure}[htbp]
	\centering
	\tikzset{every picture/.style={line width=1.5pt}} 
	\begin{tikzpicture}[x=0.75pt,y=0.75pt,yscale=-0.75,xscale=0.75]
	\draw     (50,77.11) -- (112.96,77.11) ;
	\filldraw[color={rgb, 255:red, 0; green, 0; blue, 0 }  ,draw opacity=1] (112.96,77.11) circle (2pt);
	\draw     (85.07,31.33) -- (85.07,70.72) ;
	\draw     (85.07,85.09) -- (85.07,122.88) ;
	\draw    (136.08,76.57) -- (160.08,76.57) ;
	\draw [shift={(133.08,76.57)}, rotate = 360] [color={rgb, 255:red, 0; green, 0; blue, 0 }  ][line width=0.75]    (10.93,-3.29) .. controls (6.95,-1.4) and (3.31,-0.3) .. (0,0) .. controls (3.31,0.3) and (6.95,1.4) .. (10.93,3.29)   ;
	\draw    (160.08,76.57) -- (184.08,76.57) ;
	\draw [shift={(187.08,76.57)}, rotate = 180] [color={rgb, 255:red, 0; green, 0; blue, 0 }  ][line width=0.75]    (10.93,-3.29) .. controls (6.95,-1.4) and (3.31,-0.3) .. (0,0) .. controls (3.31,0.3) and (6.95,1.4) .. (10.93,3.29)   ;
	\draw    (252.05,31.87) -- (252.05,121.81) ;
	\draw    (196.65,77.11) -- (241,77.11) ;
	\filldraw[color={rgb, 255:red, 0; green, 0; blue, 0 }  ,draw opacity=1] (241,77.11) circle (2pt);
	\draw    (286.71,76.57) -- (310.71,76.57) ;
	\draw [shift={(283.71,76.57)}, rotate = 360] [color={rgb, 255:red, 0; green, 0; blue, 0 }  ][line width=0.75]    (10.93,-3.29) .. controls (6.95,-1.4) and (3.31,-0.3) .. (0,0) .. controls (3.31,0.3) and (6.95,1.4) .. (10.93,3.29)   ;
	\draw    (310.71,76.57) -- (334.71,76.57) ;
	\draw [shift={(337.71,76.57)}, rotate = 180] [color={rgb, 255:red, 0; green, 0; blue, 0 }  ][line width=0.75]    (10.93,-3.29) .. controls (6.95,-1.4) and (3.31,-0.3) .. (0,0) .. controls (3.31,0.3) and (6.95,1.4) .. (10.93,3.29)   ;
	\draw    (385,34.53) -- (385,125.01) ;
	\draw    (349.67,77.11) -- (380.35,77.11) ;
	\draw    (391.91,77.11) -- (421,77.11) ;
	\filldraw[color={rgb, 255:red, 0; green, 0; blue, 0 }  ,draw opacity=1] (421,77.11) circle (2pt);
	\draw (300,50) node [anchor=north west][inner sep=0.75pt]   [align=left] {$\Phi_-$};
	\draw (150,50) node [anchor=north west][inner sep=0.75pt]   [align=left] {$\Phi_+$};		
\end{tikzpicture}
\caption{Forbidden knotoid moves.}	
\label{Figure 3}	
\end{figure}
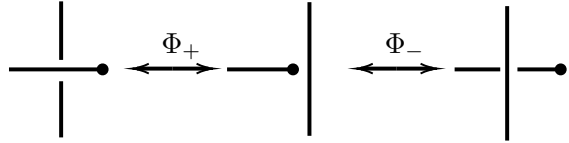
These moves are called \textit{forbidden knotoid moves}, and denoted by $\Phi_+$ and $\Phi_-$ respectively. It should be noted that if both $\Phi_+$ and $\Phi_-$-moves are allowed, then any knotoid diagram on $S^2$ (or on $\mathbb{R}^2$) can be transformed into a trivial one~\citep{gugumcu2017new}.

A \textit{virtual knot diagram} is defined to be a generic immersion that embeds a circle into a plane, where double points represent \textit{classical crossings} and \textit{virtual crossings}, are shown in \cref{Figure 4}. Among them, a virtual crossing is indicated by a circle around the crossing point of two strands.


\begin{figure}[htbp]
\centering
\tikzset{every picture/.style={line width=1.5pt}} 
	\begin{tikzpicture}[x=0.75pt,y=0.75pt,yscale=-0.7,xscale=0.7]
	\draw    (92,61) -- (191,161) ;
	\draw    (94,161) -- (131,121) ;
	\draw    (190,61) -- (151,101) ;
	\draw    (270,61) -- (370,161) ;
	\draw    (269.5,161) -- (369.5,61) ;
	\draw   (310,110) .. controls (310,104.75) and (314.25,100.5) .. (319.5,100.5) .. controls (324.75,100.5) and (329,104.75) .. (329,110) .. controls (329,115.25) and (324.75,119.5) .. (319.5,119.5) .. controls (314.25,119.5) and (310,115.25) .. (310,110) -- cycle ;			
\end{tikzpicture}
\caption{Classical crossing and virtual crossing.}
\label{Figure 4}
\end{figure}
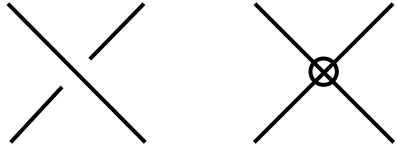

Classical Reidemeister moves ({$\Omega_{1}-\Omega_{3}$}), virtual Reidemeister moves ($V\Omega_{1}-V\Omega_{4}$) are called together \textit{generalized Reidemeister moves}, see~\cref{Figure 2} and~\cref{Figure 5}.  
\begin{figure}[htbp]
\centering
\tikzset{every picture/.style={line width=1.5pt}} 

\begin{tikzpicture}[x=0.7pt,y=0.75pt,yscale=-0.75,xscale=0.75,trim left=0pt,  
	trim right=340pt, 
	baseline   ]
	\draw    (29.16,97.85) .. controls (202.82,-3.78) and (-88.38,-16.48) .. (82.1,99) ;
	\draw   (49.28,79.37) .. controls (49.28,74.9) and (52.6,71.28) .. (56.69,71.28) .. controls (60.78,71.28) and (64.1,74.9) .. (64.1,79.37) .. controls (64.1,83.83) and (60.78,87.45) .. (56.69,87.45) .. controls (52.6,87.45) and (49.28,83.83) .. (49.28,79.37) -- cycle ;
	\draw    (193.76,95.69) .. controls (167.28,-8.24) and (283.76,-9.4) .. (258.35,95.69) ;
	\draw    (364,16) .. controls (418,13) and (430,84) .. (360,93) ;
	\draw   (389.17,28) .. controls (389.17,24.35) and (392.13,21.39) .. (395.78,21.39) .. controls (399.43,21.39) and (402.39,24.35) .. (402.39,28) .. controls (402.39,31.65) and (399.43,34.61) .. (395.78,34.61) .. controls (392.13,34.61) and (389.17,31.65) .. (389.17,28) -- cycle ;
	\draw    (430.62,92.86) .. controls (388.05,96.81) and (369.39,54.16) .. (395.78,28) .. controls (402.77,21.07) and (412.92,15.3) .. (426.62,11.86) ;
	\draw   (399.88,79.79) .. controls (400.01,83.37) and (397.22,86.38) .. (393.64,86.51) .. controls (390.06,86.64) and (387.05,83.85) .. (386.92,80.27) .. controls (386.79,76.69) and (389.58,73.68) .. (393.16,73.55) .. controls (396.74,73.42) and (399.75,76.21) .. (399.88,79.79) -- cycle ;
	\draw    (521,17) .. controls (555,21) and (561,72) .. (520,88) ;
	\draw    (595.15,87.38) .. controls (570.03,83.82) and (559.54,49.07) .. (573.64,29.46) .. controls (578.57,22.61) and (586.49,17.61) .. (597.84,16.43) ;
	\draw    (136,55) -- (160,55) ;
	\draw [shift={(163,55)}, rotate = 180] [color={rgb, 255:red, 0; green, 0; blue, 0 }  ][line width=0.75]    (10.93,-3.29) .. controls (6.95,-1.4) and (3.31,-0.3) .. (0,0) .. controls (3.31,0.3) and (6.95,1.4) .. (10.93,3.29)   ;
	\draw    (136,55) -- (115,55) ;
	\draw [shift={(112,55)}, rotate = 360] [color={rgb, 255:red, 0; green, 0; blue, 0 }  ][line width=0.75]    (10.93,-3.29) .. controls (6.95,-1.4) and (3.31,-0.3) .. (0,0) .. controls (3.31,0.3) and (6.95,1.4) .. (10.93,3.29)   ;
	\draw    (472,56) -- (496,56) ;
	\draw [shift={(499,56)}, rotate = 180] [color={rgb, 255:red, 0; green, 0; blue, 0 }  ][line width=0.75]    (10.93,-3.29) .. controls (6.95,-1.4) and (3.31,-0.3) .. (0,0) .. controls (3.31,0.3) and (6.95,1.4) .. (10.93,3.29)   ;
	\draw    (472,56) -- (451,56) ;
	\draw [shift={(448,56)}, rotate = 360] [color={rgb, 255:red, 0; green, 0; blue, 0 }  ][line width=0.75]    (10.93,-3.29) .. controls (6.95,-1.4) and (3.31,-0.3) .. (0,0) .. controls (3.31,0.3) and (6.95,1.4) .. (10.93,3.29)   ;
	\draw    (24,144) -- (101.25,225.5) ;
	\draw    (22,225) -- (104,145) ;
	\draw    (23,182) .. controls (45,139) and (82,141) .. (101,182) ;
	\draw   (54.63,184.75) .. controls (54.63,180.33) and (58.21,176.75) .. (62.63,176.75) .. controls (67.04,176.75) and (70.63,180.33) .. (70.63,184.75) .. controls (70.63,189.17) and (67.04,192.75) .. (62.63,192.75) .. controls (58.21,192.75) and (54.63,189.17) .. (54.63,184.75) -- cycle ;
	\draw   (78.63,161.75) .. controls (78.63,157.33) and (82.21,153.75) .. (86.63,153.75) .. controls (91.04,153.75) and (94.63,157.33) .. (94.63,161.75) .. controls (94.63,166.17) and (91.04,169.75) .. (86.63,169.75) .. controls (82.21,169.75) and (78.63,166.17) .. (78.63,161.75) -- cycle ;
	\draw   (31.63,160.75) .. controls (31.63,156.33) and (35.21,152.75) .. (39.63,152.75) .. controls (44.04,152.75) and (47.63,156.33) .. (47.63,160.75) .. controls (47.63,165.17) and (44.04,168.75) .. (39.63,168.75) .. controls (35.21,168.75) and (31.63,165.17) .. (31.63,160.75) -- cycle ;
	\draw    (190,142) -- (267.25,223.5) ;
	\draw    (188,223) -- (270,143) ;
	\draw    (187,188) .. controls (214,231) and (252,223) .. (274,186) ;
	\draw   (220.63,182.75) .. controls (220.63,178.33) and (224.21,174.75) .. (228.63,174.75) .. controls (233.04,174.75) and (236.63,178.33) .. (236.63,182.75) .. controls (236.63,187.17) and (233.04,190.75) .. (228.63,190.75) .. controls (224.21,190.75) and (220.63,187.17) .. (220.63,182.75) -- cycle ;
	\draw   (244.63,208.75) .. controls (244.63,204.33) and (248.21,200.75) .. (252.63,200.75) .. controls (257.04,200.75) and (260.63,204.33) .. (260.63,208.75) .. controls (260.63,213.17) and (257.04,216.75) .. (252.63,216.75) .. controls (248.21,216.75) and (244.63,213.17) .. (244.63,208.75) -- cycle ;
	\draw   (196.63,207.75) .. controls (196.63,203.33) and (200.21,199.75) .. (204.63,199.75) .. controls (209.04,199.75) and (212.63,203.33) .. (212.63,207.75) .. controls (212.63,212.17) and (209.04,215.75) .. (204.63,215.75) .. controls (200.21,215.75) and (196.63,212.17) .. (196.63,207.75) -- cycle ;
	\draw    (138,184) -- (162,184) ;
	\draw [shift={(165,184)}, rotate = 180] [color={rgb, 255:red, 0; green, 0; blue, 0 }  ][line width=0.75]    (10.93,-3.29) .. controls (6.95,-1.4) and (3.31,-0.3) .. (0,0) .. controls (3.31,0.3) and (6.95,1.4) .. (10.93,3.29)   ;
	\draw    (138,184) -- (117,184) ;
	\draw [shift={(114,184)}, rotate = 360] [color={rgb, 255:red, 0; green, 0; blue, 0 }  ][line width=0.75]    (10.93,-3.29) .. controls (6.95,-1.4) and (3.31,-0.3) .. (0,0) .. controls (3.31,0.3) and (6.95,1.4) .. (10.93,3.29)   ;
	\draw    (365.38,142.25) -- (399,179) ;
	\draw    (363,223) -- (445,143) ;
	\draw    (364,180) .. controls (386,137) and (423,139) .. (442,180) ;
	\draw   (419.63,159.75) .. controls (419.63,155.33) and (423.21,151.75) .. (427.63,151.75) .. controls (432.04,151.75) and (435.63,155.33) .. (435.63,159.75) .. controls (435.63,164.17) and (432.04,167.75) .. (427.63,167.75) .. controls (423.21,167.75) and (419.63,164.17) .. (419.63,159.75) -- cycle ;
	\draw   (372.63,158.75) .. controls (372.63,154.33) and (376.21,150.75) .. (380.63,150.75) .. controls (385.04,150.75) and (388.63,154.33) .. (388.63,158.75) .. controls (388.63,163.17) and (385.04,166.75) .. (380.63,166.75) .. controls (376.21,166.75) and (372.63,163.17) .. (372.63,158.75) -- cycle ;
	\draw    (575,186) -- (612,225) ;
	\draw    (526,225) -- (611,142) ;
	\draw    (528,186) .. controls (555,229) and (593,221) .. (615,184) ;
	\draw   (585.63,206.75) .. controls (585.63,202.33) and (589.21,198.75) .. (593.63,198.75) .. controls (598.04,198.75) and (601.63,202.33) .. (601.63,206.75) .. controls (601.63,211.17) and (598.04,214.75) .. (593.63,214.75) .. controls (589.21,214.75) and (585.63,211.17) .. (585.63,206.75) -- cycle ;
	\draw   (537.63,205.75) .. controls (537.63,201.33) and (541.21,197.75) .. (545.63,197.75) .. controls (550.04,197.75) and (553.63,201.33) .. (553.63,205.75) .. controls (553.63,210.17) and (550.04,213.75) .. (545.63,213.75) .. controls (541.21,213.75) and (537.63,210.17) .. (537.63,205.75) -- cycle ;
	\draw    (479,182) -- (503,182) ;
	\draw [shift={(506,182)}, rotate = 180] [color={rgb, 255:red, 0; green, 0; blue, 0 }  ][line width=0.75]    (10.93,-3.29) .. controls (6.95,-1.4) and (3.31,-0.3) .. (0,0) .. controls (3.31,0.3) and (6.95,1.4) .. (10.93,3.29)   ;
	\draw    (479,182) -- (458,182) ;
	\draw [shift={(455,182)}, rotate = 360] [color={rgb, 255:red, 0; green, 0; blue, 0 }  ][line width=0.75]    (10.93,-3.29) .. controls (6.95,-1.4) and (3.31,-0.3) .. (0,0) .. controls (3.31,0.3) and (6.95,1.4) .. (10.93,3.29)   ;
	\draw    (409,188) -- (442.25,223.5) ;
	\draw    (532,141) -- (565,175) ;
	\draw (120,33) node [anchor=north west][inner sep=0.75pt]   [align=left] {$V\Omega_1$};
	\draw (453,34) node [anchor=north west][inner sep=0.75pt]   [align=left] {$V\Omega_2$};
	\draw (123,162) node [anchor=north west][inner sep=0.75pt]   [align=left] {$V\Omega_3$};
	\draw (465,157) node [anchor=north west][inner sep=0.75pt]   [align=left] {$V\Omega_4$};
\end{tikzpicture}					
\caption{Virtual Reidemeister moves.}
\label{Figure 5}
\end{figure}
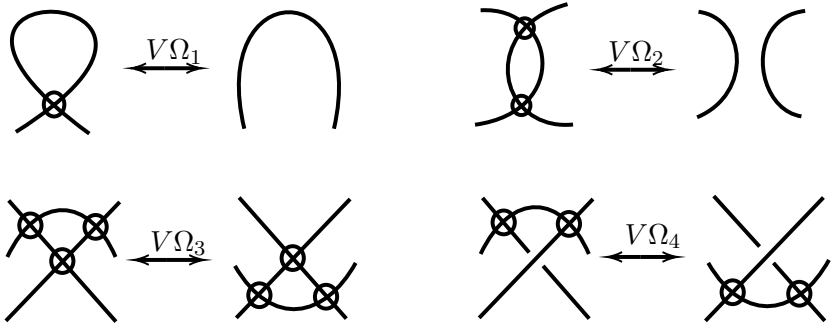 

Two virtual knot diagrams called to be \textit{equivalent} if one can be obtained from the other by a finite sequence of classical Reidemeister moves and isotopies of \(\Sigma\). The equivalence class of virtual knot diagrams is called a \textit{virtual knot}. It is forbidden that the arc with two classical crossings passes over or under the virtual crossing as shown in \cref{Figure 6}. 
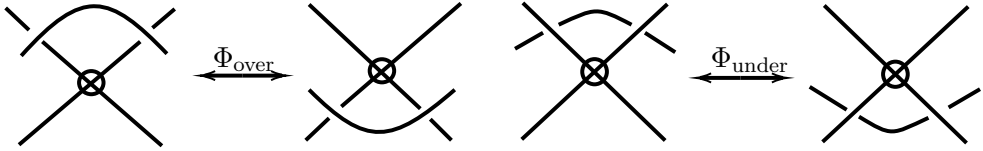
\begin{figure}[htbp]
\centering
\tikzset{every picture/.style={line width=1.5pt}} 
\begin{tikzpicture}[x=0.75pt,y=0.75pt,yscale=-0.7,xscale=0.7]	
	\draw    (55.44,38.88) -- (141.59,114) ;
	\draw    (38.62,113.24) -- (125.96,38.88) ;
	\draw   (81.89,68.47) .. controls (81.89,63.87) and (85.84,60.13) .. (90.7,60.13) .. controls (95.57,60.13) and (99.52,63.87) .. (99.52,68.47) .. controls (99.52,73.08) and (95.57,76.82) .. (90.7,76.82) .. controls (85.84,76.82) and (81.89,73.08) .. (81.89,68.47) -- cycle ;
	\draw    (247.61,12.12) -- (327.75,87.24) ;
	\draw    (270.85,85.72) -- (357,11.36) ;
	\draw   (291.29,59.92) .. controls (291.29,55.31) and (295.23,51.58) .. (300.1,51.58) .. controls (304.97,51.58) and (308.92,55.31) .. (308.92,59.92) .. controls (308.92,64.53) and (304.97,68.27) .. (300.1,68.27) .. controls (295.23,68.27) and (291.29,64.53) .. (291.29,59.92) -- cycle ;
	\draw    (199.5,63.4) -- (228.17,63.4) ;
	\draw [shift={(231.17,63.4)}, rotate = 180] [color={rgb, 255:red, 0; green, 0; blue, 0 }  ][line width=0.75]    (10.93,-3.29) .. controls (6.95,-1.4) and (3.31,-0.3) .. (0,0) .. controls (3.31,0.3) and (6.95,1.4) .. (10.93,3.29)   ;
	\draw    (199.5,63.4) -- (173,63.4) ;
	\draw [shift={(169,63.4)}, rotate = 360] [color={rgb, 255:red, 0; green, 0; blue, 0 }  ][line width=0.75]    (10.93,-3.29) .. controls (6.95,-1.4) and (3.31,-0.3) .. (0,0) .. controls (3.31,0.3) and (6.95,1.4) .. (10.93,3.29)   ;
	\draw    (40.22,49.13) .. controls (85.1,-0.95) and (103.53,4.36) .. (145.2,47.61) ;
	\draw    (45.83,30.54) -- (29,14.6) ;
	\draw    (134.78,29.02) -- (150.81,15.36) ;
	\draw    (247.61,73.58) .. controls (289.28,115.31) and (306.91,113.04) .. (352.59,73.58) ;
	\draw    (262.04,94.07) -- (245.21,110) ;
	\draw    (350.19,110) -- (334.96,94.82) ;
	\draw    (401.65,11) -- (504.16,110.17) ;
	\draw    (400.84,109.35) -- (505.78,11.83) ;
	\draw   (444.43,60.59) .. controls (444.43,55.57) and (448.41,51.5) .. (453.31,51.5) .. controls (458.21,51.5) and (462.19,55.57) .. (462.19,60.59) .. controls (462.19,65.61) and (458.21,69.68) .. (453.31,69.68) .. controls (448.41,69.68) and (444.43,65.61) .. (444.43,60.59) -- cycle ;
	\draw    (427.48,26.7) .. controls (454.12,14.31) and (454.92,13.48) .. (479.95,26.7) ;
	\draw    (618.8,12.83) -- (721.31,112) ;
	\draw    (617.59,111.17) -- (722.52,13.65) ;
	\draw   (661.18,62.41) .. controls (661.18,57.39) and (665.15,53.32) .. (670.06,53.32) .. controls (674.96,53.32) and (678.94,57.39) .. (678.94,62.41) .. controls (678.94,67.43) and (674.96,71.5) .. (670.06,71.5) .. controls (665.15,71.5) and (661.18,67.43) .. (661.18,62.41) -- cycle ;
	\draw    (643.83,92.99) .. controls (668.04,107.87) and (667.23,106.21) .. (693.87,93.82) ;
	\draw    (489.63,32.49) -- (511.43,46.54) ;
	\draw    (396,44.06) -- (416.18,32.49) ;
	\draw    (558.51,65.05) -- (589,65.05) ;
	\draw [shift={(592,65.05)}, rotate = 180] [color={rgb, 255:red, 0; green, 0; blue, 0 }  ][line width=0.75]    (10.93,-3.29) .. controls (6.95,-1.4) and (3.31,-0.3) .. (0,0) .. controls (3.31,0.3) and (6.95,1.4) .. (10.93,3.29)   ;
	\draw    (558.51,65.05) -- (529.34,65.05) ;
	\draw [shift={(526.34,65.05)}, rotate = 360] [color={rgb, 255:red, 0; green, 0; blue, 0 }  ][line width=0.75]    (10.93,-3.29) .. controls (6.95,-1.4) and (3.31,-0.3) .. (0,0) .. controls (3.31,0.3) and (6.95,1.4) .. (10.93,3.29)   ;
	\draw    (608.31,71.5) -- (634.95,88.03) ;
	\draw    (731,73.16) -- (708.4,86.38) ;
	\draw (175.33,39.98) node [anchor=north west][inner sep=0.75pt]   [align=left] {$\Phi_{\text{over}}$};
	\draw (534.2,39.98) node [anchor=north west][inner sep=0.75pt]   [align=left] {$\Phi_{\text{under}}$};
\end{tikzpicture}
\caption{Virtual forbidden moves.}
\label{Figure 6}
\end{figure} 
These moves are called \textit{virtual forbidden moves}, and denoted by $\Phi_{\text{over}}$ and $\Phi_{\text{under}}$ respectively. It should be noted that if both $\Phi_{\text{over}}$ and $\Phi_{\text{under}}$-moves are allowed, then any virtual knot diagram can be transformed into a trivial one~\citep{kanenobu2001forbidden,nelson2001unknotting}.\par

A \textit{virtual knotoid diagram} is defined to be the knotoid diagrams in \( S^2 \) with virtual crossings. The move shown in \cref{Figure 7} is called a \textit{virtual \(\Omega\)-move} and denoted by \(\Omega_v\). 
\begin{figure}[htbp]
\begin{center}
\tikzset{every picture/.style={line width=1.5pt}} 
\begin{tikzpicture}[x=0.75pt,y=0.75pt,yscale=-0.75,xscale=0.75]
	\draw    (49,130) -- (172,130) ;
	\draw    (110,71.5) -- (110,189) ;
	\draw    (263,130) -- (312,130) ;
	\draw    (324,71.5) -- (324,189) ;
	\draw    (215,130) -- (245,130) ;
	\draw [shift={(248,130)}, rotate = 180] [color={rgb, 255:red, 0; green, 0; blue, 0 }  ][line width=0.75]    (10.93,-3.29) .. controls (6.95,-1.4) and (3.31,-0.3) .. (0,0) .. controls (3.31,0.3) and (6.95,1.4) .. (10.93,3.29)   ;
	\draw    (215,130) -- (187,130) ;
	\draw [shift={(184,130)}, rotate = 360] [color={rgb, 255:red, 0; green, 0; blue, 0 }  ][line width=0.75]    (10.93,-3.29) .. controls (6.95,-1.4) and (3.31,-0.3) .. (0,0) .. controls (3.31,0.3) and (6.95,1.4) .. (10.93,3.29)   ;
	\draw   (101.75,129.5) .. controls (101.75,124.67) and (105.67,120.75) .. (110.5,120.75) .. controls (115.33,120.75) and (119.25,124.67) .. (119.25,129.5) .. controls (119.25,134.33) and (115.33,138.25) .. (110.5,138.25) .. controls (105.67,138.25) and (101.75,134.33) .. (101.75,129.5) -- cycle ;	
    \filldraw[color={rgb, 255:red, 0; green, 0; blue, 0 }  ,draw opacity=1] (172,130) circle (2pt);
    \filldraw[color={rgb, 255:red, 0; green, 0; blue, 0 }  ,draw opacity=1] (312,130) circle (2pt);
   \draw (204,100) node [anchor=north west][inner sep=0.75pt]   [align=left] {\(\Omega_v\)};
\end{tikzpicture}
\caption{$\Omega_v$-move.}
\label{Figure 7}
\end{center}
\end{figure}
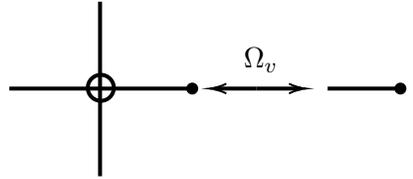
The \(\Omega_v\)-move allows a segment adjacent to the tail or the head to be slide back or forward, deleting or creating consecutive virtual crossings on that segment. That is to say, the \(\Omega_v\)-move either reduces or increases the number of virtual crossings while altering the positions of the endpoints. The classical Reidemeister moves (\(\Omega_1-\Omega_3\)), virtual Reidemeister moves ($V\Omega_{1}-V\Omega_{4}$) together with \(\Omega_v\)-move are called the \textit{generalized $\Omega$-moves}.
Two virtual knotoid diagrams are said to be \textit{equivalent} if one can be obtained from the other by a finite sequence of generalized $\Omega$-moves and isotopies of $S^2$. The corresponding equivalence classes are called \textit{virtual knotoids}. 

If the forbidden $\Phi_{\text{over}}$-move is allowed for virtual knots, then welded knot theory~\citep{rourke2007welded} can be obtained. G{\"u}g{\"u}m{\"c}{\"u} and Kauffman~\citep{gugumcu2017new} defined that two virtual knotoid diagrams are said to be \textit{w-equivalent} if one can be obtained from the other by a finite sequence of the generalized $\Omega$-moves, the forbidden $\Phi_{\text{over}}$-move and the forbidden knotoid $\Phi_-$-move. The corresponding equivalence classes are called \textit{welded virtual knotoids}. Satoh also defined \textit{w-equivalent} of two virtual knotoid diagrams in same way and he referred to virtual knotoid diagrams  as \textit{virtual arc diagrams} in~\citep{Satoh2000}. 
In this paper, we mainly focus on the plus-welded knotoids, which are the virtual knotoids allowing the forbidden knotoid $\Phi_+$-move and the forbidden $\Phi_{\text{over}}$-move. We start with the following definition.

\begin{definition} \label{Definition 2.1}
\rm{Two virtual knotoid diagrams are said to be \textit{plus-welded equivalent} if one can be transformed to another  through a finite sequence of generalized $\Omega$-moves, $\Phi_{\text{over}}$-moves, and $\Phi_+$-moves.  The corresponding equivalence classes are called \textit{plus-welded knotoids}. }
\end{definition}

The map on the set of virtual knotoids to the set of virtual knots, denoted by $c^{v}$, is referred to as \textit{virtual closure}~\citep{MANOURAS2021103402} for virtual knotoids. It is induced by the virtual closure on knotoid diagrams, when a shortcut is taken for the knotoid diagram, all additional crossings are stipulated to be virtual crossings. Namely 
$$
c^{v}([D]):=\left[D^{v}\right],
$$ 
where $D^{v}$ represents the virtual knot diagram obtained by performing the virtual closure operation on the knotoid diagram $D$, see~\cref{Figure 8}.

\vspace{-8.mm}

\begin{figure}[htbp]
\centering
\tikzset{every picture/.style={line width=1.5pt}}
\begin{tikzpicture}[x=0.75pt,y=0.75pt,yscale=-0.75,xscale=0.75]
    \draw    (119,110.5) .. controls (205,118) and (189,-39) .. (122,32) ;
    \draw    (116,62) .. controls (95,-66) and (7,104) .. (119,110.5) ;
    \draw    (115,78) .. controls (119,166) and (229,113) .. (177,90) ;
    \draw    (94,47) .. controls (68,79) and (129,61) .. (163,81) ;
    \draw   (116.98,110.44) .. controls (116.91,114.04) and (119.79,117.01) .. (123.39,117.07) .. controls (126.99,117.14) and (129.96,114.26) .. (130.02,110.66) .. controls (130.08,107.06) and (127.21,104.09) .. (123.61,104.03) .. controls (120.01,103.97) and (117.04,106.84) .. (116.98,110.44) -- cycle ;
    \draw   (134.06,66.32) -- (121.87,69.78) -- (131.83,77.61) ;
    \draw    (361,108.5) .. controls (447,116) and (441.33,-39) .. (351.33,34) ;
    \draw    (358,60) .. controls (337,-68) and (249,102) .. (361,108.5) ;
    \draw    (357,76) .. controls (361,164) and (471,111) .. (419,88) ;
    \draw    (351.33,34) .. controls (308.33,74) and (371,59) .. (405,79) ;
    \draw   (358.98,108.44) .. controls (358.91,112.04) and (361.79,115.01) .. (365.39,115.07) .. controls (368.99,115.14) and (371.96,112.26) .. (372.02,108.66) .. controls (372.08,105.06) and (369.21,102.09) .. (365.61,102.03) .. controls (362.01,101.97) and (359.04,104.84) .. (358.98,108.44) -- cycle ;
    \draw   (376.48,65.01) -- (364.01,67.23) -- (373.13,76.02) ;
    \draw   (345.33,34) .. controls (345.33,30.69) and (348.02,28) .. (351.33,28) .. controls (354.65,28) and (357.33,30.69) .. (357.33,34) .. controls (357.33,37.31) and (354.65,40) .. (351.33,40) .. controls (348.02,40) and (345.33,37.31) .. (345.33,34) -- cycle ;
    \draw    (220.33,61) -- (278.33,61) ;
    \draw [shift={(281.33,61)}, rotate = 180] [color={rgb, 255:red, 0; green, 0; blue, 0 }  ][line width=0.75]    (10.93,-3.29) .. controls (6.95,-1.4) and (3.31,-0.3) .. (0,0) .. controls (3.31,0.3) and (6.95,1.4) .. (10.93,3.29)   ;
    \filldraw[color={rgb, 255:red, 0; green, 0; blue, 0 }  ,draw opacity=1] (94,47) circle (2pt);
    \filldraw[color={rgb, 255:red, 0; green, 0; blue, 0 }  ,draw opacity=1] (122,32) circle (2pt);
    \draw (98,71) node [anchor=north west][inner sep=0.75pt]   [align=left] {\(q\)};
    \draw (179,71) node [anchor=north west][inner sep=0.75pt]   [align=left] {\(r\)};
    \draw (62,118) node [anchor=north west][inner sep=0.75pt]   [align=left] {\(D\)};
    \draw (340,69) node [anchor=north west][inner sep=0.75pt]   [align=left] {\(q\)};
    \draw (421,69) node [anchor=north west][inner sep=0.75pt]   [align=left] {\(r\)};
    \draw (316,114) node [anchor=north west][inner sep=0.75pt]   [align=left] {\(D^v\)};
    \draw (107,138) node [anchor=north west][inner sep=0.75pt]   [align=left] {(a)};
    \draw (348,138) node [anchor=north west][inner sep=0.75pt]   [align=left] {(b)};				
\end{tikzpicture}
\caption{Virtual closure.}
\label{Figure 8}
\end{figure}
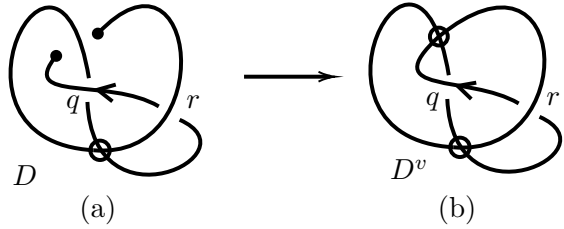 
It is easy to see that the plus-welded knotoid will turn into a welded knot after the virtual closure operation.

\begin{remark}
 \rm{The virtual crossing of a plus-welded knotoid will be called \textit{welded crossing}.} 
\end{remark}

We generalize the Gauss diagram to plus-welded knotoids analogous to~\citep{kim2018family}.
Let $D$ be a plus-welded knotoid diagram with $n$ classical crossings. The \textit{Gauss diagram} of $D$ consists of a counterclockwise-oriented arc $C$ and $n$ signed and oriented chords. These chords connect in pairs $2n$  points on $C$, which respectively correspond to the over-arcs and under-arcs of classical crossings, and are in one-to-one correspondence with classical crossings. For each classical crossing $c$ of $D$, the corresponding chord is still denoted  by $c$, the images of the over-arc and under-arc are referred to as $\text{sgn}(c)$ and $-\text{sgn}(c)$ respectively. The orientation of the chord is from $\text{sgn}(c)$ to $-\text{sgn}(c)$. In this way, we obtain the Gauss diagram of $D$, denoted by $G(D)$, as shown in~\cref{Figure 9} (a).
The tail (resp. head) of $D$ is the starting (resp. ending) point of $G(D)$, also called the \textit{tail} (resp. \textit{head}) of $G(D)$. \cref{Figure 9} (b) shows the Gauss diagram of plus-welded knotoid diagram in~\cref{Figure 8} (a). 
\begin{figure}[htbp]
\begin{center}
\tikzset{every picture/.style={line width=1.5pt}} 
\begin{tikzpicture}[x=0.75pt,y=0.75pt,yscale=-0.75,xscale=0.75]
	\draw  [draw opacity=0] (254.16,120.21) .. controls (245.04,143.99) and (220.65,161) .. (192,161) .. controls (155.55,161) and (126,133.47) .. (126,99.5) .. controls (126,65.53) and (155.55,38) .. (192,38) .. controls (220.99,38) and (245.61,55.41) .. (254.48,79.62) -- (192,99.5) -- cycle ;
	\draw   (254.16,120.21) .. controls (245.04,143.99) and (220.65,161) .. (192,161) .. controls (155.55,161) and (126,133.47) .. (126,99.5) .. controls (126,65.53) and (155.55,38) .. (192,38) .. controls (220.99,38) and (245.61,55.41) .. (254.48,79.62) ;
	\draw    (189,39) -- (189,162) ;
	\draw   (494.16,121.21) .. controls (485.04,144.99) and (460.65,162) .. (432,162) .. controls (395.55,162) and (366,134.47) .. (366,100.5) .. controls (366,66.53) and (395.55,39) .. (432,39) .. controls (460.99,39) and (485.61,56.41) .. (494.48,80.62) ;
	\draw    (467,49) -- (397,152) ;
	\draw    (392,51) -- (470,151) ;
	\draw   (394,62) -- (393.16,51.77) -- (401.83,57.27) ;
	\draw   (405,148) -- (397.28,151.48) -- (398.44,143.1) ;
	\draw   (432.5,168.65) -- (422.23,161.78) -- (432.69,155.19) ;
	\draw   (193.32,149.65) -- (189.6,160.16) -- (185,150) ;
	\draw (198,96) node [anchor=north west][inner sep=0.75pt]   [align=left] {\(c\)};
	\draw (160,167) node [anchor=north west][inner sep=0.75pt]   [align=left] {$-\text{sgn}(c)$};
	\draw (168,13) node [anchor=north west][inner sep=0.75pt]   [align=left] {$\text{sgn}(c)$};
	\draw (441,136) node [anchor=north west][inner sep=0.75pt]   [align=left] {\(r\)};
	\draw (460,70) node [anchor=north west][inner sep=0.75pt]   [align=left] {\(q\)};
	\draw (177,192) node [anchor=north west][inner sep=0.75pt]   [align=left] {(a)};
	\draw (264,108) node [anchor=north west][inner sep=0.75pt]   [align=left] {tail};
	\draw (264,70) node [anchor=north west][inner sep=0.75pt]   [align=left] {head};
	\draw (428,190) node [anchor=north west][inner sep=0.75pt]   [align=left] {(b)};
	\filldraw[color={rgb, 255:red, 0; green, 0; blue, 0 }  ,draw opacity=1] (254.16,120.21) circle (2pt);
	\filldraw[color={rgb, 255:red, 0; green, 0; blue, 0 }  ,draw opacity=1] (254.48,79.62) circle (2pt);	
	\filldraw[color={rgb, 255:red, 0; green, 0; blue, 0 }  ,draw opacity=1] (494.16,121.21) circle (2pt);
	\filldraw[color={rgb, 255:red, 0; green, 0; blue, 0 }  ,draw opacity=1] (494.48,80.62) circle (2pt);		
\end{tikzpicture}
\caption{The Gauss diagram of a plus-welded knotoid diagram.}
\label{Figure 9}
\end{center}
\end{figure}
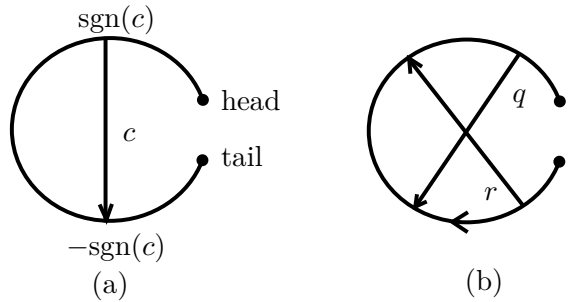

Next we consider three types of moves on Gauss diagrams that correspond to moves $\Omega_{1}$, $\Phi_{\text{over}}$ and $\Phi_+$ on plus-welded knotoid diagrams.
For the $\Omega_1$-move, its effect on a Gauss diagram is to either add or remove a chord whose endpoints are adjacent to each other, as shown in~\cref{Figure 10} (a). Such a chord is referred to as \textit{trivial}.
For the $\Phi_{\text{over}}$-move, its effect on a Gauss diagram is to change the positions of adjacent initial endpoints of two chords, regardless of their signs, as shown in~\cref{Figure 10} (b). 
For the $\Phi_+$-move, its effect on a Gauss diagram is to either remove or add a chord $c$ near the head or tail of the Gauss diagram. Specifically, the starting point of $c$ is the first starting point of encountered when moving away from any endpoint of the Gauss diagram, as shown in~\cref{Figure 10} (c).  

\smallskip 

\begin{figure}[htbp]
\centering
\tikzset{every picture/.style={line width=1.5pt}} 
\begin{tikzpicture}[x=0.75pt,y=0.75pt,yscale=-0.8,xscale=0.8]
	\draw  [draw opacity=0] (267.69,80.7) .. controls (261.35,74.51) and (257.33,66.11) .. (257.02,56.81) -- (294.07,55.64) -- cycle ; \draw [dashed]  (267.69,80.7) .. controls (261.35,74.51) and (257.33,66.11) .. (257.02,56.81) ;
	\draw    (109.97,56.77) -- (122.69,56.77) ;
	\draw [shift={(125.69,56.77)}, rotate = 180] [color={rgb, 255:red, 0; green, 0; blue, 0 }  ][line width=0.75]    (10.93,-3.29) .. controls (6.95,-1.4) and (3.31,-0.3) .. (0,0) .. controls (3.31,0.3) and (6.95,1.4) .. (10.93,3.29)   ;
	\draw    (109.97,56.77) -- (96.12,56.77) ;
	\draw [shift={(93.12,56.77)}, rotate = 1.97] [color={rgb, 255:red, 0; green, 0; blue, 0 }  ][line width=0.75]    (10.93,-3.29) .. controls (6.95,-1.4) and (3.31,-0.3) .. (0,0) .. controls (3.31,0.3) and (6.95,1.4) .. (10.93,3.29)   ;
	\draw  [draw opacity=0] (203.97,67.58) .. controls (198.94,81.54) and (185.15,91.56) .. (168.93,91.56) .. controls (148.46,91.56) and (131.87,75.59) .. (131.87,55.9) .. controls (131.87,36.21) and (148.46,20.25) .. (168.93,20.25) .. controls (185.34,20.25) and (199.25,30.5) .. (204.13,44.7) -- (168.93,55.9) -- cycle ;
	\draw   (203.97,67.58) .. controls (198.94,81.54) and (185.15,91.56) .. (168.93,91.56) .. controls (148.46,91.56) and (131.87,75.59) .. (131.87,55.9) .. controls (131.87,36.21) and (148.46,20.25) .. (168.93,20.25) .. controls (185.34,20.25) and (199.25,30.5) .. (204.13,44.7) ;	
	\draw   (146.49,78.41) -- (141.77,80.09) -- (142.88,75.06) ;
	\draw    (137.48,37.64) .. controls (168.37,47.5) and (167.81,58.51) .. (141.42,80.54) ;
	\draw    (617.92,56.77) -- (608.95,56.77) ;
	\draw [shift={(605.95,56.77)}, rotate = 360] [color={rgb, 255:red, 0; green, 0; blue, 0 }  ][line width=0.75]    (10.93,-3.29) .. controls (6.95,-1.4) and (3.31,-0.3) .. (0,0) .. controls (3.31,0.3) and (6.95,1.4) .. (10.93,3.29)   ;
	\draw    (617.92,56.77) -- (636.18,56.77) ;
	\draw [shift={(639.18,56.77)}, rotate = 180] [color={rgb, 255:red, 0; green, 0; blue, 0 }  ][line width=0.75]    (10.93,-3.29) .. controls (6.95,-1.4) and (3.31,-0.3) .. (0,0) .. controls (3.31,0.3) and (6.95,1.4) .. (10.93,3.29)   ;
	\draw    (653,45) -- (709.58,86.17) ;
	\draw   (658,55) -- (653.17,44.58) -- (664.65,47.18) ;	
	\draw    (360.56,56.77) -- (343.68,56.77) ;
	\draw [shift={(340.68,56.77)}, rotate = 1.82] [color={rgb, 255:red, 0; green, 0; blue, 0 }  ][line width=0.75]    (10.93,-3.29) .. controls (6.95,-1.4) and (3.31,-0.3) .. (0,0) .. controls (3.31,0.3) and (6.95,1.4) .. (10.93,3.29)   ;
	\draw    (360.56,56.77) -- (372.62,56.77) ;
	\draw [shift={(375.62,56.77)}, rotate = 180] [color={rgb, 255:red, 0; green, 0; blue, 0 }  ][line width=0.75]    (10.93,-3.29) .. controls (6.95,-1.4) and (3.31,-0.3) .. (0,0) .. controls (3.31,0.3) and (6.95,1.4) .. (10.93,3.29)   ;
	\draw    (265.61,32.09) .. controls (291,35) and (289,84) .. (277,88) ;
	\draw    (312,87) .. controls (299.35,73.74) and (295,44) .. (322.41,31.73) ;
	\draw   (275.88,42.05) -- (268.09,31.83) -- (280.49,33.42) ;
	\draw   (308,33) -- (321.52,32.08) -- (314.06,44.03) ;
	\draw    (456.52,34.25) -- (411,87) ;
	\draw    (396.52,32.36) -- (444,86) ;
	\draw   (400.42,44.1) -- (397.17,33.86) -- (405.95,36.97) ;
	\draw   (447.03,37.95) -- (456.74,35.23) -- (453,45) ;
	\draw  [draw opacity=0] (257,55.96) .. controls (257,55.85) and (257,55.75) .. (257,55.64) .. controls (257,41.28) and (265.83,28.9) .. (278.55,23.25) -- (294.07,55.64) -- cycle ;
	\draw   (257,55.96) .. controls (257,55.85) and (257,55.75) .. (257,55.64) .. controls (257,41.28) and (265.83,28.9) .. (278.55,23.25) ;
	\draw  [draw opacity=0] (84.1,69.32) .. controls (79.07,83.27) and (65.29,93.3) .. (49.07,93.3) .. controls (28.59,93.3) and (12,77.33) .. (12,57.64) .. controls (12,37.95) and (28.59,21.99) .. (49.07,21.99) .. controls (65.47,21.99) and (79.38,32.24) .. (84.27,46.44) -- (49.07,57.64) -- cycle ; \draw   (84.1,69.32) .. controls (79.07,83.27) and (65.29,93.3) .. (49.07,93.3) .. controls (28.59,93.3) and (12,77.33) .. (12,57.64) .. controls (12,37.95) and (28.59,21.99) .. (49.07,21.99) .. controls (65.47,21.99) and (79.38,32.24) .. (84.27,46.44) ;
	\draw  [draw opacity=0] (311.2,24.02) .. controls (319.67,28.28) and (326.22,35.59) .. (329.27,44.44) -- (294.07,55.64) -- cycle ; \draw   (311.2,24.02) .. controls (319.67,28.28) and (326.22,35.59) .. (329.27,44.44) ;
	\draw  [draw opacity=0] (318.55,82.41) .. controls (312.02,87.94) and (303.45,91.3) .. (294.07,91.3) .. controls (285.73,91.3) and (278.04,88.65) .. (271.85,84.19) -- (294.07,55.64) -- cycle ;
	\draw   (318.55,82.41) .. controls (312.02,87.94) and (303.45,91.3) .. (294.07,91.3) .. controls (285.73,91.3) and (278.04,88.65) .. (271.85,84.19) ;
	\draw  [draw opacity=0] (330.32,48.17) .. controls (330.85,50.58) and (331.13,53.08) .. (331.13,55.64) .. controls (331.13,64.58) and (327.71,72.76) .. (322.06,79.01) -- (294.07,55.64) -- cycle ;
	\draw  [dashed] (330.32,48.17) .. controls (330.85,50.58) and (331.13,53.08) .. (331.13,55.64) .. controls (331.13,64.58) and (327.71,72.76) .. (322.06,79.01) ;
	\draw  [draw opacity=0] (399.69,79.7) .. controls (393.35,73.51) and (389.33,65.11) .. (389.02,55.81) -- (426.07,54.64) -- cycle ; \draw [dashed]  (399.69,79.7) .. controls (393.35,73.51) and (389.33,65.11) .. (389.02,55.81) ;
	\draw  [draw opacity=0] (389,54.96) .. controls (389,54.85) and (389,54.75) .. (389,54.64) .. controls (389,40.28) and (397.83,27.9) .. (410.55,22.25) -- (426.07,54.64) -- cycle ;
    \draw   (389,54.96) .. controls (389,54.85) and (389,54.75) .. (389,54.64) .. controls (389,40.28) and (397.83,27.9) .. (410.55,22.25) ;
	\draw  [draw opacity=0] (443.2,23.02) .. controls (451.67,27.28) and (458.22,34.59) .. (461.27,43.44) -- (426.07,54.64) -- cycle ; \draw   (443.2,23.02) .. controls (451.67,27.28) and (458.22,34.59) .. (461.27,43.44) ;
	\draw  [draw opacity=0] (450.55,81.41) .. controls (444.02,86.94) and (435.45,90.3) .. (426.07,90.3) .. controls (417.73,90.3) and (410.04,87.65) .. (403.85,83.19) -- (426.07,54.64) -- cycle ;
	\draw   (450.55,81.41) .. controls (444.02,86.94) and (435.45,90.3) .. (426.07,90.3) .. controls (417.73,90.3) and (410.04,87.65) .. (403.85,83.19) ;
	\draw  [draw opacity=0] (462.32,47.17) .. controls (462.85,49.58) and (463.13,52.08) .. (463.13,54.64) .. controls (463.13,63.58) and (459.71,71.76) .. (454.06,78.01) -- (426.07,54.64) -- cycle ;
	\draw [dashed]  (462.32,47.17) .. controls (462.85,49.58) and (463.13,52.08) .. (463.13,54.64) .. controls (463.13,63.58) and (459.71,71.76) .. (454.06,78.01) ;
	\draw  [draw opacity=0] (530,30.28) .. controls (536.18,23.93) and (544.59,19.91) .. (553.89,19.6) -- (555.07,56.64) -- cycle ; \draw  [dashed] (530,30.28) .. controls (536.18,23.93) and (544.59,19.91) .. (553.89,19.6) ;
	\draw  [draw opacity=0] (554.74,19.58) .. controls (554.84,19.58) and (554.95,19.58) .. (555.05,19.58) .. controls (569.42,19.57) and (581.8,28.4) .. (587.45,41.12) -- (555.07,56.64) -- cycle ;
	\draw   (554.74,19.58) .. controls (554.84,19.58) and (554.95,19.58) .. (555.05,19.58) .. controls (569.42,19.57) and (581.8,28.4) .. (587.45,41.12) ;
	\draw  [draw opacity=0] (586.7,73.77) .. controls (582.44,82.24) and (575.13,88.79) .. (566.28,91.84) -- (555.07,56.64) -- cycle ; \draw   (586.7,73.77) .. controls (582.44,82.24) and (575.13,88.79) .. (566.28,91.84) ;
	\draw  [draw opacity=0] (528.3,81.13) .. controls (522.77,74.61) and (519.42,66.04) .. (519.41,56.66) .. controls (519.41,48.32) and (522.05,40.63) .. (526.51,34.44) -- (555.07,56.64) -- cycle ;
	\draw   (528.3,81.13) .. controls (522.77,74.61) and (519.42,66.04) .. (519.41,56.66) .. controls (519.41,48.32) and (522.05,40.63) .. (526.51,34.44) ;
	\draw  [draw opacity=0] (562.55,92.89) .. controls (560.14,93.43) and (557.64,93.71) .. (555.08,93.71) .. controls (546.14,93.71) and (537.96,90.29) .. (531.7,84.64) -- (555.07,56.64) -- cycle ;
	\draw [dashed]  (562.55,92.89) .. controls (560.14,93.43) and (557.64,93.71) .. (555.08,93.71) .. controls (546.14,93.71) and (537.96,90.29) .. (531.7,84.64) ;
	\draw  [draw opacity=0] (661,31.28) .. controls (667.18,24.93) and (675.59,20.91) .. (684.89,20.6) -- (686.07,57.64) -- cycle ; \draw [dashed]  (661,31.28) .. controls (667.18,24.93) and (675.59,20.91) .. (684.89,20.6) ;
	\draw  [draw opacity=0] (685.74,20.58) .. controls (685.84,20.58) and (685.95,20.58) .. (686.05,20.58) .. controls (700.42,20.57) and (712.8,29.4) .. (718.45,42.12) -- (686.07,57.64) -- cycle ;
	\draw   (685.74,20.58) .. controls (685.84,20.58) and (685.95,20.58) .. (686.05,20.58) .. controls (700.42,20.57) and (712.8,29.4) .. (718.45,42.12) ;
	\draw  [draw opacity=0] (717.7,74.77) .. controls (713.44,83.24) and (706.13,89.79) .. (697.28,92.84) -- (686.07,57.64) -- cycle ; \draw   (717.7,74.77) .. controls (713.44,83.24) and (706.13,89.79) .. (697.28,92.84) ;
	\draw  [draw opacity=0] (659.3,82.13) .. controls (653.77,75.61) and (650.42,67.04) .. (650.41,57.66) .. controls (650.41,49.32) and (653.05,41.63) .. (657.51,35.44) -- (686.07,57.64) -- cycle ;
	\draw   (659.3,82.13) .. controls (653.77,75.61) and (650.42,67.04) .. (650.41,57.66) .. controls (650.41,49.32) and (653.05,41.63) .. (657.51,35.44) ;
	\draw  [draw opacity=0] (693.55,93.89) .. controls (691.14,94.43) and (688.64,94.71) .. (686.08,94.71) .. controls (677.14,94.71) and (668.96,91.29) .. (662.7,85.64) -- (686.07,57.64) -- cycle ;
	\draw  [dashed] (693.55,93.89) .. controls (691.14,94.43) and (688.64,94.71) .. (686.08,94.71) .. controls (677.14,94.71) and (668.96,91.29) .. (662.7,85.64) ;
	\filldraw[color={rgb, 255:red, 0; green, 0; blue, 0 }  ,draw opacity=1] (84.1,69.32) circle (2pt);
	\filldraw[color={rgb, 255:red, 0; green, 0; blue, 0 }  ,draw opacity=1] (84.27,46.44) circle (2pt);	
	\filldraw[color={rgb, 255:red, 0; green, 0; blue, 0 }  ,draw opacity=1] (203.97,67.58) circle (2pt);
	\filldraw[color={rgb, 255:red, 0; green, 0; blue, 0 }  ,draw opacity=1] (204.13,44.7) circle (2pt);	
	\filldraw[color={rgb, 255:red, 0; green, 0; blue, 0 }  ,draw opacity=1] (278.55,23.25) circle (2pt);
	\filldraw[color={rgb, 255:red, 0; green, 0; blue, 0 }  ,draw opacity=1] (311.2,24.02) circle (2pt);	
	\filldraw[color={rgb, 255:red, 0; green, 0; blue, 0 }  ,draw opacity=1] (410.55,22.25) circle (2pt);
	\filldraw[color={rgb, 255:red, 0; green, 0; blue, 0 }  ,draw opacity=1] (443.2,23.02) circle (2pt);	
	\filldraw[color={rgb, 255:red, 0; green, 0; blue, 0 }  ,draw opacity=1] (587.45,41.12) circle (2pt);
	\filldraw[color={rgb, 255:red, 0; green, 0; blue, 0 }  ,draw opacity=1] (586.7,73.77) circle (2pt);	
	\filldraw[color={rgb, 255:red, 0; green, 0; blue, 0 }  ,draw opacity=1] (718.45,42.12) circle (2pt);
	\filldraw[color={rgb, 255:red, 0; green, 0; blue, 0 }  ,draw opacity=1] (717.7,74.77) circle (2pt);	
	\draw (690,60) node [anchor=north west][inner sep=0.75pt]   [align=left] {$c$};
	\draw (88.35,103) node [anchor=north west][inner sep=0.75pt]   [align=left] {(a)};
	\draw (344.95,103) node [anchor=north west][inner sep=0.75pt]   [align=left] {(b)};
	\draw (610.99,103) node [anchor=north west][inner sep=0.75pt]   [align=left] {(c)};
\end{tikzpicture}
\caption{The Gauss diagrams presenting $\Omega_{1}$-move, $\Phi_{\text{over}}$-move, and $\Phi_+$-move.}
\label{Figure 10}
\end{figure}
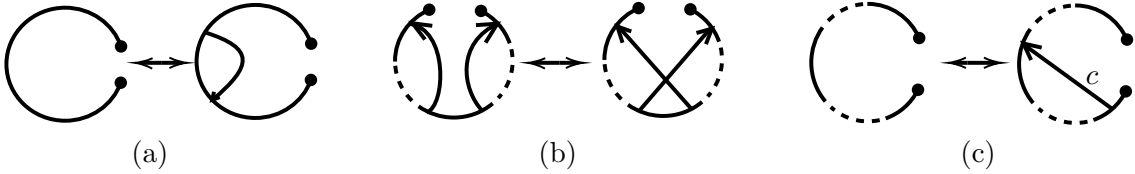

\section{Descending diagrams for plus-welded knotoids} \label{section3}

Before introducing the definition of a descending diagram, we first explain what is a base point on a plus-welded  knotoid diagram.
Let $D$ be a diagram of a plus-welded knotoid $K$. If a point $a$ on the diagram $D$ is neither any crossing point nor any endpoint of $D$, then such a point $a$ is called a \textit{base point} of $D$. The diagram $D$ with the base point $a$ will be denoted by $D_a$.
Next we generalize the definition of a descending diagram to plus-welded knotoid similar to~\citep{satoh2018crossing}.

\begin{definition} \label{Definition 3.1}
\rm{
Let $D$ be a plus-welded knotoid diagram. If there exists a base point $a$ such that when traversing along the orientation of $D$ from $a$ towards the head and then proceeding from the tail along the orientation of $D$ to $a$, at every classical crossing, the over-crossing is encountered for the first time and the under-crossing is encountered for the second time, then $D$ is referred to as \textit{descending diagram}.
}
\end{definition}

It should be noted that if a plus-welded knotoid diagram $D$ is descending, then the welded knot diagram $D^{v}$ obtained through virtual closure operation is also descending.

With reference to~\cref{Figure 11} (a), if we start from the base point $a$ and traverse all arcs of $D$ in accordance with its orientation, we observe that at the classical crossings $q$ and $r$, the over-crossing is encountered for the first time and the under-crossing is encountered for the second time. Thus, the diagram $D$ in~\cref{Figure 11} (a) is descending, but the diagram $D'$ in~\cref{Figure 11} (b) is not.  


\begin{figure}[htbp]
\centering
\tikzset{every picture/.style={line width=1.5pt}} 
\begin{tikzpicture}[x=0.75pt,y=0.75pt,yscale=-0.75,xscale=0.75]
	\draw    (129,119.5) .. controls (215,127) and (199,-30) .. (132,41) ;
	\draw    (126,71) .. controls (105,-57) and (17,113) .. (129,119.5) ;
	\draw    (125,87) .. controls (129,175) and (239,122) .. (187,99) ;
	\draw    (104,56) .. controls (78,88) and (139,70) .. (173,90) ;
	\draw   (126.98,119.44) .. controls (126.91,123.04) and (129.79,126.01) .. (133.39,126.07) .. controls (136.99,126.14) and (139.96,123.26) .. (140.02,119.66) .. controls (140.08,116.06) and (137.21,113.09) .. (133.61,113.03) .. controls (130.01,112.97) and (127.04,115.84) .. (126.98,119.44) -- cycle ;
	\draw   (144.06,75.32) -- (131.87,78.78) -- (141.83,86.61) ;
	\draw    (355,121.5) .. controls (441,129) and (425,-28) .. (358,43) ;
	\draw    (352,73) .. controls (331,-55) and (243,115) .. (355,121.5) ;
	\draw    (351,89) .. controls (355,177) and (465,124) .. (413,101) ;
	\draw    (330,58) .. controls (304,90) and (365,72) .. (399,92) ;	
	\draw   (352.98,121.44) .. controls (352.91,125.04) and (355.79,128.01) .. (359.39,128.07) .. controls (362.99,128.14) and (365.96,125.26) .. (366.02,121.66) .. controls (366.08,118.06) and (363.21,115.09) .. (359.61,115.03) .. controls (356.01,114.97) and (353.04,117.84) .. (352.98,121.44) -- cycle ;
	\draw   (357.12,86.9) -- (369.01,82.55) -- (358.5,75.48) ;
	\filldraw[color={rgb, 255:red, 0; green, 0; blue, 0 }  ,draw opacity=1] (132,41) circle (2pt);
	\filldraw[color={rgb, 255:red, 0; green, 0; blue, 0 }  ,draw opacity=1] (104,56) circle (2pt);
	\filldraw[color={rgb, 255:red, 0; green, 0; blue, 0 }  ,draw opacity=1] (330,58) circle (2pt);
	\filldraw[color={rgb, 255:red, 0; green, 0; blue, 0 }  ,draw opacity=1] (358,43) circle (2pt);
	\filldraw[color={rgb, 255:red, 0; green, 0; blue, 0 }  ,draw opacity=1] (155,83) circle (2pt);
	\filldraw[color={rgb, 255:red, 0; green, 0; blue, 0 }  ,draw opacity=1] (381,85) circle (2pt);
	\draw (108,80) node [anchor=north west][inner sep=0.75pt]   [align=left] {\(q\)};
	\draw (189,80) node [anchor=north west][inner sep=0.75pt]   [align=left] {\(r\)};
	\draw (155,70) node [anchor=north west][inner sep=0.75pt]   [align=left] {\(a\)};
	\draw (334,82) node [anchor=north west][inner sep=0.75pt]   [align=left] {\(q\)};
	\draw (415,82) node [anchor=north west][inner sep=0.75pt]   [align=left] {\(r\)};
	\draw (381,70) node [anchor=north west][inner sep=0.75pt]   [align=left] {\(a\)};
	\draw (72,127) node [anchor=north west][inner sep=0.75pt]   [align=left] {\(D\)};
	\draw (302,127) node [anchor=north west][inner sep=0.75pt]   [align=left] {\(D'\)};
	\draw (127,153) node [anchor=north west][inner sep=0.75pt]   [align=left] {(a)};
	\draw (357,154) node [anchor=north west][inner sep=0.75pt]   [align=left] {(b)};		
\end{tikzpicture}
\caption{Two plus-welded virtual knotoid diagrams \(D\) and \(D'\).}
\label{Figure 11}
\end{figure}
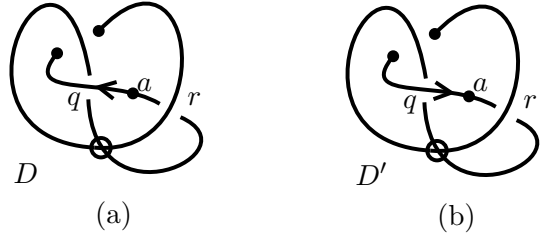

It is stated in~\citep{satoh2018crossing} that any descending diagram D of a welded knot is related to the trivial diagram by a finite sequence of moves $\Omega_1$, $V\Omega_1-V\Omega_4$, and $\Phi_{\text{over}}$. For plus-welded  knotoids, we have:

\begin{theorem} \label{Theorem 3.1}
Any descending diagram $D$ of plus-welded  knotoid can be transformed into a trivial one by a finite sequence of moves $\Omega_1$, $V\Omega_1-V\Omega_4$, $\Omega_v$, $\Phi_{\text{over}}$, and $\Phi_+$.
\end{theorem}

Before proving~\cref{Theorem 3.1}, we need to prove~\cref{Lemma 3.1} and~\cref{Lemma 3.2}. The relationship between virtual knots and the corresponding Gauss diagrams is described in \citep{Goussarov2000Finite,satoh2018crossing}. Similarly, we have 

\begin{lemma} \label{Lemma 3.1}
A Gauss diagram determines a virtual knotoid diagram up to virtual Reidemeister moves $V\Omega_1 - V\Omega_4$ and $\Omega_v$-move.  
\end{lemma}

\begin{proof} 
That is to say that two virtual knotoid diagrams $D$ and $D'$ define the same Gauss diagrams $G(D) = G(D')$ if and only if $D$ can be transformed into $D'$ by a finite sequence of virtual Reidemeister moves $V\Omega_1 - V\Omega_4$ and $\Omega_v$-moves. 
Indeed, since virtual crossings are not presented on Gauss diagrams, virtual Reidemeister moves and $\Omega_v$-move that only involve virtual crossings do not affect Gauss diagrams. 
On the other hand, the equality $G(D) = G(D')$ implies that the quantities and relative positional relationship of the classical crossings in $D$ and $D'$ are identical when ignoring virtual crossings, and if $D$ and $D'$ are different, then this difference stems from the positions and quantities of the virtual crossings. Since virtual Reidemeister moves and $\Omega_v$-move allow one to move the interior of any arc that does not pass through classical crossings, these differences can be removed. 
\end{proof}

Let $x$ be a classical crossing of a plus-welded  knotoid diagram $D$. Under smoothing $D$ at $x$ according to the knotoid orientation, see \cref{Figure 12},  $D$ will be splitted in two paths. One of them, going from the tail to the head, will be called an \textit{open path}. Anther will be called a \textit{closed} path. 
\smallskip 

\begin{figure}[htbp]
\centering
\tikzset{every picture/.style={line width=1.5pt}}   
\begin{tikzpicture}[x=0.75pt,y=0.75pt,yscale=-0.7,xscale=0.7]

	\draw    (64,34) -- (113.5,82.5) ;
	\draw    (66,130) -- (103,93) ;
	\draw  [pink,line width=2pt]  (162,34) -- (123,73) ;
	
	\draw    (186,79.17) -- (231,79.02) ;
	\draw [shift={(233,79.02)}, rotate = 179.81] [color={rgb, 255:red, 0; green, 0; blue, 0 }  ][line width=0.75]    (10.93,-3.29) .. controls (6.95,-1.4) and (3.31,-0.3) .. (0,0) .. controls (3.31,0.3) and (6.95,1.4) .. (10.93,3.29)   ;
	\draw [pink,line width=2pt]   (113.5,82.5) -- (163,130) ;
	\draw   (158,118.02) -- (162.97,130.68) -- (149.71,127.65) ;
	\draw   (79.16,127.01) -- (66.15,131.01) -- (70.18,118.01) ;
	\draw    (281,31.02) .. controls (323,45.02) and (325,112.02) .. (287,130.02) ;
	\draw  [pink,line width=2pt]  (378,33) .. controls (334,48.02) and (347,119.02) .. (380,130) ;
	\draw   (300.82,129.19) -- (287.23,129.81) -- (294.38,118.23) ;
	\draw   (375.03,119.12) -- (381.38,131.15) -- (367.86,129.61) ;
	\draw   (481.63,130) -- (531,81.02) ; 
	\draw [pink,line width=2pt]   (578.15,129.85) -- (541.2,91.81) ;
	\draw  (482.87,33) -- (521.76,71.26) ;
	\draw  [pink,line width=2pt]   (531,82.5) -- (580.36,33) ;
	\draw   (494.14,126.68) -- (480.74,129.05) -- (486.34,116.65) ;
	\draw   (574.53,116.59) -- (578.16,129.7) -- (565.28,125.31) ;
	\draw    (446,81.02) -- (401,81.02) ;
	\draw [shift={(399,81.02)}, rotate = 360] [color={rgb, 255:red, 0; green, 0; blue, 0 }  ][line width=0.75]    (10.93,-3.29) .. controls (6.95,-1.4) and (3.31,-0.3) .. (0,0) .. controls (3.31,0.3) and (6.95,1.4) .. (10.93,3.29)   ;
	\draw (75,73) node [anchor=north west][inner sep=0.75pt]   [align=left] {$x$};
	\draw (569,77) node [anchor=north west][inner sep=0.75pt]   [align=left] {$x$};
\end{tikzpicture}	
\caption{Smoothing at crossing $x$.}
\label{Figure 12}
\end{figure}
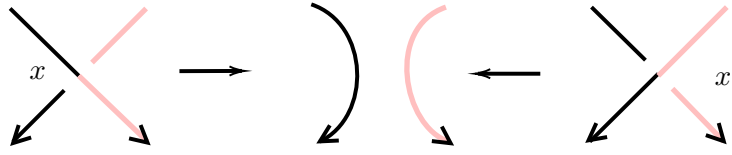

An example is presented in~\cref{Figure 13}, where smoothing $D$ at $x$ will split of $D$ into the open path, formed by pink arcs, and the closed path, formed by black arcs. 

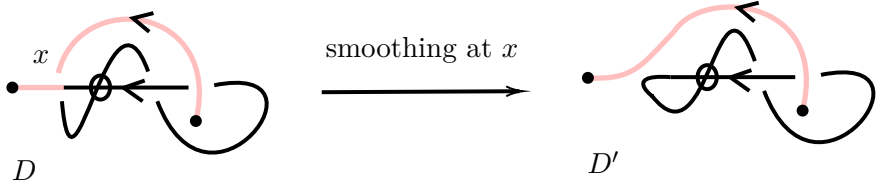
\begin{figure}[htbp]
\tikzset{every picture/.style={line width=1.5pt}}      

\begin{tikzpicture}[x=0.75pt,y=0.75pt,yscale=-0.75,xscale=0.75]
	
	\draw    (207.72,76.66) .. controls (291.62,59.75) and (204.02,178.11) .. (169.47,87.42) ; 
	\draw  [pink,line width=2pt]  (102.85,68.43) .. controls (117,10.67) and (213.89,22.32) .. (195.38,100.72) ;
	\draw    (105.31,87.42) .. controls (112.72,173.5) and (137.39,-4.81) .. (163.3,67.43) ;
	\draw  [pink,line width=2pt]  (72,78.19) -- (106.55,78.19) ;
	\draw   (159.6,87.42) -- (146.86,78.37) -- (160.99,73.48) ;
	\draw   (124.75,78.19) .. controls (124.75,73.74) and (127.65,70.12) .. (131.22,70.12) .. controls (134.8,70.12) and (137.7,73.74) .. (137.7,78.19) .. controls (137.7,82.65) and (134.8,86.26) .. (131.22,86.26) .. controls (127.65,86.26) and (124.75,82.65) .. (124.75,78.19) -- cycle ;
	\draw   (106.55,78.19) -- (190.45,78.19) ;
	\draw   (164,40.67) -- (153.07,31.57) -- (166.87,28.1) ;
	\draw    (280,81.33) -- (413,80.68) ;
	\draw [shift={(415,80.67)}, rotate = 179.7] [color={rgb, 255:red, 0; green, 0; blue, 0 }  ][line width=0.75]    (10.93,-3.29) .. controls (6.95,-1.4) and (3.31,-0.3) .. (0,0) .. controls (3.31,0.3) and (6.95,1.4) .. (10.93,3.29)   ;
	\draw    (615.72,69.66) .. controls (699.62,52.75) and (612.02,171.12) .. (577.47,80.42) ;
	\draw  [pink,line width=2pt]  (517,40.67) .. controls (534,11.67) and (621.89,15.32) .. (603.38,93.72) ; 
	\draw    (502,83.67) .. controls (541,133.67) and (545.39,-11.81) .. (571.3,60.44) ;
	\draw   (567.6,80.42) -- (554.86,71.37) -- (568.99,66.49) ;
	\draw   (532.75,71.2) .. controls (532.75,66.74) and (535.65,63.13) .. (539.22,63.13) .. controls (542.8,63.13) and (545.7,66.74) .. (545.7,71.2) .. controls (545.7,75.66) and (542.8,79.27) .. (539.22,79.27) .. controls (535.65,79.27) and (532.75,75.66) .. (532.75,71.2) -- cycle ;
	\draw    (514.55,71.2) -- (598.45,71.2) ;
	\draw   (571,33.68) -- (560.07,24.57) -- (573.87,21.11) ;
	\draw [pink,line width=2pt]   (459,71.67) .. controls (486.04,73.36) and (493.8,65.06) .. (508.23,49.77) .. controls (510.87,46.97) and (513.74,43.93) .. (517,40.67) ;
	\draw    (502,83.67) .. controls (507,87.67) and (479,68.67) .. (514.55,71.2) ;
	
	\filldraw[color={rgb, 255:red, 0; green, 0; blue, 0 }  ,draw opacity=1] (195.38,100.72) circle (2pt);
	\filldraw[color={rgb, 255:red, 0; green, 0; blue, 0 }  ,draw opacity=1] (72,78.19) circle (2pt);
	\filldraw[color={rgb, 255:red, 0; green, 0; blue, 0 }  ,draw opacity=1] (459,71.67) circle (2pt);
	\filldraw[color={rgb, 255:red, 0; green, 0; blue, 0 }  ,draw opacity=1] (603.38,93.72) circle (2pt);
	
	\draw (83.16,50.48) node [anchor=north west][inner sep=0.75pt]   [align=left] {$x$};
	\draw (69,124.33) node [anchor=north west][inner sep=0.75pt]   [align=left] {$D$};
	\draw (456,118.33) node [anchor=north west][inner sep=0.75pt]   [align=left] {$D'$};
	\draw (280,41.33) node [anchor=north west][inner sep=0.75pt]   [align=left] {smoothing at $x$};
	
\end{tikzpicture}
\caption{After smoothing a diagram will split into open and closed paths.}
\label{Figure 13}
\end{figure}

\begin{lemma} \label{Lemma 3.2}	
Let $x$ be a classical crossing of a plus-welded knotoid diagram $D$. Consider two paths in which $D$ will split after smoothing at $x$. Suppose that one of the paths contains no under-crossing points other than $x$. Let $E$ be the plus-welded knotoid diagram obtained from $D$ by replacing $x$ with a welded crossing. Then, $D$ can be transformed into $E$ by a finite sequence of moves $\Omega_1$, $V\Omega_1 - V\Omega_4$, $\Omega_v$, $\Phi_{\text{over}}$, and $\Phi_+$. 
\end{lemma}

\begin{proof}
Let $\alpha$ be the path which has no under-crossing except $x$. Then $\alpha$ can be either a closed path or an open path. Since $x$ is the sole under-crossing on $\alpha$, the Gauss diagram corresponding to $\alpha$ includes only the starting points of the chords and does not contain the ending points of the chords other than $x$. Next, we will discuss it in two cases.

\underline{\textbf{Case 1.} $\alpha$ is a closed path.} We repeatedly apply $\Phi_{\text{over}}$-moves to contract the chord $c$ corresponding to $x$ into a trivial chord $c_1$, and then use one $\Omega_1$-move to eliminate $c_1$, thereby obtaining a new Gauss diagram $G(D')$. We find that $G(D') = G(E)$, and by~\cref{Lemma 3.1}, it follows that $D'$ can be transformed into $E$ under virtual Reidemeister moves $V\Omega_1 - V\Omega_4$ and $\Omega_v$-move. Consequently,
$D$ can be transformed into $E$ by a finite sequence of moves $\Omega_1$, $V\Omega_1 - V\Omega_4$, $\Omega_v$, and $\Phi_{\text{over}}$. The schematic diagram of this proof is shown in~\cref{Figure 14}. 
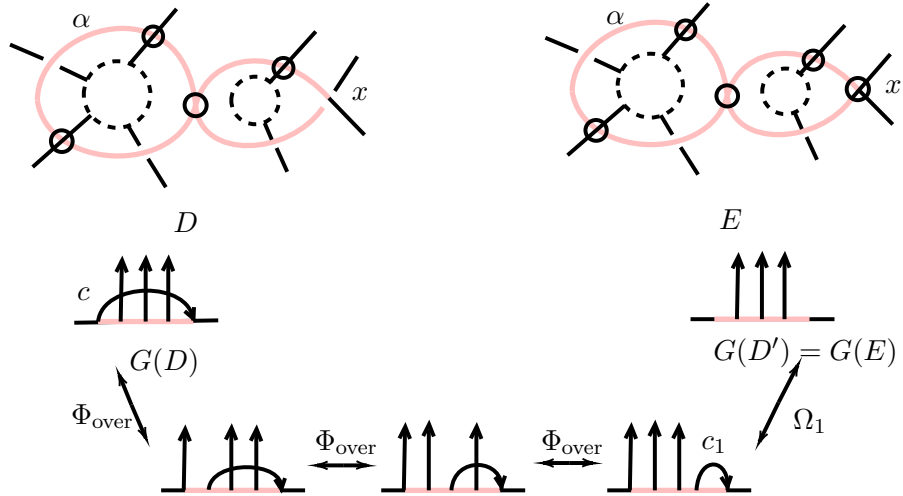
\begin{figure}[htbp]
\centering
\tikzset{every picture/.style={line width=1.5pt}}
\begin{tikzpicture}[x=0.75pt,y=0.75pt,yscale=-0.6,xscale=0.6]
	\draw  [pink,line width=2pt]  (178,95) .. controls (190,50) and (249,43) .. (290,89) ;
	\draw  [pink,line width=2pt]   (45,91) .. controls (53,149) and (162,156) .. (178,95) ;
	\draw  [pink,line width=2pt]  (45,91) .. controls (45,28) and (168.07,-7.52) .. (178,95) ;
	\draw  [pink,line width=2pt]   (285,97) .. controls (257,138) and (182,142) .. (178,95) ;
	\draw    (295,83) -- (314,54) ;
	\draw    (241,75) -- (279,33) ;
	\draw    (121,112) -- (131,129) ;
	\draw    (290,89) -- (320,120) ;
	\draw    (122.76,60.13) -- (162,14) ;
	\draw    (64,63) -- (87,74) ;
	\draw    (153,164) -- (138,140) ;
	\draw   (243.47,63.63) .. controls (243.42,59) and (247.13,55.25) .. (251.76,55.25) .. controls (256.38,55.25) and (260.17,59) .. (260.22,63.63) .. controls (260.27,68.25) and (256.56,72) .. (251.94,72) .. controls (247.31,72) and (243.52,68.25) .. (243.47,63.63) -- cycle ;
	\draw   (56.38,125.25) .. controls (56.32,120.21) and (60.36,116.13) .. (65.4,116.13) .. controls (70.44,116.13) and (74.57,120.21) .. (74.62,125.25) .. controls (74.68,130.29) and (70.64,134.37) .. (65.6,134.37) .. controls (60.56,134.37) and (56.43,130.29) .. (56.38,125.25) -- cycle ;
	\draw   (134.12,37.07) .. controls (134.07,32.51) and (137.73,28.81) .. (142.29,28.81) .. controls (146.85,28.81) and (150.59,32.51) .. (150.64,37.07) .. controls (150.69,41.63) and (147.03,45.33) .. (142.47,45.33) .. controls (137.91,45.33) and (134.17,41.63) .. (134.12,37.07) -- cycle ;
	\draw   (168.74,95) .. controls (168.68,89.88) and (172.78,85.74) .. (177.9,85.74) .. controls (183.02,85.74) and (187.21,89.88) .. (187.26,95) .. controls (187.32,100.12) and (183.22,104.26) .. (178.1,104.26) .. controls (172.98,104.26) and (168.79,100.12) .. (168.74,95) -- cycle ;
	\draw [dashed]  (84,86) .. controls (84,70.54) and (96.54,58) .. (112,58) .. controls (127.46,58) and (140,70.54) .. (140,86) .. controls (140,101.46) and (127.46,114) .. (112,114) .. controls (96.54,114) and (84,101.46) .. (84,86) -- cycle ;
	\draw [dashed]  (208,91) .. controls (208,79.95) and (216.95,71) .. (228,71) .. controls (239.05,71) and (248,79.95) .. (248,91) .. controls (248,102.05) and (239.05,111) .. (228,111) .. controls (216.95,111) and (208,102.05) .. (208,91) -- cycle ;
	\draw    (50,56) -- (22,44) ;
	\draw    (242,124) -- (236,110) ;
	\draw    (246,131) -- (255,151) ;
	\draw    (90,103.5) -- (41,147) ;
	\draw [pink,line width=2pt]   (624.86,87.18) .. controls (636.62,44.03) and (694.43,37.31) .. (734.6,81.43) ;
	\draw  [pink,line width=2pt]  (494.54,83.35) .. controls (502.38,138.97) and (609.18,145.69) .. (624.86,87.18) ;
	\draw  [pink,line width=2pt]  (494.54,83.35) .. controls (494.54,22.93) and (615.13,-11.14) .. (624.86,87.18) ;
	\draw [pink,line width=2pt]   (729.7,89.1) .. controls (702.27,128.42) and (628.78,132.26) .. (624.86,87.18) ;
	\draw    (729.7,89.1) -- (762,52) ;
	\draw    (686.59,68) -- (723.83,27.72) ;
	\draw    (569.01,103.49) -- (578.81,119.79) ;
	\draw    (734.6,81.43) -- (764,111.16) ;
	\draw    (570.73,53.74) -- (609.18,9.5) ;
	\draw    (513.15,56.49) -- (535.69,67.04) ;
	\draw    (600.36,153.36) -- (585.66,130.34) ;
	\draw   (689.01,57.09) .. controls (688.97,52.66) and (692.6,49.06) .. (697.13,49.06) .. controls (701.67,49.06) and (705.38,52.66) .. (705.43,57.09) .. controls (705.48,61.53) and (701.84,65.13) .. (697.31,65.13) .. controls (692.78,65.13) and (689.06,61.53) .. (689.01,57.09) -- cycle ;
	\draw   (505.68,116.2) .. controls (505.63,111.36) and (509.59,107.44) .. (514.53,107.44) .. controls (519.47,107.44) and (523.51,111.36) .. (523.57,116.2) .. controls (523.62,121.03) and (519.66,124.95) .. (514.72,124.95) .. controls (509.78,124.95) and (505.74,121.03) .. (505.68,116.2) -- cycle ;
	\draw   (581.86,31.62) .. controls (581.82,27.25) and (585.4,23.7) .. (589.87,23.7) .. controls (594.34,23.7) and (598,27.25) .. (598.05,31.62) .. controls (598.1,36) and (594.51,39.54) .. (590.04,39.54) .. controls (585.57,39.54) and (581.91,36) .. (581.86,31.62) -- cycle ;
	\draw   (615.78,87.18) .. controls (615.73,82.28) and (619.75,78.3) .. (624.76,78.3) .. controls (629.77,78.3) and (633.88,82.28) .. (633.94,87.18) .. controls (633.99,92.09) and (629.97,96.07) .. (624.96,96.07) .. controls (619.94,96.07) and (615.84,92.09) .. (615.78,87.18) -- cycle ;
	\draw [dashed]  (532.75,78.55) .. controls (532.75,63.72) and (545.04,51.7) .. (560.19,51.7) .. controls (575.34,51.7) and (587.62,63.72) .. (587.62,78.55) .. controls (587.62,93.38) and (575.34,105.41) .. (560.19,105.41) .. controls (545.04,105.41) and (532.75,93.38) .. (532.75,78.55) -- cycle ;
	\draw [dashed]  (654.26,83.35) .. controls (654.26,72.75) and (663.03,64.17) .. (673.85,64.17) .. controls (684.68,64.17) and (693.45,72.75) .. (693.45,83.35) .. controls (693.45,93.94) and (684.68,102.53) .. (673.85,102.53) .. controls (663.03,102.53) and (654.26,93.94) .. (654.26,83.35) -- cycle ;
	\draw    (499.44,49.78) -- (472,38.27) ;
	\draw    (687.57,115) -- (681.69,101.57) ;
	\draw    (691.49,121.71) -- (700.31,140.89) ;
	\draw    (538.63,95.34) -- (490.62,137.06) ;
	\draw    (76,278) -- (197,276) ;
	\draw  [pink,line width=2pt]  (96,277) -- (176,277) ;
	\draw    (116,225) -- (115,277) ;
	\draw    (136,225) -- (136,277) ;
	\draw    (155,225) -- (155,277) ;
	\draw    (96,277) .. controls (100,243) and (173,242) .. (176,276) ;
	\draw   (176.62,264.84) -- (176.22,274.63) -- (168.79,268.25) ;
	\draw   (111.28,235.29) -- (116.22,225.51) -- (120.27,235.7) ;
	\draw   (131.28,234.29) -- (136.22,224.51) -- (140.27,234.7) ;
	\draw   (150.28,234.29) -- (155.22,224.51) -- (159.27,234.7) ;
	\draw    (149,419) -- (270,419) ;
	\draw  [pink,line width=2pt]    (169,419) -- (249,419) ;
	\draw    (169,368) -- (168,419) ;
	\draw    (229,368) -- (229,419) ;
	\draw    (208,368) -- (208,419) ;
	\draw    (189,418.5) .. controls (191,392.5) and (253,394) .. (250,418) ;
	\draw   (253.16,408.69) -- (250,417.96) -- (244.68,409.74) ;
	\draw   (164.28,377.29) -- (169.22,367.51) -- (173.27,377.7) ;
	\draw   (224.28,377.29) -- (229.22,367.51) -- (233.27,377.7) ;
	\draw   (203.28,376.29) -- (208.22,366.51) -- (212.27,376.7) ;
	\draw    (492,396) -- (512,396) ;
	\draw [shift={(514,396.2)}, rotate = 182.6] [color={rgb, 255:red, 0; green, 0; blue, 0 }  ][line width=0.75]    (10.93,-3.29) .. controls (6.95,-1.4) and (3.31,-0.3) .. (0,0) .. controls (3.31,0.3) and (6.95,1.4) .. (10.93,3.29)   ;
	\draw    (492,396) -- (470,396) ;
	\draw [shift={(468,396)}, rotate = 360] [color={rgb, 255:red, 0; green, 0; blue, 0 }  ][line width=0.75]    (10.93,-3.29) .. controls (6.95,-1.4) and (3.31,-0.3) .. (0,0) .. controls (3.31,0.3) and (6.95,1.4) .. (10.93,3.29)   ;
	\draw    (334,419) -- (455,419) ;
	\draw  [pink,line width=2pt]   (354,419) -- (434,419) ;
	\draw    (354,368) -- (353,419) ;	
	\draw    (414,368) -- (414,419) ;
	\draw    (374,368) -- (374,419) ;
	\draw    (393,419) .. controls (395,389) and (438,392) .. (435,419) ;
	\draw   (438.16,406.69) -- (435,415.96) -- (429.68,407.74) ;
	\draw   (349.28,375.29) -- (354.22,365.51) -- (358.27,375.7) ;
	\draw   (409.28,375.29) -- (414.22,365.51) -- (418.27,375.7) ;
	\draw   (369.28,373.29) -- (374.22,363.51) -- (378.27,373.7) ;
	\draw    (302,398) -- (322,398) ;
	\draw [shift={(324,398.2)}, rotate = 182.6] [color={rgb, 255:red, 0; green, 0; blue, 0 }  ][line width=0.75]    (10.93,-3.29) .. controls (6.95,-1.4) and (3.31,-0.3) .. (0,0) .. controls (3.31,0.3) and (6.95,1.4) .. (10.93,3.29)   ;
	\draw    (302,398) -- (280,398) ;
	\draw [shift={(278,398)}, rotate = 360] [color={rgb, 255:red, 0; green, 0; blue, 0 }  ][line width=0.75]    (10.93,-3.29) .. controls (6.95,-1.4) and (3.31,-0.3) .. (0,0) .. controls (3.31,0.3) and (6.95,1.4) .. (10.93,3.29)   ;		
	\draw    (524,419) -- (645,419) ;
	\draw  [pink,line width=2pt]   (544,419) -- (624,419) ;
	\draw    (544,366) -- (543,419) ;
	\draw    (584.5,366) -- (584.5,419) ;
	\draw    (564,366) -- (564,419) ;
	\draw    (599,419) .. controls (601,388) and (628,392) .. (625,419) ;
	\draw   (628.16,406.69) -- (625,415.96) -- (619.68,407.74) ;
	\draw   (539.28,375.29) -- (544.22,365.51) -- (548.27,375.7) ;
	\draw   (579.28,375.29) -- (584.22,365.51) -- (588.27,375.7) ;
	\draw   (559.28,373.29) -- (564.22,363.51) -- (568.27,373.7) ;
	\draw    (594,275) -- (715,275) ;
	\draw [pink,line width=2pt]    (614,275) -- (694,275) ;
	\draw    (634,223) -- (633,275) ;
	\draw    (654,223) -- (654,275) ;
	\draw    (673,223) -- (673,275) ;	
	\draw   (629.28,232.29) -- (634.22,222.51) -- (638.27,232.7) ;
	\draw   (649.28,231.29) -- (654.22,221.51) -- (658.27,231.7) ;
	\draw   (668.28,231.29) -- (673.22,221.51) -- (677.27,231.7) ;
	\draw    (124,347) -- (136.2,375.16) ;
	\draw [shift={(137,377)}, rotate = 246.57] [color={rgb, 255:red, 0; green, 0; blue, 0 }  ][line width=0.75]    (10.93,-3.29) .. controls (6.95,-1.4) and (3.31,-0.3) .. (0,0) .. controls (3.31,0.3) and (6.95,1.4) .. (10.93,3.29)   ;
	\draw    (124,347) -- (113.75,321.85) ;
	\draw [shift={(113,320)}, rotate = 67.83] [color={rgb, 255:red, 0; green, 0; blue, 0 }  ][line width=0.75]    (10.93,-3.29) .. controls (6.95,-1.4) and (3.31,-0.3) .. (0,0) .. controls (3.31,0.3) and (6.95,1.4) .. (10.93,3.29)   ;	
	\draw    (668,345) -- (684.08,313.78) ;
	\draw [shift={(685,312)}, rotate = 117.26] [color={rgb, 255:red, 0; green, 0; blue, 0 }  ][line width=0.75]    (10.93,-3.29) .. controls (6.95,-1.4) and (3.31,-0.3) .. (0,0) .. controls (3.31,0.3) and (6.95,1.4) .. (10.93,3.29)   ;
	\draw    (668,345) -- (651.87,378.2) ;
	\draw [shift={(651,380)}, rotate = 295.91] [color={rgb, 255:red, 0; green, 0; blue, 0 }  ][line width=0.75]    (10.93,-3.29) .. controls (6.95,-1.4) and (3.31,-0.3) .. (0,0) .. controls (3.31,0.3) and (6.95,1.4) .. (10.93,3.29)   ;
	\draw   (725.27,81.43) .. controls (725.27,76.27) and (729.45,72.09) .. (734.6,72.09) .. controls (739.76,72.09) and (743.94,76.27) .. (743.94,81.43) .. controls (743.94,86.59) and (739.76,90.77) .. (734.6,90.77) .. controls (729.45,90.77) and (725.27,86.59) .. (725.27,81.43) -- cycle ;
	\draw (71,15) node [anchor=north west][inner sep=0.75pt]   [align=left] {$\alpha$};
	\draw (156,181.5) node [anchor=north west][inner sep=0.75pt]   [align=left] {$D$};
	\draw (306,77) node [anchor=north west][inner sep=0.75pt]   [align=left] {$x$};
	\draw (519.75,12.51) node [anchor=north west][inner sep=0.75pt]   [align=left] {$\alpha$};
	\draw (614.83,179.08) node [anchor=north west][inner sep=0.75pt]   [align=left] {$E$};
	\draw (754.17,70.57) node [anchor=north west][inner sep=0.75pt]   [align=left] {$x$};
	\draw (119,296) node [anchor=north west][inner sep=0.75pt]   [align=left] {$G(D)$};
	\draw (275,365) node [anchor=north west][inner sep=0.75pt]   [align=left] {$\Phi_{\text{over}}$};
	\draw (465,365) node [anchor=north west][inner sep=0.75pt]   [align=left] {$\Phi_{\text{over}}$};
	\draw (610,287) node [anchor=north west][inner sep=0.75pt]   [align=left] {$G(D')=G(E)$};
	\draw (70,343) node [anchor=north west][inner sep=0.75pt]   [align=left] {$\Phi_{\text{over}}$};
	\draw (677,347) node [anchor=north west][inner sep=0.75pt]   [align=left] {$\Omega_{1}$};
	\draw (75,245) node [anchor=north west][inner sep=0.75pt]   [align=left] {$c$};
	\draw (600,370) node [anchor=north west][inner sep=0.75pt]   [align=left] {$c_1$};
	
\end{tikzpicture}
\caption{A schematic diagram when $\alpha$ is a closed path.}
\label{Figure 14}
\end{figure}

\underline{\textbf{Case 2.} $\alpha$ is an open path.}  
Firstly, we repeatedly apply $\Phi_{\text{over}}$-moves to contract the chord $c$ corresponding to $x$ into a chord $c_1$, where the starting point of $c_1$ is the first starting point encountered when departing from any endpoint of the Gauss diagram. 
Subsequently, we utilize one $\Phi_+$-move to eliminate $c_1$, thus obtaining a new Gauss diagram $G(D')$.
It can be observed that $G(D') = G(E)$. By virtue of~\cref{Lemma 3.1}, $D'$ can be transformed into $E$ under virtual Reidemeister moves $V\Omega_1 - V\Omega_4$ and $\Omega_v$-moves. Therefore, $D$ can be transformed into $E$ through a finite sequence of moves $\Omega_{1}$, $V\Omega_1 - V\Omega_4$, $\Omega_v$, $\Phi_{\text{over}}$, and $\Phi_+$. The schematic diagram of this proof is shown in~\cref{Figure 15}. 
\begin{figure}[htbp]
\centering
\tikzset{every picture/.style={line width=1.5pt}} 
\begin{tikzpicture}[x=0.75pt,y=0.75pt,yscale=-0.6,xscale=0.6]
	\draw [pink,line width=2pt]  (175.83,98.49) .. controls (187.83,53.49) and (246.83,46.49) .. (287.83,92.49) ;
	\draw [pink,line width=2pt]   (42.83,94.49) .. controls (36,149) and (159.83,159.49) .. (175.83,98.49) ;
	\draw  [pink,line width=2pt]   (282.83,100.49) .. controls (254.83,141.49) and (179.83,145.49) .. (175.83,98.49) ;
	\draw  [pink,line width=2pt] (42.83,94.49) .. controls (43,71) and (62,41) .. (83,41) ;
	\draw  [pink,line width=2pt]  (110,38) .. controls (145,27) and (170,56) .. (175.83,98.49) ;
    \filldraw[color={rgb, 255:red, 0; green, 0; blue, 0 }  ,draw opacity=1] (83,41) circle (2pt);
	\filldraw[color={rgb, 255:red, 0; green, 0; blue, 0 }  ,draw opacity=1] (110,38) circle (2pt);
	\draw    (292.83,86.49) -- (311.83,57.49) ;
	\draw    (238.83,78.49) -- (276.83,36.49) ;
	\draw    (118.83,115.49) -- (128.83,132.49) ;
	\draw    (287.83,92.49) -- (317.83,123.49) ;
	\draw    (120.59,63.62) -- (159.83,17.49) ;
	\draw    (61.83,66.49) -- (84.83,77.49) ;
	\draw    (150.83,167.49) -- (135.83,143.49) ;
	\draw   (241.3,67.12) .. controls (241.25,62.49) and (244.96,58.74) .. (249.59,58.74) .. controls (254.21,58.74) and (258,62.49) .. (258.05,67.12) .. controls (258.1,71.74) and (254.39,75.49) .. (249.77,75.49) .. controls (245.14,75.49) and (241.35,71.74) .. (241.3,67.12) -- cycle ;
	\draw   (54.2,128.74) .. controls (54.15,123.7) and (58.19,119.62) .. (63.23,119.62) .. controls (68.27,119.62) and (72.4,123.7) .. (72.45,128.74) .. controls (72.51,133.78) and (68.47,137.87) .. (63.43,137.87) .. controls (58.39,137.87) and (54.26,133.78) .. (54.2,128.74) -- cycle ;
	\draw   (131.95,40.56) .. controls (131.9,36) and (135.56,32.3) .. (140.12,32.3) .. controls (144.68,32.3) and (148.42,36) .. (148.47,40.56) .. controls (148.52,45.12) and (144.86,48.82) .. (140.3,48.82) .. controls (135.74,48.82) and (132,45.12) .. (131.95,40.56) -- cycle ;
	\draw   (166.57,98.49) .. controls (166.51,93.38) and (170.61,89.23) .. (175.73,89.23) .. controls (180.84,89.23) and (185.04,93.38) .. (185.09,98.49) .. controls (185.15,103.61) and (181.04,107.75) .. (175.93,107.75) .. controls (170.81,107.75) and (166.62,103.61) .. (166.57,98.49) -- cycle ;
	\draw [dashed]  (81.83,89.49) .. controls (81.83,74.03) and (94.36,61.49) .. (109.83,61.49) .. controls (125.29,61.49) and (137.83,74.03) .. (137.83,89.49) .. controls (137.83,104.96) and (125.29,117.49) .. (109.83,117.49) .. controls (94.36,117.49) and (81.83,104.96) .. (81.83,89.49) -- cycle ;
	\draw [dashed]  (205.83,94.49) .. controls (205.83,83.45) and (214.78,74.49) .. (225.83,74.49) .. controls (236.87,74.49) and (245.83,83.45) .. (245.83,94.49) .. controls (245.83,105.54) and (236.87,114.49) .. (225.83,114.49) .. controls (214.78,114.49) and (205.83,105.54) .. (205.83,94.49) -- cycle ;
	\draw    (47.83,59.49) -- (19.83,47.49) ;
	\draw    (239.83,127.49) -- (233.83,113.49) ;
	\draw    (243.83,134.49) -- (252.83,154.49) ;
	\draw    (87.83,106.99) -- (38.83,150.49) ;
	\draw  [pink,line width=2pt]   (622.69,90.68) .. controls (634.45,47.52) and (692.26,40.8) .. (732.43,84.92) ;
	\draw  [pink,line width=2pt]   (492.37,86.84) .. controls (500.2,142.46) and (607.01,149.18) .. (622.69,90.68) ;
	\draw  [pink,line width=2pt]  (727.53,92.59) .. controls (700.1,131.92) and (626.61,135.75) .. (622.69,90.68) ;
	\draw [pink,line width=2pt]   (556.86,30.18) .. controls (583,21) and (616.86,48.18) .. (622.69,90.68) ;
	\draw  [pink,line width=2pt]  (492.37,86.84) .. controls (492.54,63.35) and (511.54,33.35) .. (532.54,33.35) ;
	\filldraw[color={rgb, 255:red, 0; green, 0; blue, 0 }  ,draw opacity=1] (532.54,33.35) circle (2pt);
	\filldraw[color={rgb, 255:red, 0; green, 0; blue, 0 }  ,draw opacity=1] (556.86,30.18) circle (2pt);
	\draw    (727.53,92.59) -- (759.83,55.49) ;
	\draw    (684.42,71.49) -- (721.65,31.21) ;
	\draw    (566.84,106.98) -- (576.63,123.28) ;
	\draw    (732.43,84.92) -- (761.83,114.65) ;
	\draw    (568.56,57.24) -- (607.01,12.99) ;
	\draw    (510.98,59.99) -- (533.52,70.54) ;
	\draw    (598.19,156.85) -- (583.49,133.83) ;
	\draw   (686.84,60.58) .. controls (686.79,56.15) and (690.43,52.55) .. (694.96,52.55) .. controls (699.49,52.55) and (703.21,56.15) .. (703.26,60.58) .. controls (703.31,65.02) and (699.67,68.62) .. (695.14,68.62) .. controls (690.61,68.62) and (686.89,65.02) .. (686.84,60.58) -- cycle ;
	\draw   (503.51,119.69) .. controls (503.46,114.85) and (507.42,110.94) .. (512.36,110.94) .. controls (517.29,110.94) and (521.34,114.85) .. (521.39,119.69) .. controls (521.45,124.52) and (517.49,128.44) .. (512.55,128.44) .. controls (507.61,128.44) and (503.56,124.52) .. (503.51,119.69) -- cycle ;
	\draw   (579.69,35.11) .. controls (579.64,30.74) and (583.23,27.19) .. (587.7,27.19) .. controls (592.17,27.19) and (595.83,30.74) .. (595.88,35.11) .. controls (595.93,39.49) and (592.34,43.03) .. (587.87,43.03) .. controls (583.4,43.03) and (579.74,39.49) .. (579.69,35.11) -- cycle ;
	\draw   (613.61,90.68) .. controls (613.56,85.77) and (617.58,81.79) .. (622.59,81.79) .. controls (627.6,81.79) and (631.71,85.77) .. (631.76,90.68) .. controls (631.82,95.58) and (627.8,99.56) .. (622.79,99.56) .. controls (617.77,99.56) and (613.67,95.58) .. (613.61,90.68) -- cycle ;
	\draw [dashed]  (530.58,82.04) .. controls (530.58,67.21) and (542.86,55.19) .. (558.02,55.19) .. controls (573.17,55.19) and (585.45,67.21) .. (585.45,82.04) .. controls (585.45,96.87) and (573.17,108.9) .. (558.02,108.9) .. controls (542.86,108.9) and (530.58,96.87) .. (530.58,82.04) -- cycle ;
	\draw [dashed]  (652.08,86.84) .. controls (652.08,76.25) and (660.86,67.66) .. (671.68,67.66) .. controls (682.5,67.66) and (691.28,76.25) .. (691.28,86.84) .. controls (691.28,97.43) and (682.5,106.02) .. (671.68,106.02) .. controls (660.86,106.02) and (652.08,97.43) .. (652.08,86.84) -- cycle ;
	\draw    (497.27,53.27) -- (469.83,41.76) ;
	\draw    (685.4,118.49) -- (679.52,105.06) ;
	\draw    (689.32,125.2) -- (698.14,144.38) ;
	\draw    (536.46,98.83) -- (488.45,140.55) ;
	\draw   (723.1,84.92) .. controls (723.1,79.76) and (727.28,75.58) .. (732.43,75.58) .. controls (737.59,75.58) and (741.77,79.76) .. (741.77,84.92) .. controls (741.77,90.08) and (737.59,94.26) .. (732.43,94.26) .. controls (727.28,94.26) and (723.1,90.08) .. (723.1,84.92) -- cycle ;
	\draw    (152,281) -- (236,281) ;
	\draw  [pink,line width=2pt]  (152,281) -- (194,281) ;
	\draw    (82,229) -- (81,281) ;
	\draw    (112,229) -- (112,281) ;
	\draw    (171,230) -- (171,281) ;
	\draw    (28,281) -- (132,281) ;
	\draw  [pink,line width=2pt]   (50,281) -- (132,281) ;
	\filldraw[color={rgb, 255:red, 0; green, 0; blue, 0 }  ,draw opacity=1] (132,281) circle (2pt);
	\filldraw[color={rgb, 255:red, 0; green, 0; blue, 0 }  ,draw opacity=1] (152,281) circle (2pt);
	\draw    (50,281) .. controls (54,244.5) and (186,250.5) .. (194,278.5) ;
	\draw   (193.36,268.52) -- (195.03,278.17) -- (186.42,273.49) ;
	\draw   (77.28,238.29) -- (82.22,228.51) -- (86.27,238.7) ;
	\draw   (107.28,236.29) -- (112.22,226.51) -- (116.27,236.7) ;
	\draw   (166.28,238.29) -- (171.22,228.51) -- (175.27,238.7) ;
	\draw    (209,458) -- (291,458) ;
	\draw [pink,line width=2pt]  (209,458) -- (262,458) ;
	\draw    (110,404.5) -- (109,458) ;
	\draw    (170,407) -- (170,458) ;
	\draw    (78,458) -- (191,458) ;
	\draw  [pink,line width=2pt]   (109,458) -- (191,458) ;
	\draw   (230.5,407.75) -- (230.5,458) ;
	\filldraw[color={rgb, 255:red, 0; green, 0; blue, 0 }  ,draw opacity=1] (191,458) circle (2pt);
	\filldraw[color={rgb, 255:red, 0; green, 0; blue, 0 }  ,draw opacity=1] (209,458) circle (2pt);
	\draw    (139,457) .. controls (143,423) and (254,429) .. (262,457) ;
	\draw   (261.36,447.52) -- (263.03,457.17) -- (254.42,452.49) ;
	\draw   (105.28,413.29) -- (110.22,403.51) -- (114.27,413.7) ;
	\draw   (165.28,416.29) -- (170.22,406.51) -- (174.27,416.7) ;
	\draw   (225.28,416.29) -- (230.22,406.51) -- (234.27,416.7) ;
	\draw    (551,458) -- (633,458) ;
	\draw [pink,line width=2pt]    (551,458) -- (604,458) ;
	\draw    (452,405) -- (451,458) ;
	\draw    (481,406) -- (481,458) ;
	\draw    (572.5,408.25) -- (572.5,458) ;
	\draw    (420,458) -- (533,458) ;
	\draw  [pink,line width=2pt]   (451,458) -- (533,458) ;
	\filldraw[color={rgb, 255:red, 0; green, 0; blue, 0 }  ,draw opacity=1] (551,458) circle (2pt);
	\filldraw[color={rgb, 255:red, 0; green, 0; blue, 0 }  ,draw opacity=1] (533,458) circle (2pt);
	\draw    (511,458) .. controls (515,426) and (596,429.5) .. (604,457.5) ;
	\draw   (603.36,448.02) -- (605.03,457.67) -- (596.42,452.99) ;
	\draw   (447.28,413.79) -- (452.22,404.01) -- (456.27,414.2) ;
	\draw   (476.28,414.79) -- (481.22,405.01) -- (485.27,415.2) ;
	\draw   (567.28,416.79) -- (572.22,407.01) -- (576.27,417.2) ;
	\draw    (642,284) -- (724,284) ;
	\draw  [pink,line width=2pt]  (642,284) -- (692,284) ;
	\draw    (543,229) -- (542,284) ;
	\draw    (572,230) -- (572,284) ;
	\draw    (511,284) -- (624,284) ;
	\draw [pink,line width=2pt]   (542,284) -- (624,284) ;
	\draw    (663.5,232.25) -- (663.5,284) ;
	\filldraw[color={rgb, 255:red, 0; green, 0; blue, 0 }  ,draw opacity=1] (624,284) circle (2pt);
	\filldraw[color={rgb, 255:red, 0; green, 0; blue, 0 }  ,draw opacity=1] (642,284) circle (2pt);
	\draw   (538.28,237.79) -- (543.22,228.01) -- (547.27,238.2) ;
	\draw   (567.28,238.79) -- (572.22,229.01) -- (576.27,239.2) ;
	\draw   (658.28,240.79) -- (663.22,231.01) -- (667.27,241.2) ;
	\draw    (86,376) -- (98.2,404.16) ;
	\draw [shift={(99,406)}, rotate = 246.57] [color={rgb, 255:red, 0; green, 0; blue, 0 }  ][line width=0.75]    (10.93,-3.29) .. controls (6.95,-1.4) and (3.31,-0.3) .. (0,0) .. controls (3.31,0.3) and (6.95,1.4) .. (10.93,3.29)   ;
	\draw    (86,376) -- (75.75,350.85) ;
	\draw [shift={(75,349)}, rotate = 67.83] [color={rgb, 255:red, 0; green, 0; blue, 0 }  ][line width=0.75]    (10.93,-3.29) .. controls (6.95,-1.4) and (3.31,-0.3) .. (0,0) .. controls (3.31,0.3) and (6.95,1.4) .. (10.93,3.29)   ;
	\draw    (630,392) -- (646.08,360.78) ;
	\draw [shift={(647,359)}, rotate = 117.26] [color={rgb, 255:red, 0; green, 0; blue, 0 }  ][line width=0.75]    (10.93,-3.29) .. controls (6.95,-1.4) and (3.31,-0.3) .. (0,0) .. controls (3.31,0.3) and (6.95,1.4) .. (10.93,3.29)   ;
	\draw    (630,392) -- (616.82,421.18) ;
	\draw [shift={(616,423)}, rotate = 294.3] [color={rgb, 255:red, 0; green, 0; blue, 0 }  ][line width=0.75]    (10.93,-3.29) .. controls (6.95,-1.4) and (3.31,-0.3) .. (0,0) .. controls (3.31,0.3) and (6.95,1.4) .. (10.93,3.29)   ;
	\draw    (360,441) -- (398,441) ;
	\draw [shift={(400,441)}, rotate = 180] [color={rgb, 255:red, 0; green, 0; blue, 0 }  ][line width=0.75]    (10.93,-3.29) .. controls (6.95,-1.4) and (3.31,-0.3) .. (0,0) .. controls (3.31,0.3) and (6.95,1.4) .. (10.93,3.29)   ;
	\draw    (360,441) -- (324,441) ;
	\draw [shift={(322,441.2)}, rotate = 358.49] [color={rgb, 255:red, 0; green, 0; blue, 0 }  ][line width=0.75]    (10.93,-3.29) .. controls (6.95,-1.4) and (3.31,-0.3) .. (0,0) .. controls (3.31,0.3) and (6.95,1.4) .. (10.93,3.29)   ;
	\draw (8.83,83.99) node [anchor=north west][inner sep=0.75pt]   [align=left] {$\alpha$};
	\draw (153.83,184.99) node [anchor=north west][inner sep=0.75pt]   [align=left] {$D$};
	\draw (303.83,80.49) node [anchor=north west][inner sep=0.75pt]   [align=left] {$x$};
	\draw (459.58,81) node [anchor=north west][inner sep=0.75pt]   [align=left] {$\alpha$};
	\draw (612.66,182.57) node [anchor=north west][inner sep=0.75pt]   [align=left] {$E$};
	\draw (752,74.06) node [anchor=north west][inner sep=0.75pt]   [align=left] {$x$};
	\draw (110,330) node [anchor=north west][inner sep=0.75pt]   [align=left] {$G(D)$};
	\draw (580,325) node [anchor=north west][inner sep=0.75pt]   [align=left] {$G(D')=G(E)$};
	\draw (70,292.5) node [anchor=north west][inner sep=0.75pt]   [align=left] {head};
	\draw (150,292.5) node [anchor=north west][inner sep=0.75pt]   [align=left] {tail};
	\draw (130,470.5) node [anchor=north west][inner sep=0.75pt]   [align=left] {head};
	\draw (205,470.5) node [anchor=north west][inner sep=0.75pt]   [align=left] {tail};
	\draw (480,472) node [anchor=north west][inner sep=0.75pt]   [align=left] {head};
	\draw (550,472) node [anchor=north west][inner sep=0.75pt]   [align=left] {tail};
	\draw (570,292) node [anchor=north west][inner sep=0.75pt]   [align=left] {head};
	\draw (640,292) node [anchor=north west][inner sep=0.75pt]   [align=left] {tail};	
	\draw (40,8) node [anchor=north west][inner sep=0.75pt]   [align=left] {head};
	\draw (105,7) node [anchor=north west][inner sep=0.75pt]   [align=left] {tail};
	\draw (480,3) node [anchor=north west][inner sep=0.75pt]   [align=left] {head};
	\draw (555,0) node [anchor=north west][inner sep=0.75pt]   [align=left] {tail};
	\draw (30,375) node [anchor=north west][inner sep=0.75pt]   [align=left] {$\Phi_{\text{over}}$};
	\draw (639,394) node [anchor=north west][inner sep=0.75pt]   [align=left] {$\Phi_+$};
	\draw (335,410) node [anchor=north west][inner sep=0.75pt]   [align=left] {$\Phi_{\text{over}}$};
	\draw (42,245) node [anchor=north west][inner sep=0.75pt]   [align=left] {$c$};
	\draw (536,411) node [anchor=north west][inner sep=0.75pt]   [align=left] {$c_1$};
\end{tikzpicture}
\caption{A schematic diagram when $\alpha$ is an open path.}
\label{Figure 15}
\end{figure}
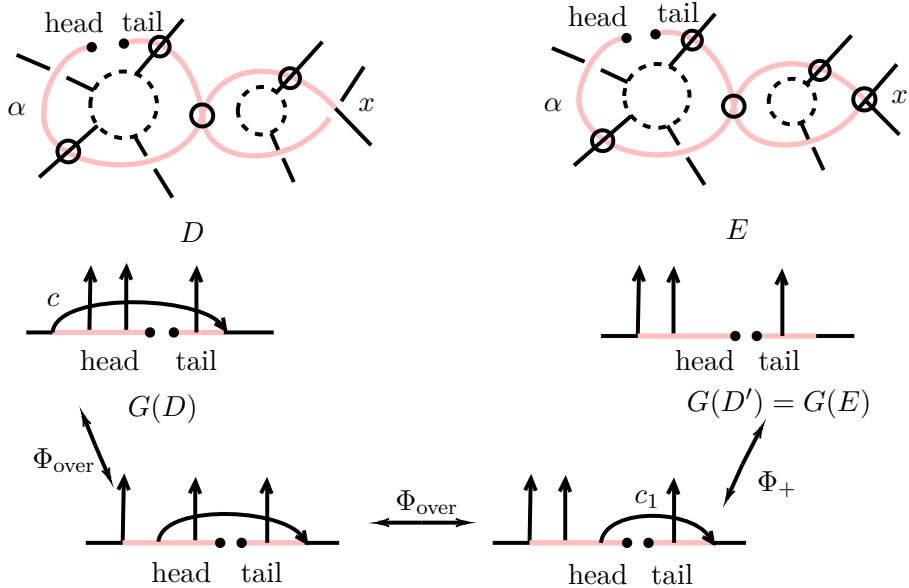

This completes the proof of \cref{Lemma 3.2}.
\end{proof}

Next, we will give the proof of~\cref{Theorem 3.1}.

\begin{proof}[Proof of Theorem 3.1]
Let $x$ be a classical crossing of $D$ which is the first under-crossing encountered while traveling along the orientation of $D$ from the base point.
Since $D$ is a descending diagram, the path starting from the over-arc $\alpha_x$ of $x$ and then reaching the under-arc $\beta_x$ of $x$  will pass through the over-arcs of other classical crossings, but will not pass through the under-arc of any classical crossing.
By Lemma \autoref{Lemma 3.2}, we can replace $x$ with a welded crossing by a finite sequence of moves $\Omega_{1}$, $V\Omega_1 - V\Omega_4$, $\Omega_v$, $\Phi_{\text{over}}$, and $\Phi_+$. Since the diagram obtained through this process remains a descending one, by repeating this process, $D$ can be deformed into the diagram where all crossings are welded. Such a plus-welded knotoid diagram can be transformed into a trivial one by $\Omega_v$-moves. 
This completes the proof of~\cref{Theorem 3.1}.
\end{proof}

\begin{example}
\rm{
Let $D$ be a plus-welded knotoid  diagram as shown in~\cref{Figure 11} (a), then by \cref{Theorem 3.1}, $D$ can be transformed into a trivial knotoid diagram. Indeed, in~\cref{Figure 11} (a), the plus-welded knotoid diagram with base point $a$ is descending. In~\cref{Figure 16}, we use one $\Phi_+$-move at classical crossing $q$ to transform $D_1$ into $D_2$, then use one $V\Omega_1$-move at welded crossing $m$ to transform $D_2$ into $D_3$, and finally use one $\Omega_1$-move at classical crossing $r$ to transform $D_3$ into a diagram $D_4$ of the  trivial knotoid.} 

\vspace{-6.mm}

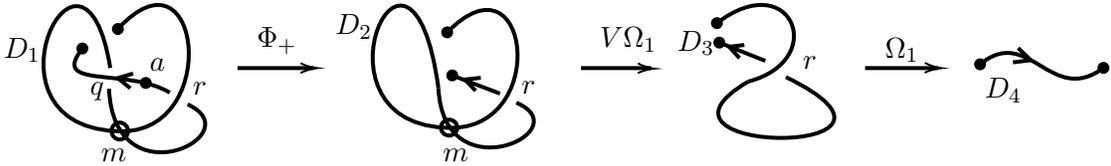
\begin{figure}[htbp]
\centering 
\tikzset{every picture/.style={line width=1.5pt}}    
\begin{tikzpicture}[x=0.75pt,y=0.75pt,yscale=-0.75,xscale=0.75] 
	\draw    (88.67,112.33) .. controls (161.93,118.89) and (148.3,-18.41) .. (91.22,43.68) ;
	\draw    (86.11,69.92) .. controls (68.22,-42.02) and (-6.74,106.64) .. (88.67,112.33) ;
	\draw    (85.26,83.91) .. controls (88.67,160.86) and (182.37,114.51) .. (138.07,94.4) ;
	\draw    (67.37,56.8) .. controls (45.22,84.78) and (97.19,69.04) .. (126.15,86.53) ;
	\draw   (86.94,112.28) .. controls (86.89,115.43) and (89.34,118.02) .. (92.41,118.08) .. controls (95.47,118.13) and (98,115.62) .. (98.06,112.47) .. controls (98.11,109.32) and (95.66,106.72) .. (92.59,106.67) .. controls (89.52,106.61) and (86.99,109.13) .. (86.94,112.28) -- cycle ;
	\draw   (101.49,73.69) -- (91.12,76.72) -- (99.6,83.56) ;
	
	\draw    (171.67,70) -- (223,70) ;
	\draw [shift={(225,70)}, rotate = 179.95] [color={rgb, 255:red, 0; green, 0; blue, 0 }  ][line width=0.75]    (10.93,-3.29) .. controls (6.95,-1.4) and (3.31,-0.3) .. (0,0) .. controls (3.31,0.3) and (6.95,1.4) .. (10.93,3.29)   ;
	\draw    (310.15,110.02) .. controls (382.83,116.14) and (369.31,-12.01) .. (312.68,45.94) ;
	\draw    (306.77,83.49) .. controls (288.17,-53.23) and (215.49,104.72) .. (310.15,110.02) ;
	\draw    (306.77,83.49) .. controls (310.15,155.33) and (403.12,112.06) .. (359.17,93.29) ; 
	\draw   (308.44,109.97) .. controls (308.39,112.91) and (310.81,115.34) .. (313.86,115.39) .. controls (316.9,115.44) and (319.41,113.09) .. (319.46,110.15) .. controls (319.52,107.21) and (317.09,104.79) .. (314.04,104.74) .. controls (311,104.69) and (308.49,107.03) .. (308.44,109.97) -- cycle ;
	\draw   (335.27,76.55) -- (324.63,77.67) -- (331.81,85.33) ; 
	\draw    (316.06,74.92) -- (348.18,87.17) ;

	\draw    (402.67,70) -- (449.78,70) ;
	\draw [shift={(451.78,70.35)}, rotate = 181.21] [color={rgb, 255:red, 0; green, 0; blue, 0 }  ][line width=0.75]    (10.93,-3.29) .. controls (6.95,-1.4) and (3.31,-0.3) .. (0,0) .. controls (3.31,0.3) and (6.95,1.4) .. (10.93,3.29)   ;
	\draw    (515.2,79.03) .. controls (569.71,64.79) and (544.49,1.19) .. (494.05,39.74) ;
	\draw    (515.2,79.03) .. controls (428.97,129.32) and (653.5,130.99) .. (540.42,76.04) ;
	\draw   (512.51,54.43) -- (502.28,55.66) -- (509.28,63.29) ;
	\draw    (495.68,52.84) -- (526.59,65.14) ;

	\draw    (593.67,70) -- (642.79,70) ;
	\draw [shift={(644.79,70)}, rotate = 181.11] [color={rgb, 255:red, 0; green, 0; blue, 0 }  ][line width=0.75]    (10.93,-3.29) .. controls (6.95,-1.4) and (3.31,-0.3) .. (0,0) .. controls (3.31,0.3) and (6.95,1.4) .. (10.93,3.29)   ;
	\draw    (671.01,67.45) .. controls (704.04,40.25) and (718.43,96.25) .. (754,69.85) ;
	\draw   (696.63,56.5) -- (705.07,64.66) -- (693.1,66.15) ;
	\filldraw[color={rgb, 255:red, 0; green, 0; blue, 0 }  ,draw opacity=1] (91.22,43.68) circle (2pt);
	\filldraw[color={rgb, 255:red, 0; green, 0; blue, 0 }  ,draw opacity=1] (67.37,56.8) circle (2pt);
	\filldraw[color={rgb, 255:red, 0; green, 0; blue, 0 }  ,draw opacity=1] (312.68,45.94) circle (2pt);
	\filldraw[color={rgb, 255:red, 0; green, 0; blue, 0 }  ,draw opacity=1] (316.06,74.92) circle (2pt);
	\filldraw[color={rgb, 255:red, 0; green, 0; blue, 0 }  ,draw opacity=1] (494.05,39.74) circle (2pt);
	\filldraw[color={rgb, 255:red, 0; green, 0; blue, 0 }  ,draw opacity=1] (495.68,52.84) circle (2pt);
	\filldraw[color={rgb, 255:red, 0; green, 0; blue, 0 }  ,draw opacity=1] (671.01,67.45) circle (2pt);
	\filldraw[color={rgb, 255:red, 0; green, 0; blue, 0 }  ,draw opacity=1] (754,69.85) circle (2pt);
	\filldraw[color={rgb, 255:red, 0; green, 0; blue, 0 }  ,draw opacity=1] (110,80) circle (2pt);
	
	\draw (69.96,77) node [anchor=north west][inner sep=0.75pt]   [align=left] {$q$};
	\draw (138.96,76.72) node [anchor=north west][inner sep=0.75pt]   [align=left] {$r$};
	\draw (110,60.98) node [anchor=north west][inner sep=0.75pt]   [align=left] {$a$};
	\draw (182.09,40) node [anchor=north west][inner sep=0.75pt]   [align=left] {$\Phi_+$};
	\draw (11.89,46.99) node [anchor=north west][inner sep=0.75pt]   [align=left] {$D_1$};
	\draw (360.01,76.22) node [anchor=north west][inner sep=0.75pt]   [align=left] {$r$};
	\draw (412.72,40) node [anchor=north west][inner sep=0.75pt]   [align=left] {$V\Omega_1$};
	\draw (549.16,59.23) node [anchor=north west][inner sep=0.75pt]   [align=left] {$r$};
	\draw (604.66,46.9) node [anchor=north west][inner sep=0.75pt]   [align=left] {$\Omega_1$};
	\draw (235,31.32) node [anchor=north west][inner sep=0.75pt]   [align=left] {$D_2$};
	\draw (465,42.47) node [anchor=north west][inner sep=0.75pt]   [align=left] {$D_3$};
	\draw (77.33,121.32) node [anchor=north west][inner sep=0.75pt]   [align=left] {$m$};
	\draw (307.3,121.93) node [anchor=north west][inner sep=0.75pt]   [align=left] {$m$};
	\draw (671.71,74.1) node [anchor=north west][inner sep=0.75pt]   [align=left] {$D_4$};
\end{tikzpicture}
\caption{Transforming a descending diagram into a trivial knotoid diagram.}
\label{Figure 16}
\end{figure}
\end{example}

\section{Warping degrees for plus-welded knotoids} \label{section4}

In this section, we generalize the definition and properties of the warping degree to plus-welded knotoids analogously to~\citep{shimizu2010warping}.

Let $D$ be a diagram of the plus-welded virtual knotoid $K$ with a base point $a$. Starting from the base point $a$ along the orientation $D$, if we encounter a classical crossing first at the under-crossing, then this classical crossing is called a \textit{warping crossing} of $D_a$. 
For example, the diagram $D'$ with the base point $a$, the classical crossings $q$ and $r$ are warping crossings of $D'$, as shown in~\cref{Figure 11} (b).

\begin{definition} \label{Definition 4.1} 
\rm{ 
Let $D_{a}$ be a diagram $D$ of the plus-welded knotoid $K$ with a base point $a$.
\begin{itemize}
\item The \textit{warping degree of $D_a$}, denoted by $d(D_a)$, is the number of warping crossings of $D_a$.
\item The \textit{warping degree of $D$}, denoted by $d(D)$, is the minimal warping degree for all base points of $D$, namely  
$$
d(D) = \min_{a} \{d(D_a) \mid \text{$a$ is a base point of $D$}\}.
$$
\item The \textit{warping degree of $K$}, denoted by $d(K)$, is the minimal warping degree of $D$ and $-D$, where $D$ is a diagram of $K$ and $-D$ is the inverse of $D$, namely 
$$ 
d(K) = \min_{D} \{ d(D), d(-D) \mid \text{$D$ is a diagram} \text{ of $K$}\}. 
$$
\end{itemize}
}
\end{definition}

 By~\cref{Definition 4.1}, the warping degree of plus-welded  knotoids is independent of the orientation of diagrams.

It is easy to see that the positions of the base points in $D$ affect the calculation of warping degree of $D$. Remark that, if base points $a$ and $b$ are such that there is no classical crossings on the path from $a$ to $b$, then $a$ and $b$ have the same type, that  means $d(D_b) = d(D_a)$.
If $D$ be the diagram with $n$ classical crossings, then there are  $2n$ types of base points. 

Next, we will give an example of calculating warping degree. 

\begin{example} \label{Example 4.1} 
\rm{
Let $D$ be a diagram of plus-welded knotoid $K$ as shown in~\cref{Figure 17} (a), the inverse $-D$ of $D$ is shown in~\cref{Figure 17} (b).

\vspace{-6.mm}

\begin{figure}[htbp]
\centering
\tikzset{every picture/.style={line width=1.5pt}} 
\begin{tikzpicture}[x=0.75pt,y=0.75pt,yscale=-0.75,xscale=0.75]
	\draw    (110,132.5) .. controls (153,135) and (164,116) .. (163,80) ;
	\draw    (107,84) .. controls (86,-44) and (-2,126) .. (110,132.5) ;
	\draw    (106,100) .. controls (110,188) and (220,135) .. (168,112) ;
	\draw    (88,59) .. controls (52,97) and (120,83) .. (154,103) ;
	\draw   (107.98,132.44) .. controls (107.91,136.04) and (110.79,139.01) .. (114.39,139.07) .. controls (117.99,139.14) and (120.96,136.26) .. (121.02,132.66) .. controls (121.08,129.06) and (118.21,126.09) .. (114.61,126.03) .. controls (111.01,125.97) and (108.04,128.84) .. (107.98,132.44) -- cycle ;
	\draw   (125.16,88.47) -- (112.9,91.65) -- (122.67,99.71) ;
	\draw    (106,48) .. controls (136,18) and (154,37) .. (160,55) ;
	\draw    (361,135.5) .. controls (404,138) and (415,119) .. (414,83) ;
	\draw    (358,87) .. controls (337,-41) and (249,129) .. (361,135.5) ;
	\draw    (357,103) .. controls (361,191) and (471,138) .. (419,115) ;
	\draw    (339,62) .. controls (303,100) and (371,86) .. (405,106) ;
	\draw   (358.98,135.44) .. controls (358.91,139.04) and (361.79,142.01) .. (365.39,142.07) .. controls (368.99,142.14) and (371.96,139.26) .. (372.02,135.66) .. controls (372.08,132.06) and (369.21,129.09) .. (365.61,129.03) .. controls (362.01,128.97) and (359.04,131.84) .. (358.98,135.44) -- cycle ;
	\draw   (363.1,99.88) -- (375.01,95.57) -- (364.53,88.46) ;
	\draw    (357,51) .. controls (387,21) and (405,40) .. (411,58) ;
	\filldraw[color={rgb, 255:red, 0; green, 0; blue, 0 }  ,draw opacity=1] (160,55) circle (2pt);
	\filldraw[color={rgb, 255:red, 0; green, 0; blue, 0 }  ,draw opacity=1] (163,80) circle (2pt);
	\filldraw[color={rgb, 255:red, 0; green, 0; blue, 0 }  ,draw opacity=1] (411,58) circle (2pt);
	\filldraw[color={rgb, 255:red, 0; green, 0; blue, 0 }  ,draw opacity=1] (414,83) circle (2pt);
	\filldraw[color={rgb, 255:red, 0; green, 0; blue, 0 }  ,draw opacity=1] (136,95) circle (2pt);
	\filldraw[color={rgb, 255:red, 0; green, 0; blue, 0 }  ,draw opacity=1] (153,121) circle (2pt);
	\filldraw[color={rgb, 255:red, 0; green, 0; blue, 0 }  ,draw opacity=1] (78,78) circle (2pt);
	\filldraw[color={rgb, 255:red, 0; green, 0; blue, 0 }  ,draw opacity=1] (140,34) circle (2pt);
	\filldraw[color={rgb, 255:red, 0; green, 0; blue, 0 }  ,draw opacity=1] (103,65) circle (2pt);
	\filldraw[color={rgb, 255:red, 0; green, 0; blue, 0 }  ,draw opacity=1] (156,148) circle (2pt);
	\filldraw[color={rgb, 255:red, 0; green, 0; blue, 0 }  ,draw opacity=1] (387,99) circle (2pt);
	\filldraw[color={rgb, 255:red, 0; green, 0; blue, 0 }  ,draw opacity=1] (330,80) circle (2pt);
	\filldraw[color={rgb, 255:red, 0; green, 0; blue, 0 }  ,draw opacity=1] (395,37) circle (2pt);
	\filldraw[color={rgb, 255:red, 0; green, 0; blue, 0 }  ,draw opacity=1] (401,127) circle (2pt);
	\filldraw[color={rgb, 255:red, 0; green, 0; blue, 0 }  ,draw opacity=1] (353,65) circle (2pt);
	\filldraw[color={rgb, 255:red, 0; green, 0; blue, 0 }  ,draw opacity=1] (411,150) circle (2pt);	
	\draw (48,154) node [anchor=north west][inner sep=0.75pt]   [align=left] {$D$};
	\draw (309,155) node [anchor=north west][inner sep=0.75pt]   [align=left] {$-D$};
	\draw (136,80) node [anchor=north west][inner sep=0.75pt]   [align=left] {$a$};
	\draw (60,69) node [anchor=north west][inner sep=0.75pt]   [align=left] {$b$};
	\draw (140,18) node [anchor=north west][inner sep=0.75pt]   [align=left] {$c$};
	\draw (153,121) node [anchor=north west][inner sep=0.75pt]   [align=left] {$d$};
	\draw (112,58) node [anchor=north west][inner sep=0.75pt]   [align=left] {$e$};
	\draw (149,155) node [anchor=north west][inner sep=0.75pt]   [align=left] {$f$};
	\draw (387,83) node [anchor=north west][inner sep=0.75pt]   [align=left] {$a$};
	\draw (310,73) node [anchor=north west][inner sep=0.75pt]   [align=left] {$b$};
	\draw (395,20) node [anchor=north west][inner sep=0.75pt]   [align=left] {$c$};
	\draw (401,127) node [anchor=north west][inner sep=0.75pt]   [align=left] {$d$};
	\draw (365,60) node [anchor=north west][inner sep=0.75pt]   [align=left] {$e$};
	\draw (411,155) node [anchor=north west][inner sep=0.75pt]   [align=left] {$f$};
	\draw (103,172) node [anchor=north west][inner sep=0.75pt]   [align=left] {(a)};
	\draw (358,172) node [anchor=north west][inner sep=0.75pt]   [align=left] {(b)};
\end{tikzpicture}
\caption{Two knotoid diagrams $D$ and $-D$ with base points $a$, $b$, $c$, $d$, $e$, $f$.}
\label{Figure 17}
\end{figure}
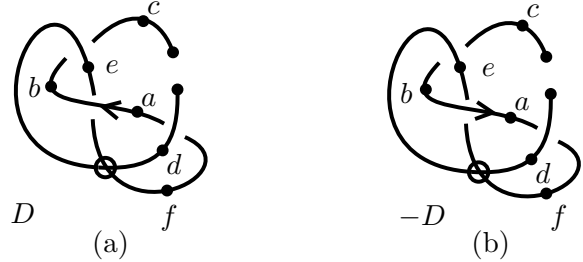 

In~\cref{Figure 17} (a) and (b) there are presented six possible base points $a$, $b$, $c$, $d$, $e$, and $f$ for diagrams $D$ and $-D$. For the diagram $D$, we get
$$
d(D_a)=1, \quad d(D_b)=2, \quad d(D_c)=1, \quad d(D_d)=2, \quad d(D_e)=3, \quad d(D_f)=2,  
$$
and for the diagram $-D$, we get
$$
d(-D_a)=2, \quad d(-D_b)=1, \quad(-D_c)=2, \quad d(-D_d)=1, \quad d(-D_e)=0, \quad d(-D_f)=1.
$$
Thus, the warping degrees of $D$ and $-D$ are different, $d(D)=1$, $d(-D)=0$.  According to~\cref{Definition 4.1}, we have $d(K)=0$. 
} 
\end{example}

\begin{definition}
{\rm 
Let $D$ be a plus-welded knotoid diagram, if there is a base point $a$ on $D$ such that $d(D_a) = 0$, then $D$ is referred to as \textit{monotone}.
}
\end{definition}

It is easy to see that a plus-welded  knotoid  diagram $D$ is monotone if and only if it is descending. 
By~\cref{Theorem 3.1}, we have that any monotone diagram $D$ of a plus-welded knotoid can be transformed into a diagram of trivial knotoid by a finite  sequence of moves $\Omega_{1}$, $V\Omega_1 - V\Omega_4$, $\Omega_v$, $\Phi_{\text{over}}$, and $\Phi_+ $. Thus, we obtain: 

\begin{corollary} \label{Corollary 4.1}
Any monotone diagram $D$ of plus-welded knotoid is plus-welded equivalent to a diagram of trivial knotoid. 
\end{corollary}

\begin{lemma} \label{Lemma 4.1}
Let $D_a$ be a plus-welded knotoid diagram with a base point $a$, then we have $d(D_a) + d(-D_a) = \operatorname{cr}(D) $, where $\operatorname{cr}(D)$ is the classical crossing number of $D$. 
\end{lemma}

In~\citep{shimizu2010warping}, Shimizu introduced a method to determine whether a classical crossing in a knot diagram is a warping crossing or not. Before proving~\cref{Lemma 4.1}, we first generalize this method to plus-welded  knotoids. Let $D$ be a diagram of the plus-welded knotoid $K$ and $a$ the base point of $D$. We notice that $D_a$ is divided into $\operatorname{cr}(D)+ 2$ arcs by cutting $D$ at the base point and under-crossing points. Then, we label these arcs according to the orientation of $D$. 

\begin{definition} \label{Definition 4.2}
Let $D_a$ be a plus-welded knotoid diagram with a base point $a$, and let $p$ be a classical crossing on $D_a$. It consists of one over-arc labeled $\alpha$ and two under-arcs labeled $\beta$ and $\gamma$, as shown in \cref{Figure 18}. We define the cutting number of $p$ in $D_a$, denoted as $\operatorname{cut}_{D_a}(p)$, by the following formula: $\operatorname{cut}_{D_a}(p)=2\alpha - \beta - \gamma$.
\end{definition} 

\begin{figure}[htbp]
\begin{center}
	\tikzset{every picture/.style={line width=1.5pt}} 
    \begin{tikzpicture}[x=0.75pt,y=0.75pt,yscale=-1,xscale=1]
    \draw    (190,2) -- (190,101) ;
    \draw    (129,51) -- (180,51) ;
    \draw    (199,51) -- (251,51) ;
    \draw (173,3) node [anchor=north west][inner sep=0.75pt]   [align=left] {$\alpha$};
    \draw (131,54) node [anchor=north west][inner sep=0.75pt]   [align=left] {$\beta$};
    \draw (239,54) node [anchor=north west][inner sep=0.75pt]   [align=left] {$\gamma$};
    \draw (169,84) node [anchor=north west][inner sep=0.75pt]   [align=left] {$\alpha$};
    \draw (199,27) node [anchor=north west][inner sep=0.75pt]   [align=left] {$p$};			
\end{tikzpicture}
\caption{A classical crossing $p$ with labeling.}
\label{Figure 18}
\end{center}
\end{figure}
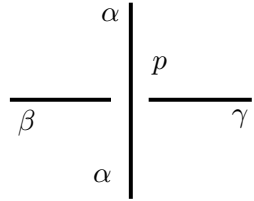
It should be noted that the cutting number is always odd. For the convenience of discussion, we may assume that $\beta$ is less than $\gamma$, namely, $\gamma=\beta + 1$. Then we have
$\operatorname{cut}_{D_a}(p)=2(\alpha-\beta)- 1$, which is odd. According to the definition of the warping crossing, the following Lemma holds immediately:

\begin{lemma} \label{Lemma 4.2}
Let $a$ be a base point of a plus-welded knotoid diagram $D$ and $p$ be any classical crossing on $D_a$. Then we have: 
\begin{itemize}	
\item[(1)] A classical crossing $p$ is a warping crossing of $D_a$ if and only if $\operatorname{cut}_{D_a}(p)> 0$. 
\item[(2)] A classical crossing $p$ is a non-warping crossing of $D_a$ if and only if $\operatorname{cut}_{D_a}(p)< 0$.
\end{itemize}
\end{lemma} 

Next, we will prove~\cref{Lemma 4.1} by utilizing the properties of the cutting number. 

\begin{proof}[Proof of Lemma 4.1]
 Let the ordered set of arcs of $D_a$ be $(k_1, k_2, \ldots, k_n)$, where $k_i$ represents the $i$-th arc of $D_a$ and $n=\operatorname{cr}(D)+2$. Let the ordered set of arcs of $-D_a$ be $(l_1, l_2, \ldots, l_n)$, where $l_i$ represents the $i$-th arc of $-D_a$ and $n=\operatorname{cr}(D)+2$. Then we observe that for $i = 1, 2, \ldots, n$, the arc $k_i$ and the arc $l_{n+1-i}$ are the same arc. Then we have
$$
\operatorname{cut}_{-D_a}(p) = 2(n + 1 - \alpha) - (n + 1 - \beta) - (n + 1 - \gamma) =-(2\alpha - \beta - \gamma) =-\operatorname{cut}_{D_a}(p).
$$
By~\cref{Lemma 4.2}, $p$ is a warping crossing of $D_a$ if and only if $p$ is a non-warping crossing of $-D_a$.
Hence, we obtain
\begin{eqnarray*}
\operatorname{cr}(D) &= & d (D_a)+\sharp(\text{non-warping crossings of } D_a) \\
& = & d(D_a)+\sharp(\text{warping crossings of } -D_a) \\ 
& = & d(D_a)+d(-D_a).
\end{eqnarray*}
The lemma is proved. 
\end{proof}

By applying~\cref{Lemma 4.1} to the mirror image of $D$, we obtain:

\begin{corollary} \label{Corollary 4.2}
Let $D$ be a plus-welded  knotoid diagram, and let $D^*$ be the mirror image of $D$. Then $d(D^*) = d(-D)$.
\end{corollary}

\begin{proof}	
 Suppose $D_a$ is a plus-welded knotoid diagram with a base point $a$. Let $D_a^*$ be the mirror image of $D_a$, and let $p$ be any classical crossing on $D_a$ and $D_a^*$, as shown in~\cref{Figure 19}. 
 \begin{figure}[htbp]
\begin{center}
\tikzset{every picture/.style={line width=1.5pt}} 
\begin{tikzpicture}[x=0.75pt,y=0.75pt,yscale=-1,xscale=1]
	\draw    (119,11) -- (119,113) ;
	\draw    (127,61) -- (179,61) ;
	\draw    (59,61) -- (110,61) ;
	\draw    (240,61) -- (341,61) ;
	\draw    (291,109) -- (291,67) ;
	\draw    (291,53) -- (291,12) ;
	\draw (128,39) node [anchor=north west][inner sep=0.75pt]   [align=left] {$p$};
	\draw (302,40) node [anchor=north west][inner sep=0.75pt]   [align=left] {$p$};
	\draw (52,86) node [anchor=north west][inner sep=0.75pt]   [align=left] {$D_a$};
	\draw (239,82) node [anchor=north west][inner sep=0.75pt]   [align=left] {$D_a^*$};
\end{tikzpicture}
\caption{A classical crossing $p$ in $D_a$ and $D_a^*$.}
\label{Figure 19}
\end{center}
\end{figure}
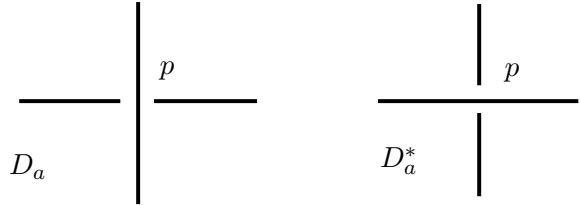
 
 Depending on the differences between the classical crossings of $D_a$ and $D_a^*$, we find that if we encounter the under-arc of the classical crossing $p$ for the first time while starting from base point $a$ and traversing along the orientation of $D_a$, then we encounter the over-arc of the classical crossing $p$ for the first time while starting from base point $a$ and traversing along the orientation of $D_a^*$. Thus, we obtain that a classical crossing $p$ is a warping crossing in $D_a^*$ if and only if $p$ is a non-warping crossing in $D_a$. By~\cref{Lemma 4.1}, we obtain that $p$ is a warping crossing in $D_a^*$ if and only if $p$ is a warping crossing in $-D_a$. Therefore, $d(D^*) = d(-D)$.
\end{proof}
		
Let $D$ be a plus-welded virtual knotoid diagram, for two base points on $D$ which traverses across a classical crossing, we have: 

\begin{lemma} \label{Lemma 4.3} 
Consider the base points $a_1$, $a_2$ on $D$ such that the path from $a_1$ to $a_2$ traverses a classical crossing. 
\begin{itemize}
\item[(1)] If the path traverses across an over-crossing, then regardless of whether the path traverses through the head and the tail, we have:
\begin{equation*}
\label{equation 4.1}	
 d(D_{a_{2}}) = d(D_{a_1}) + 1.
\end{equation*} 
\item[(2)] If the path traverses across an under-crossing, then regardless of whether the path traverses through the head and the tail, we have
\begin{equation*}
 d(D_{a_2}) = d(D_{a_1}) - 1.
\end{equation*}
\end{itemize} 
\end{lemma}

\begin{proof}
$(1)$ The path from $a_1$ to $a_2$ which traverses across an over-crossing $p$ is shown in~\cref{Figure 20} (a),(b). 
If $D$ has $a_1$ as its base point, and when starting from $a_1$ then $p$ is a non-warping crossing of $D_{a_1}$. If $D$ has $a_2$ as its base point, then $p$ is a warping crossing of $D_{a_2}$. The warping crossings and non-warping crossings of $D_{a_1}$ and $D_{a_2}$ are all the same except for $p$. Therefore, $D_{a_2}$ has one more warping crossing than $D_{a_1}$, and thus the case (1) holds. 
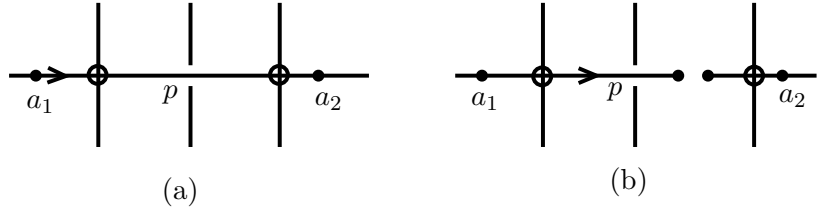
\begin{figure}[htbp]
\begin{center}
\tikzset{every picture/.style={line width=1.5pt}} 
\begin{tikzpicture}[x=0.75pt,y=0.75pt,yscale=-0.75,xscale=0.75]
	\draw    (80,61) -- (322,61) ;
	\draw    (202,109) -- (202,68) ;
	\draw    (202,54) -- (202,12) ;
	\draw    (140,12) -- (140,109) ;
	\draw    (262,12) -- (262,109) ;
	\draw   (133.5,60.5) .. controls (133.5,57.19) and (136.19,54.5) .. (139.5,54.5) .. controls (142.81,54.5) and (145.5,57.19) .. (145.5,60.5) .. controls (145.5,63.81) and (142.81,66.5) .. (139.5,66.5) .. controls (136.19,66.5) and (133.5,63.81) .. (133.5,60.5) -- cycle ;
	\draw   (255.5,60.5) .. controls (255.5,57.19) and (258.19,54.5) .. (261.5,54.5) .. controls (264.81,54.5) and (267.5,57.19) .. (267.5,60.5) .. controls (267.5,63.81) and (264.81,66.5) .. (261.5,66.5) .. controls (258.19,66.5) and (255.5,63.81) .. (255.5,60.5) -- cycle ;
	\draw   (106.27,55.69) -- (117.99,61.32) -- (105.74,65.68) ;
	\filldraw[color={rgb, 255:red, 0; green, 0; blue, 0 }  ,draw opacity=1] (98,61) circle (2pt);
	\filldraw[color={rgb, 255:red, 0; green, 0; blue, 0 }  ,draw opacity=1] (288,61) circle (2pt);
	\draw    (380,61) -- (530,61) ;
	\draw    (501,109) -- (501,68) ;
	\draw    (501,54) -- (501,12) ;
	\draw    (439,12) -- (439,109) ;
	\draw    (581,12) -- (581,109) ;
	\draw   (432.5,61.5) .. controls (432.5,58.19) and (435.19,55.5) .. (438.5,55.5) .. controls (441.81,55.5) and (444.5,58.19) .. (444.5,61.5) .. controls (444.5,64.81) and (441.81,67.5) .. (438.5,67.5) .. controls (435.19,67.5) and (432.5,64.81) .. (432.5,61.5) -- cycle ;
	\draw   (575,61) .. controls (575,57.69) and (577.69,55) .. (581,55) .. controls (584.31,55) and (587,57.69) .. (587,61) .. controls (587,64.31) and (584.31,67) .. (581,67) .. controls (577.69,67) and (575,64.31) .. (575,61) -- cycle ;
	\draw   (463,56) -- (475,61) -- (463,66) ;
	\draw    (550,61) -- (623,61) ;
	\filldraw[color={rgb, 255:red, 0; green, 0; blue, 0 }  ,draw opacity=1] (398,61) circle (2pt);
	\filldraw[color={rgb, 255:red, 0; green, 0; blue, 0 }  ,draw opacity=1] (600,61) circle (2pt);
	\filldraw[color={rgb, 255:red, 0; green, 0; blue, 0 }  ,draw opacity=1] (530,61) circle (2pt);
	\filldraw[color={rgb, 255:red, 0; green, 0; blue, 0 }  ,draw opacity=1] ((550,61) circle (2pt);
	\draw (181,65) node [anchor=north west][inner sep=0.75pt]   [align=left] {\(p\)};
	\draw (89,72) node [anchor=north west][inner sep=0.75pt]   [align=left] {\(a_1\)};
	\draw (283,70) node [anchor=north west][inner sep=0.75pt]   [align=left] {\(a_2\)};
	\draw (480,63) node [anchor=north west][inner sep=0.75pt]   [align=left] {\(p\)};
	\draw (388,70) node [anchor=north west][inner sep=0.75pt]   [align=left] {\(a_1\)};
	\draw (595,68) node [anchor=north west][inner sep=0.75pt]   [align=left] {\(a_2\)};
	\draw (180,127) node [anchor=north west][inner sep=0.75pt]   [align=left] {(a)};
	\draw (481,120) node [anchor=north west][inner sep=0.75pt]   [align=left] {(b)};
\end{tikzpicture}
\caption{The path from $a_1$ to $a_2$ which traverses across an over-crossing $p$.}
\label{Figure 20}
\end{center}
\end{figure}
	
$(2)$ The path from $a_1$ to $a_2$ which traverses across an under-crossing $p$ is shown in~\cref{Figure 21} (a),(b). The proof in this case is analogous to the case  $(1)$.

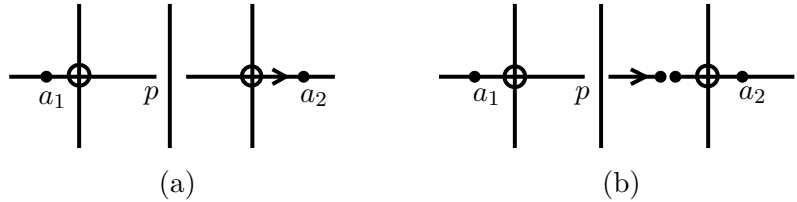
\begin{figure}[htbp]
\begin{center}
\tikzset{every picture/.style={line width=1.5pt}} 
\begin{tikzpicture}[x=0.75pt,y=0.75pt,yscale=-0.75,xscale=0.75]
	\draw    (170,12) -- (170,109) ;
	\draw    (62,61) -- (161,61) ;
	\draw    (181,61) -- (281,61) ;
	\draw    (109,12) -- (109,109) ;
	\draw    (225.5,12) -- (225.5,109) ;
	\draw   (239,56) -- (248,61) -- (239,66) ;
	\draw   (102,60) .. controls (102,56.13) and (105.13,53) .. (109,53) .. controls (112.87,53) and (116,56.13) .. (116,60) .. controls (116,63.87) and (112.87,67) .. (109,67) .. controls (105.13,67) and (102,63.87) .. (102,60) -- cycle ;
	\draw   (218.5,60.5) .. controls (218.5,56.91) and (221.41,54) .. (225,54) .. controls (228.59,54) and (231.5,56.91) .. (231.5,60.5) .. controls (231.5,64.09) and (228.59,67) .. (225,67) .. controls (221.41,67) and (218.5,64.09) .. (218.5,60.5) -- cycle ;
	\filldraw[color={rgb, 255:red, 0; green, 0; blue, 0 }  ,draw opacity=1] (87,61) circle (2pt);
	\filldraw[color={rgb, 255:red, 0; green, 0; blue, 0 }  ,draw opacity=1] (260,61) circle (2pt);
	\draw    (460,12) -- (460,109) ;
	\draw    (351,61) -- (449,61) ;
	\draw    (465,61) -- (500,61) ;
	\draw    (510,61) -- (590,61) ;
	\draw    (402,12) -- (402,109) ;
	\draw    (532,12) -- (532,109) ;
	\draw   (480,56) -- (490,61) -- (480,66) ;
	\draw   (395,61) .. controls (395,57.13) and (398.13,54) .. (402,54) .. controls (405.87,54) and (409,57.13) .. (409,61) .. controls (409,64.87) and (405.87,68) .. (402,68) .. controls (398.13,68) and (395,64.87) .. (395,61) -- cycle ;
	\draw   (524.5,60.5) .. controls (524.5,56.63) and (527.63,53.5) .. (531.5,53.5) .. controls (535.37,53.5) and (538.5,56.63) .. (538.5,60.5) .. controls (538.5,64.37) and (535.37,67.5) .. (531.5,67.5) .. controls (527.63,67.5) and (524.5,64.37) .. (524.5,60.5) -- cycle ;
	\filldraw[color={rgb, 255:red, 0; green, 0; blue, 0 }  ,draw opacity=1] (375,61) circle (2pt);
	\filldraw[color={rgb, 255:red, 0; green, 0; blue, 0 }  ,draw opacity=1] (555,61) circle (2pt);
	\filldraw[color={rgb, 255:red, 0; green, 0; blue, 0 }  ,draw opacity=1] (500,61) circle (2pt);
	\filldraw[color={rgb, 255:red, 0; green, 0; blue, 0 }  ,draw opacity=1] (510,61) circle (2pt);
	\draw (150,65) node [anchor=north west][inner sep=0.75pt]   [align=left] {$p$};
	\draw (79,68) node [anchor=north west][inner sep=0.75pt]   [align=left] {$a_1$};
	\draw (255,67) node [anchor=north west][inner sep=0.75pt]   [align=left] {$a_2$};
	\draw (440,65) node [anchor=north west][inner sep=0.75pt]   [align=left] {$p$};
	\draw (371,66) node [anchor=north west][inner sep=0.75pt]   [align=left] {$a_1$};
	\draw (550,65) node [anchor=north west][inner sep=0.75pt]   [align=left] {$a_2$};
	\draw (160,121) node [anchor=north west][inner sep=0.75pt]   [align=left] {(a)};
	\draw (458,121) node [anchor=north west][inner sep=0.75pt]   [align=left] {(b)};
\end{tikzpicture}
\caption{The path from $a_1$ to $a_2$ which traverses across an under-crossing $p$.}
\label{Figure 21}
\end{center}
\end{figure}	 
This completes the proof of \cref{Lemma 4.3}.
\end{proof}		

Let $D$ be a diagram of plus-welded  knotoid, if the over-crossing points and under-crossing points of classical crossings of $D$ alternate with each other, then $D$ is referred to as an \textit{alternating diagram}. 
By~\cref{Lemma 4.3}, we have: 

\begin{lemma} \label{Lemma 4.4}
Let $D$ be an alternating plus-welded knotoid diagram. Let $a$ be a base point of $D$, which is just before an over-crossing as shown in~\cref{Figure 22} (a), (b). Then we have $d(D_a)=d(D)$. 
\end{lemma}

\begin{figure}[htbp]
\begin{center}
\tikzset{every picture/.style={line width=1.5pt}} 
\begin{tikzpicture}[x=0.75pt,y=0.75pt,yscale=-0.75,xscale=0.75]	
	\draw    (315,61) -- (350,61) ;
	\draw    (203,112) -- (203,73) ;
	\draw    (203,52) -- (203,12) ;
	\draw    (255,12) -- (255,109) ;
	\draw   (631.25,60) .. controls (631.25,56) and (634.5,52.75) .. (638.5,52.75) .. controls (642.5,52.75) and (645.75,56) .. (645.75,60) .. controls (645.75,64) and (642.5,67.25) .. (638.5,67.25) .. controls (634.5,67.25) and (631.25,64) .. (631.25,60) -- cycle ;
	\draw    (105,12) -- (105,109) ;
	\draw    (24,61) -- (96,61) ;
	\draw    (119,61) -- (294,61) ;
	\draw    (63,12) -- (63,109) ;
	\draw    (153.5,12) -- (153.5,109) ;
	\draw   (78.95,56.12) -- (86.34,61.31) -- (78.98,66.55) ;
	\draw   (515.75,61.5) .. controls (515.75,57.5) and (519,54.25) .. (523,54.25) .. controls (527,54.25) and (530.25,57.5) .. (530.25,61.5) .. controls (530.25,65.5) and (527,68.75) .. (523,68.75) .. controls (519,68.75) and (515.75,65.5) .. (515.75,61.5) -- cycle ;
	\draw   (424.75,62) .. controls (424.75,58) and (428,54.75) .. (432,54.75) .. controls (436,54.75) and (439.25,58) .. (439.25,62) .. controls (439.25,66) and (436,69.25) .. (432,69.25) .. controls (428,69.25) and (424.75,66) .. (424.75,62) -- cycle ;
	\draw   (248.25,60.5) .. controls (248.25,56.5) and (251.5,53.25) .. (255.5,53.25) .. controls (259.5,53.25) and (262.75,56.5) .. (262.75,60.5) .. controls (262.75,64.5) and (259.5,67.75) .. (255.5,67.75) .. controls (251.5,67.75) and (248.25,64.5) .. (248.25,60.5) -- cycle ;
	\draw   (146.75,61.5) .. controls (146.75,57.5) and (150,54.25) .. (154,54.25) .. controls (158,54.25) and (161.25,57.5) .. (161.25,61.5) .. controls (161.25,65.5) and (158,68.75) .. (154,68.75) .. controls (150,68.75) and (146.75,65.5) .. (146.75,61.5) -- cycle ;
	\draw   (55.75,62) .. controls (55.75,58) and (59,54.75) .. (63,54.75) .. controls (67,54.75) and (70.25,58) .. (70.25,62) .. controls (70.25,66) and (67,69.25) .. (63,69.25) .. controls (59,69.25) and (55.75,66) .. (55.75,62) -- cycle ;
	\filldraw[color={rgb, 255:red, 0; green, 0; blue, 0 }  ,draw opacity=1] (183,61) circle (2.5pt);
	\draw    (303,12) -- (303,109) ;
	\draw    (684,61) -- (731,61) ;
	\draw    (560,109) -- (560,73) ;
	\draw    (560,52) -- (560,12) ;
	\draw    (639,12) -- (639,109) ;
	\draw    (475,12) -- (475,109) ;
	\draw    (400,61) -- (465,61) ;
	\draw    (488,61) -- (589,61) ;
	\draw    (432,12) -- (432,109) ;
	\draw    (522.5,12) -- (522.5,109) ;
	\draw   (447.95,56.12) -- (455.34,61.31) -- (447.98,66.55) ;
	\draw    (675,12) -- (675,109) ;
	\draw    (615,61) -- (670,61) ;
	\filldraw[color={rgb, 255:red, 0; green, 0; blue, 0 }  ,draw opacity=1] (589,61) circle (2.5pt);
	\filldraw[color={rgb, 255:red, 0; green, 0; blue, 0 }  ,draw opacity=1] (615,61) circle (2.5pt);
	\filldraw[color={rgb, 255:red, 0; green, 0; blue, 0 }  ,draw opacity=1] (545,61) circle (2.5pt);
	\draw (175,66) node [anchor=north west][inner sep=0.75pt]   [align=left] {$a$};
	\draw (540,66) node [anchor=north west][inner sep=0.75pt]   [align=left] {$a$};
	\draw (155,119) node [anchor=north west][inner sep=0.75pt]   [align=left] {(a)};
	\draw (543,119) node [anchor=north west][inner sep=0.75pt]   [align=left] {(b)};
\end{tikzpicture}
\caption{A base point $a$ just before an over-crossing.}
\label{Figure 22}
\end{center}
\end{figure}
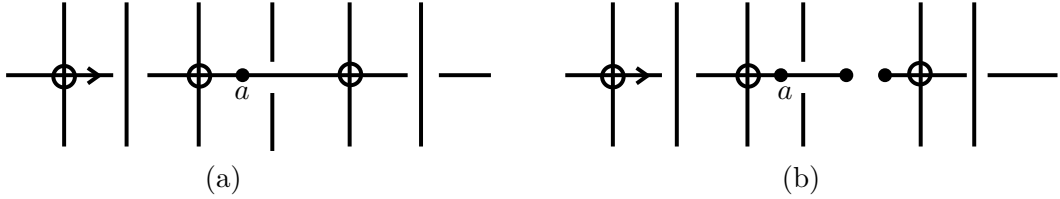

As noted in~\citep{satoh2018crossing,li2017unknotting}, any non-trivial welded knot $K$ satisfies $\operatorname{cr}(K) \geq 3$, where $\operatorname{cr}(K)= \min_{D} \{\operatorname{cr}(D) \mid \text{$D$ is a diagram} \text{ of $K$}\}$ is the classical crossing number of $K$. In the case of plus-welded  knotoid $K$, we have 

\begin{proposition} \label{Proposition 4.1}
Any non-trivial plus-welded knotoid $K$ satisfies $\operatorname{cr}(K)\geq3$. 
\end{proposition}

\begin{proof}
It is sufficient to show that if a plus-welded knotoid $K$ is satisfies $\operatorname{cr}(K) < 3$, then $K$ is trivial. 
	
\textbf{Case~1.} If $D$ is a diagram of $K$  with $\operatorname{cr}(D) = 0$, then  all crossings of $D$ are welded and $D$ can be transformed into a trivial knotoid diagram via $\Omega_v$-moves. Thus, $K$ is trivial.
	
\textbf{Case~2.} If $D$ is a diagram of $K$ with  $\operatorname{cr}(D) = 1$, that is $D$ is monotone. By~\cref{Corollary 4.1}, $D_a$ is a trivial diagram. Thus, $K$ is a trivial knotoid. 
	
\textbf{Case~3.} Let $D$ be a diagram of $K$ with $\operatorname{cr}(D) = 2$ and $a$ be a base point of $D$. 
	
\textbf{Subcase~3.1.} If $d(D_a)=0$, then $D_a$ is a trivial diagram by~\cref{Corollary 4.1}. Hence, $K$ is a trivial knotoid.
	
\textbf{Subcase~3.2.} If $d(D_a)=2$, then we obtain $d(-D_a)=0$ by~\cref{Lemma 4.1}. By the argument of Case~1, $K$ is a trivial knotoid. 
	
\textbf{Subcase~3.3.} If $d(D_a) = 1$, suppose that starting from $a$ we firstly arrive to an over-crossing, denoted by $q$, and secondly arrive to an under-crossing, denoted by $r$. 

Let us move the base point $a$ along the reverse orientation of $D$ passing the first classical crossing (denoted by $m$) and creating the base point $b$.
We notice that the path starting from the base point $a$ along the orientation of $D$ and returning to $a$, will meet the classical crossing $m$ either follows the over-arc of $r$ or the under-arc of~$q$.
\begin{figure}[htbp]
\begin{center}		
\tikzset{every picture/.style={line width=1.5pt}} 
\begin{tikzpicture}[x=0.75pt,y=0.75pt,yscale=-0.75,xscale=0.75]			
	\draw    (97,96.17) .. controls (68,21.17) and (199,-18.83) .. (236,71.17) ;
	\draw    (164,125.17) .. controls (232,168.17) and (260,109.17) .. (236,71.17) ;
	\draw    (102,116.17) .. controls (124,216.34) and (236,176.17) .. (207,105.17) ;
	\draw    (133,131.17) .. controls (103,67.17) and (169,40.17) .. (207,105.17) ;
	\draw    (121,183.17) .. controls (24,102.17) and (119,88.17) .. (151,119.17) ;
	\draw    (133,131.17) .. controls (173,188.17) and (145,1.17) .. (260,63.17) ;	
	\draw   (118.83,105.67) .. controls (118.83,101.53) and (122.19,98.17) .. (126.33,98.17) .. controls (130.47,98.17) and (133.83,101.53) .. (133.83,105.67) .. controls (133.83,109.81) and (130.47,113.17) .. (126.33,113.17) .. controls (122.19,113.17) and (118.83,109.81) .. (118.83,105.67) -- cycle ;
	\draw   (202.83,139.67) .. controls (202.83,135.53) and (206.19,132.17) .. (210.33,132.17) .. controls (214.47,132.17) and (217.83,135.53) .. (217.83,139.67) .. controls (217.83,143.81) and (214.47,147.17) .. (210.33,147.17) .. controls (206.19,147.17) and (202.83,143.81) .. (202.83,139.67) -- cycle ;
	\draw   (171.83,75.67) .. controls (171.83,71.53) and (175.19,68.17) .. (179.33,68.17) .. controls (183.47,68.17) and (186.83,71.53) .. (186.83,75.67) .. controls (186.83,79.81) and (183.47,83.17) .. (179.33,83.17) .. controls (175.19,83.17) and (171.83,79.81) .. (171.83,75.67) -- cycle ;
	\draw   (217.83,50.67) .. controls (217.83,46.53) and (221.19,43.17) .. (225.33,43.17) .. controls (229.47,43.17) and (232.83,46.53) .. (232.83,50.67) .. controls (232.83,54.81) and (229.47,58.17) .. (225.33,58.17) .. controls (221.19,58.17) and (217.83,54.81) .. (217.83,50.67) -- cycle ;
	\draw   (75.18,116.41) -- (85.83,110.49) -- (84.76,122.62) ;
	\draw [dashed]   (81,162.17) .. controls (149,274.17) and (334,141.17) .. (248,70.17) ;
	\draw [dashed]    (132,143.17) .. controls (186,192.17) and (180,22.17) .. (248,70.17) ;
	\draw     (136.02,154.16) -- (131.05,142.29) -- (143.58,145.2) ;	
	\draw    (460,109.17) .. controls (454,64.17) and (525,26.17) .. (529,106.17) ;
	\draw    (492,130.17) .. controls (553,163.17) and (640,142.17) .. (559,86.17) ;
	\draw    (484,147.17) .. controls (505,181.17) and (534,160.17) .. (529,106.17) ;
	\draw    (484,147.17) .. controls (471,104.17) and (503,54.17) .. (559,86.17) ;	
	\draw   (563.83,94.67) .. controls (563.83,90.53) and (567.19,87.17) .. (571.33,87.17) .. controls (575.47,87.17) and (578.83,90.53) .. (578.83,94.67) .. controls (578.83,98.81) and (575.47,102.17) .. (571.33,102.17) .. controls (567.19,102.17) and (563.83,98.81) .. (563.83,94.67) -- cycle ;
	\draw   (519.83,140.67) .. controls (519.83,136.53) and (523.19,133.17) .. (527.33,133.17) .. controls (531.47,133.17) and (534.83,136.53) .. (534.83,140.67) .. controls (534.83,144.81) and (531.47,148.17) .. (527.33,148.17) .. controls (523.19,148.17) and (519.83,144.81) .. (519.83,140.67) -- cycle ;
	\draw   (514.83,76.67) .. controls (514.83,72.53) and (518.19,69.17) .. (522.33,69.17) .. controls (526.47,69.17) and (529.83,72.53) .. (529.83,76.67) .. controls (529.83,80.81) and (526.47,84.17) .. (522.33,84.17) .. controls (518.19,84.17) and (514.83,80.81) .. (514.83,76.67) -- cycle ;
	\draw   (559.83,143.67) .. controls (559.83,139.53) and (563.19,136.17) .. (567.33,136.17) .. controls (571.47,136.17) and (574.83,139.53) .. (574.83,143.67) .. controls (574.83,147.81) and (571.47,151.17) .. (567.33,151.17) .. controls (563.19,151.17) and (559.83,147.81) .. (559.83,143.67) -- cycle ;
	\draw   (440.41,90.57) -- (440.43,102.75) -- (430.33,95.93) ;	
	\draw  [dashed]  (455,80.17) .. controls (433,242.17) and (537,215.17) .. (574,159.17) ;
	\draw  [dashed]  (426,76.17) .. controls (448,-13.83) and (650,8.17) .. (574,159.17) ;
	\draw   (447.81,87.83) -- (454.83,80.41) -- (458.55,89.93) ;
	\draw    (460,126.17) .. controls (490,264.17) and (623,137.17) .. (555,68.17) ;
	\draw    (492,41.17) .. controls (444,47.17) and (405,103.17) .. (472,121.17) ;	
	\filldraw[color={rgb, 255:red, 0; green, 0; blue, 0 }  ,draw opacity=1] (121,183.17) circle (2pt);
	\filldraw[color={rgb, 255:red, 0; green, 0; blue, 0 }  ,draw opacity=1] (260,63.17) circle (2pt);
	\filldraw[color={rgb, 255:red, 0; green, 0; blue, 0 }  ,draw opacity=1] (82,140) circle (2pt);
	\filldraw[color={rgb, 255:red, 0; green, 0; blue, 0 }  ,draw opacity=1] (130,128) circle (2pt);
	\filldraw[color={rgb, 255:red, 0; green, 0; blue, 0 }  ,draw opacity=1] (492,41.17) circle (2pt);
	\filldraw[color={rgb, 255:red, 0; green, 0; blue, 0 }  ,draw opacity=1] (555,68.17) circle (2pt);
	\filldraw[color={rgb, 255:red, 0; green, 0; blue, 0 }  ,draw opacity=1] (439,76) circle (2pt);
	\filldraw[color={rgb, 255:red, 0; green, 0; blue, 0 }  ,draw opacity=1] (467,80) circle (2pt);
	\draw (60,138) node [anchor=north west][inner sep=0.75pt]   [align=left] {$a$};
	\draw (77,81) node [anchor=north west][inner sep=0.75pt]   [align=left] {$q$};
	\draw (150,96) node [anchor=north west][inner sep=0.75pt]   [align=left] {$r$};
	\draw (115,123) node [anchor=north west][inner sep=0.75pt]   [align=left] {$b$};
	\draw (141,225) node [anchor=north west][inner sep=0.75pt]   [align=left] {(a)};
	\draw (92,205) node [anchor=north west][inner sep=0.75pt]   [align=left] {$D$};	
	\draw (417,76) node [anchor=north west][inner sep=0.75pt]   [align=left] {$a$};
	\draw (435,115) node [anchor=north west][inner sep=0.75pt]   [align=left] {$q$};
	\draw (488,106) node [anchor=north west][inner sep=0.75pt]   [align=left] {$r$};
	\draw (450,60) node [anchor=north west][inner sep=0.75pt]   [align=left] {$b$};
	\draw (490,226) node [anchor=north west][inner sep=0.75pt]   [align=left] {(b)};
	\draw (437,205) node [anchor=north west][inner sep=0.75pt]   [align=left] {$D$};					
\end{tikzpicture}
\caption{Moving the base point $a$ to the new base point $b$ along the reverse orientation of $D$.}		
\label{Figure 23}
\end{center}
\end{figure}
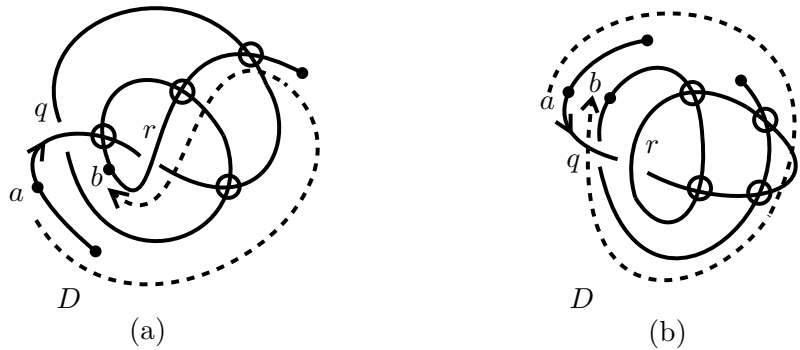
	
If $m = r$, then the path from $a$ to $b$ will go through the over-arc of $r$, as shown in~\cref{Figure 23} (a). In this case the path from $b$ along the orientation of $D$, will meet $r$ first time following the over-arc, and will meet $q$ first time also following  the over-arc. Thus $d(D_b) = 0$. By Case~1, $K$ is a trivial knotoid.
	
If $m = q$, then the path from $a$ to $b$ will go through the under-arc of $q$, as shown in~\cref{Figure 23} (b). In this case the path from $b$ along the orientation of $D$, will meet $q$ follows the under-arc, and will meet $r$ the first time  also following the under-arc. Thus $d(D_b) = 2$. By Case~2, $K$ is a trivial knotoid. 

Summarizing, if $\operatorname{cr}(D)<3$, then the plus-welded knotoid is trivial. Consequently, any non-trivial plus-welded knotoid $K$ satisfies $\operatorname{cr}(K)\geq3$. The proof of~\cref{Proposition 4.1} is completed. 
\end{proof}

\begin{theorem} \label{Theorem 4.1}
Let $D$ be a plus-welded knotoid diagram which has at least $3$ classical crossings, then we have
\begin{equation} \label{equation 4.3}
d(D) + d(-D) + 1 \leq \operatorname{cr}(D).
\end{equation}
Moreover, the equality holds if and only if $D$ is an alternating diagram.
\end{theorem}

\begin{proof}
 For any plus-welded  knotoid  diagram $D$ with $\operatorname{cr}(D) \geq 3$, the following inequality holds:
\begin{equation} \label{equation 4.4}
\max_b\{d(D_b)\} - \min_b\{d(D_b)\} \geq 1. 
\end{equation}
The equality in~\cref{equation 4.4} holds if and only if $D$ is an alternating diagram. Indeed, by~\cref{Lemma 4.3}, if $D$ is an alternating diagram, then the equality in~\cref{equation 4.4} holds.
On the other hand, if the equality in~\cref{equation 4.4} holds, $D$ is an alternating diagram, namely there are no two adjacent over-crossings or under-crossings when ignoring the welded crossings, as shown in~\cref{Figure 24}.
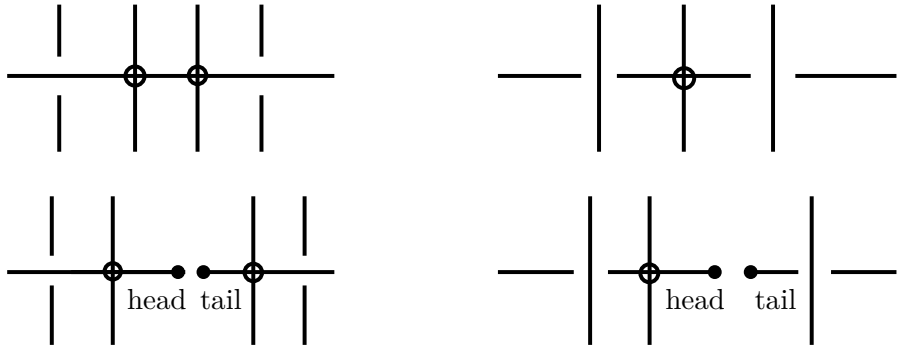
\begin{figure}[htbp]
\begin{center}
\tikzset{every picture/.style={line width=1.5pt}} 
\begin{tikzpicture}[x=0.75pt,y=0.75pt,yscale=-0.75,xscale=0.75]
	\draw    (580,59) -- (648,59) ;
	\draw    (505,11) -- (505,110) ;
	\draw    (565,11) -- (565,110) ;
	\draw    (448,11) -- (448,110) ;
	\draw    (460,59) -- (550,59) ;
	\draw    (380,59) -- (436,59) ;
	\draw    (604,191) -- (648,191) ;
	\draw    (250,240) -- (250,200) ;
	\draw    (250,180) -- (250,140) ;
	\draw    (482,140) -- (482,240) ;
	\draw   (475.81,191.81) .. controls (475.81,188.5) and (478.5,185.81) .. (481.81,185.81) .. controls (485.13,185.81) and (487.81,188.5) .. (487.81,191.81) .. controls (487.81,195.13) and (485.13,197.81) .. (481.81,197.81) .. controls (478.5,197.81) and (475.81,195.13) .. (475.81,191.81) -- cycle ;
 	\draw    (121,140) -- (121,240) ;
	\draw    (50,191) -- (164,191) ;
	\draw    (182,191) -- (270,191) ;
	\draw    (80,240) -- (80,200) ;
	\draw    (80,180) -- (80,140) ;
	\draw    (215.5,140) -- (215.5,240) ;
	\draw   (115.5,190.5) .. controls (115.5,187.46) and (117.96,185) .. (121,185) .. controls (124.04,185) and (126.5,187.46) .. (126.5,190.5) .. controls (126.5,193.54) and (124.04,196) .. (121,196) .. controls (117.96,196) and (115.5,193.54) .. (115.5,190.5) -- cycle ;
	\draw   (209.63,191.25) .. controls (209.63,188.01) and (212.26,185.38) .. (215.5,185.38) .. controls (218.74,185.38) and (221.38,188.01) .. (221.38,191.25) .. controls (221.38,194.49) and (218.74,197.13) .. (215.5,197.13) .. controls (212.26,197.13) and (209.63,194.49) .. (209.63,191.25) -- cycle ;
	\draw    (591,140) -- (591,240) ;
	\draw    (442,140) -- (442,240) ;
	\draw    (454,191) -- (530,191) ;
	\draw    (380,191) -- (431,191) ;
	\draw    (221,110) -- (221,72) ;
	\draw    (221,46) -- (221,11) ;
	\draw    (50,59) -- (270,59) ;
	\draw    (85,11) -- (85,46) ;
	\draw    (85,72) -- (85,110) ;
	\draw    (178,11) -- (178,110) ;
    \draw    (136,11) -- (136,110) ;
	\draw   (129.5,59.06) .. controls (129.5,55.51) and (132.38,52.63) .. (135.94,52.63) .. controls (139.49,52.63) and (142.38,55.51) .. (142.38,59.06) .. controls (142.38,62.62) and (139.49,65.5) .. (135.94,65.5) .. controls (132.38,65.5) and (129.5,62.62) .. (129.5,59.06) -- cycle ;
	\draw   (172.13,58.63) .. controls (172.13,55.38) and (174.76,52.75) .. (178,52.75) .. controls (181.24,52.75) and (183.88,55.38) .. (183.88,58.63) .. controls (183.88,61.87) and (181.24,64.5) .. (178,64.5) .. controls (174.76,64.5) and (172.13,61.87) .. (172.13,58.63) -- cycle ;
	\draw    (93,191) -- (169.75,191) ;
	\draw    (550,191) -- (582,191) ;
	\draw   (498.47,60.5) .. controls (498.47,56.77) and (501.49,53.75) .. (505.22,53.75) .. controls (508.95,53.75) and (511.97,56.77) .. (511.97,60.5) .. controls (511.97,64.23) and (508.95,67.25) .. (505.22,67.25) .. controls (501.49,67.25) and (498.47,64.23) .. (498.47,60.5) -- cycle ;
	\filldraw[color={rgb, 255:red, 0; green, 0; blue, 0 }  ,draw opacity=1] (164.88,191) circle (2.5pt);
	\filldraw[color={rgb, 255:red, 0; green, 0; blue, 0 }  ,draw opacity=1] (182,191) circle (2.5pt);
	\filldraw[color={rgb, 255:red, 0; green, 0; blue, 0 }  ,draw opacity=1] (525.81,191) circle (2.5pt);
	\filldraw[color={rgb, 255:red, 0; green, 0; blue, 0 }  ,draw opacity=1] (550,191) circle (2.5pt);
	\draw (128,200) node [anchor=north west][inner sep=0.75pt]   [align=left] {head};
	\draw (177,200) node [anchor=north west][inner sep=0.75pt]   [align=left] {tail};
	\draw (490,200) node [anchor=north west][inner sep=0.75pt]   [align=left] {head};
	\draw (550,200) node [anchor=north west][inner sep=0.75pt]   [align=left] {tail};
\end{tikzpicture}
\caption{Two adjacent over-crossings or under-crossings.}	
\label{Figure 24}
\end{center}
\end{figure}

Let $a$ and $a'$ be the base points which satisfy:
\begin{equation} \label{equation 4.5}
 d(D_a) = d(D), \quad d(-D_{a'}) = d(-D).
\end{equation}

By~\cref{Lemma 4.1}, we have that $p$ is a warping crossing of $-D_{a'}$ if and only if $p$ is a non-warping crossing of $D_{a'}$. Moreover, $d(-D_{a'})=d(-D)=\min\{d(-D_a)\mid a \text{ is a base point of}-D\}$. If the number of warping crossings of $-D_{a'}$ is minimized, then the number of non-warping crossings of $D_{a'}$ is minimized, and thus the number of warping crossings of $D_{a'}$ is maximized. Thus, we have
$$
   d(D_{a'})=\max_{a}\{d(D_a)\},  d(D_a)=\min_{a}\{d(D_a)\}.
$$
Thus, the~\cref{equation 4.4} becomes $d(D_{a'}) - d(D_a)\geq 1$, which further implies
\begin{equation} \label{equation 4.6}
d(D_{a'})\geq d(D_a)+ 1.
\end{equation}

By adding $d(-D_{a'})$ to the both sides of~\cref{equation 4.6}, we have
$$
d(D_a)+1 + d(-D_{a'})\leq d(D_{a'})+d(-D_{a'}).
$$
	
By~\cref{Lemma 4.1} and~\cref{equation 4.5}, we have \cref{equation 4.3}
and the equality in~\cref{equation 4.3} holds if and only if $D$ is an alternating diagram.
\end{proof}

By~\cref{Proposition 4.1} and~\cref{Theorem 4.1}, we have: 
 
\begin{corollary}
	\label{corollary 4.3}
Let $K$ be a non-trivial plus-welded knotoid and $D$ be any diagram of $K$, then 
\begin{equation} \label{equation 4.7}
		d(D) + d(-D) + 1 \leq \operatorname{cr}(D).
\end{equation}
Moreover, the equality holds if and only if $D$ is an alternating diagram. 
\end{corollary}

\section{Crossing changes for plus-welded knotoids} \label{section5}

In this section, we will introduce the unknotting number of plus-welded knotoids and describe its properties. 

\begin{proposition} \label{Proposition 5.1}
Any plus-welded knotoid diagram can be transformed into a trivial knotoid diagram by crossing changes.
\end{proposition} 

\begin{proof}
By a definition, a plus-welded knotoid allows $\Phi_+$ and $\Phi_{\text{over}}$-moves.
Observe that $\Phi_+$-move (resp. $\Phi_{\text{over}}$-move) become $\Phi_-$-move (resp.  $\Phi_{\text{under}}$-move) by crossing changes.
Hence, any plus-welded knotoid can be transformed into a trivial one by crossing changes.
\end{proof}

\begin{corollary}
Crossing change is an unknotting operation for plus-welded knotoids and their mirror images. 
\end{corollary}

\begin{definition} \label{Definition 5.1}
\rm{
For a plus-welded knotoid diagram $D$, we define the \textit{unknotting number} of $D$, denote $u(D)$, to be the minimal number of crossing changes needs to transform $D$ into a diagram of the trivial knotoid. For a plus-welded knotoid $K$, we define the \textit{unknotting number} of $K$, denote by $u(K)$, as the minimal number of $u(D)$ among all diagrams $D$ of $K$, namely 
$$
u(K) = \min_{D} \{ u(D) \mid \text{$D$ is a diagram} \text{ of $K$}\}.
$$
}
\end{definition}

According to the definition of warping degree and~\cref{Corollary 4.1}, we have 

\begin{theorem} \label{Theorem 5.1}
For any non-trivial plus-welded knotoid $K$, we have
$$
u(K) \leq d(K).
$$
\end{theorem}

\begin{proof}
(1) If $d(K) = 0$,
since 
$$
d(K) = \min_{D} \{ d(D), d(-D) \mid \text{$D$ is a diagram} \text{ of $K$}\},
$$
there exists a diagram $D$ of $K$ and a base point $a$ such that $d(D_a) = 0$. Since $D$ is a monotone diagram, by~\cref{Corollary 4.1}, $D$ can be transformed into a diagram of the trivial knotoid without crossing changes, hence $u(K) = 0$. Thus, $u(K) = 0 = d(K)$.

(2) If $d(K) \neq 0$, let $D$ be the diagram of $K$ which satisfies $d(K) = d(D)$ and $a$ be the base point on $D$ such that $d(D) = d(D_a)$. Via $d(D_a)$ crossing changes, $D_a$ can be transformed into $D'_a$. Since $D'_a$ is monotonic, by~\cref{Corollary 4.1}, $D_{a}$ is equivalent to a diagram of trivial knotoid. Thus, $D_a$ can be transformed into a diagram of trivial knotoid via $d( D_a)$ crossing changes. By the definitions of $d(K)$ and $u(K)$, we have 
 $$
 u(K) \leq u(D) \leq d( D_a)=  d(D)=d(K).
$$
Thus, $u(K) \leq d(K)$. This completes the proof of~\cref{Theorem 5.1}.
\end{proof}

From the above proof, we have

\begin{corollary} \label{Corollary 5.1}
A  plus-welded knotoid $K$ is trivial if and only if $d(K) = 0$.
\end{corollary}

\begin{corollary} \label{Corollary 5.2}
For any non-trivial  plus-welded knotoid $K$, we have
$$ 
u(K) \leq \dfrac{\operatorname{cr}(K) - 1}{2}, 
$$
\end{corollary}

\begin{proof}
 Let $D$ be the diagram of $K$ which satisfies $\operatorname{cr}(K) = \operatorname{cr}(D)$. By~\cref{Theorem 4.1}, we have
$$ 
d(D) + d(-D) + 1 \leq \operatorname{cr}(D) = \operatorname{cr}(K). 
$$
By~\cref{Theorem 5.1} and the definition of $d(K)$, we obtain
$$
u(K) \leq d(K) \leq \dfrac{d(D) + d(-D)}{2} \leq \dfrac{\operatorname{cr}(D) - 1}{2} =\dfrac{\operatorname{cr}(K) - 1}{2}.
$$
\end{proof}

Observe, that the upper bound in~\cref{Corollary 5.2} is the same as the upper bound for the case of non-trivial welded knots, see~\citep[Proposition~2.4]{satoh2018crossing}.  

\begin{corollary} \label{Corollary 5.3}
For any  plus-welded knotoid $K$, $u(K)=0$ if and only if $d(K)=0$.
\end{corollary}

\begin{proof}
It can be seen from the proof of~\cref{Theorem 5.1} that if $d(K) = 0$, we have $u(K) = 0$. 

On the other hand, if $u(K) = 0$, then there exists a diagram $D$ of $K$ such that $D$ can be transformed into a diagram $D'$ of trivial knotoid through a sequence of generalized $\Omega$-moves, $\Phi_{\text{over}}$-move and $\Phi_+$- move without crossing changes. In this case, $D'$ is a diagram of $K$ and then the warping degree of diagram $D'$ is $d(D') = 0$, we have $d(K) = 0$.
\end{proof}

Similarly, we can obtain 

\begin{corollary} \label{Corollary 5.4}
For any welded knot $K$, $u(K)=0$ if and only if $d(K)=0$.
\end{corollary}

In reference to the virtual closure of a plus-welded  knotoid, we have 

\begin{corollary}	\label{Corollary 5.5}
Let $K$ be a plus-welded knotoid and $K^{v}$ the welded knot obtained after the virtual closure of $K$, if $u(K)=0$, then we have $u(K^{v})=0$.
\end{corollary} 

\begin{proof}
If $u(K) = 0$, according to~\cref{Corollary 5.3}, we  have $d(K) = 0$. There exists a diagram $D$ of $K$ such that $D$ is a monotone diagram. Then the welded knot diagram $D^{v}$ obtained after the virtual closure of $D$ is also a monotone diagram following~\citep{li2017unknotting}. In this case, $D^{v}$ is a diagram of $K^{v}$ and then the warping degree of $K^{v}$ is $d(K^{v})=0$.
As stated in~\citep{li2017unknotting}, $u(K^{v}) \leq d(K^{v})$, furthermore, it follows that $u(K^{v}) \leq d(K^{v}) = 0$. In light of the properties of the unknotting number of welded knots, since $u(K^{v})\geq 0$, we consequently have $u(K^{v}) = 0$.
\end{proof}

\section{Crossing virtualizations for plus-welded  knotoids} \label{section6}

It has been pointed out in~\citep{Goussarov2000Finite} that crossing virtualization is more elementary than crossing change, where \textit{crossing virtualization}~\citep{Miyazawa2020Generalized} refers to a local move that transforms a classical crossing into a virtual crossing, as shown in~\cref{Figure 25}.
In this section, we introduced the virtualization unknotting number defined by crossing virtualization to plus-welded  knotoids analogously to~\citep{Gill2021UnknottingInvariant}. 
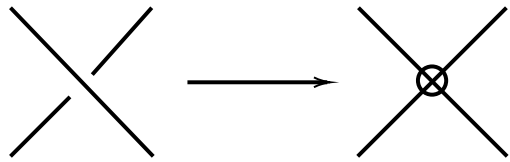
\begin{figure}[htbp]
\begin{center}
\tikzset{every picture/.style={line width=1.5pt}} 
\begin{tikzpicture}[x=0.75pt,y=0.75pt,yscale=-0.75,xscale=0.75]	
	\draw    (55,20) -- (151,120) ;
	\draw    (55,120) -- (95,80) ;
	\draw    (150,20) -- (110,65) ;
	\draw    (289,20) -- (389,120) ;
	\draw    (288.5,120) -- (388.5,20) ;
	\draw   (329,69) .. controls (329,63.75) and (333.25,59.5) .. (338.5,59.5) .. controls (343.75,59.5) and (348,63.75) .. (348,69) .. controls (348,74.25) and (343.75,78.5) .. (338.5,78.5) .. controls (333.25,78.5) and (329,74.25) .. (329,69) -- cycle ;	
	\draw    (174,70.17) -- (268,70.17) ;
	\draw [shift={(270,70.17)}, rotate = 180] [color={rgb, 255:red, 0; green, 0; blue, 0 }  ][line width=0.75]    (10.93,-3.29) .. controls (6.95,-1.4) and (3.31,-0.3) .. (0,0) .. controls (3.31,0.3) and (6.95,1.4) .. (10.93,3.29)   ;		
\end{tikzpicture}
\caption{Crossing virtualization.}	
\label{Figure 25}
\end{center}
\end{figure}

\begin{proposition} \label{Proposition 6.1}
Any plus-welded knotoid diagram can be transformed into a trivial knotoid diagram by crossing virtualizations.
\end{proposition}

\begin{proof}
Any plus-welded knotoid diagram can be transformed into the diagram where all crossings are welded by crossing virtualizations. Such a plus-welded knotoid diagram can be transformed into a trivial knotoid diagram by $\Omega_v$-moves.
\end{proof}

\begin{remark}
By~\cref{Proposition 6.1}, crossing virtualization is an unknotting operation for plus-welded knotoids and their mirror images. 
\end{remark}

\begin{definition} \label{Definition 6.1}
\rm{
For a plus-welded virtual knotoid diagram $D$, we define the \textit{virtualization unknotting number} of $D$, denote by $u_v(D)$, as the minimal number of crossing virtualizations to transform $D$ into a diagram of trivial knotoid. 
	
We define the \textit{virtualization unknotting number} of a plus-welded knotoid $K$ as the minimal number of $u_v(D)$ among all diagrams $D$ of $K$, denoted by $u_v(K)$, namely 
$$
u_v(K) = \min_{D} \{ u_v(D) \mid \text{D is a diagram} \text{ of } K\}.
$$
} 
\end{definition}


\begin{theorem} \label{Theorem 6.1}
For any non-trivial plus-welded knotoid $K$, we have
$$
u_v(K) \leq d(K).
$$
\end{theorem}

\begin{proof}
The proof of~\cref{Theorem 6.1} is similar to that of~\cref{Theorem 5.1}, we only need to replace crossing change with crossing virtualization. Then we present the specific proof below. 

(1) If $d(K) = 0$, since 
$$
d(K) = \min_{D} \{ d(D), d(-D) \mid \text{$D$ is a diagram} \text{ of $K$}\},
$$
there exists a diagram $D$ of $K$ and a base point $a$ such that $d(D_a) = 0$. $D$ is a monotone diagram, by~\cref{Corollary 4.1}, a monotone diagram can be transformed into a diagram of the trivial knotoid without crossing virtualizations, namely $u_w(K) = 0$.  Thus, $u_w(K) = 0 = d(K)$. 

(2) If $d(K) \neq 0$, let $D$ be the diagram of $K$ which satisfies $d(K) = d(D)$ and $a$ be the base point on $D$ such that $d(D) = d(D_a)$. Via $d(D_a)$ crossing virtualizations, $D_a$ can be transformed into $D'_a$. By~\cref{Corollary 4.1}, $D'_a$ is equivalent to a diagram of the trivial knotoid. Thus, $D_a$ can be transformed into a diagram of the trivial knotoid via $d( D_a)$ crossing virtualizations. By the definitions of $d(K)$ and $u(K)$, we have 
	$$
    u_v(K) \leq u_v(D) \leq d( D_a)=  d(D)=d(K).
	$$
Thus, $u_v(K) \leq d(K)$. This completes the proof of~\cref{Theorem 6.1}.
\end{proof} 

By~\cref{Theorem 6.1}, we have the following~\cref{Corollary 6.1,Corollary 6.2,Corollary 6.3,Corollary 6.4}  and the proofs of these corollaries are similar to the proofs of~\cref{Corollary 5.2,Corollary 5.3,Corollary 5.4,Corollary 5.5}. 

\begin{corollary} \label{Corollary 6.1}
For any non-trivial plus-welded knotoid $K$, we have
$$
u_v(K) \leq \dfrac{\operatorname{cr}(K) - 1}{2}.
$$
\end{corollary}

\begin{corollary} \label{Corollary 6.2}
For any plus-welded knotoid $K$, $u_v(K)=0$ if and only if $d(K)=0$.
\end{corollary}
 
\begin{definition} \citep{Gill2021UnknottingInvariant}
Let $D$ be a diagram of welded knot.
The welded unknotting number $u_w(D)$ is the minimum number of crossing virtualizations to transform $D$ into a diagram of trivial knot. 

The welded unknotting number $u_w(K)$ of a welded knot $K$ is defined as the minimal number of $u_w(D)$ among all diagrams $D$ of $K$, namely
$$
u_w(K)=\min_{D} \{ u_w(D) \mid \text{D is a diagram} \text{ of } K\}.
$$
\end{definition}

\begin{corollary} \label{Corollary 6.3}
For any welded knot $K$, $u_w(K)=0$ if and only if $d(K)=0$.
\end{corollary}

\begin{corollary} \label{Corollary 6.4}
Let $K$ be a plus-welded  knotoid and $K^{v}$ the welded knot obtained after the virtual closure of $K$, if $u_v(K)=0$, then we have $u_w(K^{v})=0$.
\end{corollary}

\normalem 
\bibliographystyle{plain}
\bibliography{references} 

\end{document}